\documentclass{amsart}
\usepackage{amssymb,amsmath}

\usepackage{graphicx}

\usepackage{tikz}
\usepackage{tikz-cd}
\usepackage[pdftex]{hyperref}
\usetikzlibrary{matrix,calc}
\usepackage{ifthen}

\newcommand{\bP}{{\mathbb P}}

\newcommand{\bR}{{\mathbb R}}
\newcommand{\bZ}{{\mathbb Z}}
\newcommand{\bN}{{\mathbb N}}

\newcommand{\bF}{{\mathbb F}}
\newcommand{\bY}{{\mathbb Y}}
\newcommand{\bX}{{\mathbb X}}
\newcommand{\bW}{{\mathbb W}}
\newcommand{\bM}{{\mathbb M}}
\newcommand{\bO}{{\mathbb O}}
\newcommand{\bfk}{{\mathbf k}}
\newcommand{\bfm}{{\mathbf m}}
\newcommand{\bfn}{{\mathbf n}}

\newcommand{\cF}{\mathcal F}

\newcommand{\cK}{\mathcal K}
\newcommand{\cA}{\mathcal A}

\newcommand{\cC}{\mathcal C}
\newcommand{\cD}{\mathcal D}

\newcommand{\cM}{\mathcal M}

\newcommand{\cP}{\mathcal P}
\newcommand{\cQ}{\mathcal Q}

\newcommand{\cV}{\mathcal V}

\newcommand{\cX}{\mathcal X}

\newcommand{\fs}{\mathfrak{s}}
\newcommand{\fg}{\mathfrak{g}}
\newcommand{\ro}{\mathrm{o}}
\newcommand{\gl}{\mathrm{gl}}

\newcommand{\flow}{\operatorname{flow}}
\newcommand{\Flow}{\operatorname{Flow}}

\newcommand{\Face}{\operatorname{Face}}
\newcommand{\Iso}{\operatorname{Biject}}

\newcommand{\coker}{\operatorname{coker}}
\newcommand{\Ch}{\operatorname{Ch}}

\newcommand{\Hom}{\operatorname{Hom}}

\newcommand{\Chart}{\operatorname{Chart}}
\newcommand{\Kur}{\operatorname{Kur}}
\newcommand{\Top}{\operatorname{Top}}
\newcommand{\Set}{\operatorname{Set}}
\newcommand{\Cat}{\operatorname{Cat}}

\newcommand{\Ob}{\operatorname{Ob}}

\newcommand{\id}{\operatorname{id}}
\newcommand{\pt}{\operatorname{pt}}

\newcommand{\codim}{\operatorname{codim}}
\newcommand{\colim}{\operatorname*{colim}}
\newcommand{\hocolim}{\operatorname*{hocolim}}

\newcommand{\rel}{\operatorname{rel}}

\newcommand{\Multimod}{\operatorname{Multimod}}

\newcommand{\Int}{\operatorname{Int}}
\def\co{\colon\thinspace}

\newtheorem{thm}{Theorem}[section]
\newtheorem{cor}[thm]{Corollary}
\newtheorem{lem}[thm]{Lemma}

\newtheorem{prop}[thm]{Proposition}
\newtheorem{defin}[thm]{Definition}
\newtheorem{def-lem}[thm]{Definition-Lemma}

\theoremstyle{remark}

\newtheorem{rem}[thm]{Remark}

\renewcommand{\th}[0]{\textsuperscript{th}}

\numberwithin{equation}{section}

\makeatletter
\newcommand\RSloop{\@ifnextchar\bgroup\RSloopa\RSloopb}
\makeatother
\newcommand\RSloopa[1]{\bgroup\RSloop#1\relax\egroup\RSloop}
\newcommand\RSloopb[1]%
  {\ifx\relax#1%
   \else
     \ifcsname RS:#1\endcsname
       \csname RS:#1\endcsname
     \else
       \GenericError{(RS)}{RS Error: operator #1 undefined}{}{}%
     \fi
   \expandafter\RSloop
   \fi
  }
\newcommand\X{0}
\newcommand\RS[1]%
  {\begin{tikzpicture}
     [every node/.style=
       {circle,draw,fill,minimum size=1.5pt,inner sep=0pt,outer sep=0pt},
      line cap=round
     ]
   \coordinate(\X) at (0,0);
   \RSloop{#1}\relax
   \end{tikzpicture}
  }
\newcommand\RSdef[1]{\expandafter\def\csname RS:#1\endcsname}
\newlength\RSu
\RSdef{i}{\draw (\X) -- +(90:\RSu) node{};}
\RSdef{l}{\draw (\X) -- +(135:\RSu) node{};}
\RSdef{r}{\draw (\X) -- +(45:\RSu) node{};}
\RSdef{I}{\draw (\X) -- +(90:\RSu) coordinate(\X I);\edef\X{\X I}}
\RSdef{L}{\draw (\X) -- +(135:\RSu) coordinate(\X L);\edef\X{\X L}}
\RSdef{R}{\draw (\X) -- +(45:\RSu) coordinate(\X R);\edef\X{\X R}}

\title{An axiomatic approach to virtual chains}
\author{Mohammed Abouzaid}

\begin{document}
\maketitle
\setcounter{tocdepth}{1}

\begin{abstract}
  We introduce a category of Kuranishi presentations, whose objects are a variant of the Kuranishi structures introduced by Fukaya and Ono, and which can be seen as a refinement of the version studied by Pardon. We then formulate the notion of virtual chains categorically as a natural transformation between two functors from this category to the category of chain complexes; we call such a datum \emph{a theory of virtual counts}. To show that this definition carries non-trivial content, we then construct a multicategory whose objects are Kuranishi flow categories, and show that a theory of virtual counts determines a multifunctor to the multicategory of chain complexes. We then implement this construction in the setting of Hamiltonian Floer theory, borrowing from some joint work with Groman and Varolgunes, yielding a construction of Hamiltonian Floer groups (and operations on them) as an output of this machine.  We plan to provide a similar account for Lagrangian Floer theory in subsequent joint work.
\end{abstract}

\tableofcontents
\section{Introduction}
\label{sec:introduction}

The proof that the differential in Morse and Floer squares to $0$ is a consequence, in the simplest situations, of expressing the boundary of a $1$-dimensional moduli space as a fibre product of $0$-dimensional manifolds. Algebraically, this amounts to writing the boundary of a fundamental chain for the $1$-dimensional manifold as a sum of products of the fundamental chains for the $0$-dimensional manifolds that define the differential. This algebraic reformulation has the advantage that it can make sense in contexts in which Floer theory does not produce manifolds with boundary; this is the notion of virtual fundamental chains appearing in the title. More generally, we have in mind the extensions of this type of argument which are required to construct chain-level structures in Floer theory, e.g. the $A_\infty$-structures which are of central importance in the study of Lagrangian submanifolds.

In the literature, the notion of a virtual fundamental chain is tied to a specific chain-level model for (co)-homology: Fukaya and Ono \cite{FukayaOno1999} used singular chains, as did Liu and Tian \cite{LiuTian1998}, while in their work on toric varieties Fukaya, Oh, Ohta, and Ono used differential forms \cite{FOOO2016}, which is also the approach taken in the polyfold framework of Hofer, Wysovski, and Zehnder \cite{HoferWysockiZehnder2009}. On the other hand, Pardon uses \v{C}ech cohomology \cite{Pardon2016}, as do McDuff and Wehrheim in \cite{McDuff2017}. In all these instances, it is difficult to disentangle the notion of a virtual chain (in the chosen theory) from its use in Floer theory, as most of the work goes in producing virtual fundamental chains in the chosen theory from the geometry of the moduli spaces of pseudo-holomorphic curves.

The purpose of this paper is provide an axiomatic formulation of the notion of virtual chains, and to verify that one of the currently available approaches (namely Pardon's \cite{Pardon2016}) can be modified to satisfy these axioms. The reader will have to decide whether this modification should count as a simplification (the papers are largely logically independent), but we would like to stress at this stage that many of the techniques that we use have an antecedent in \cite{Pardon2016}; those familiar with this formalism will likely find our use the homotopy normal bundles to be the most novel ingredient (c.f. Section \ref{sec:thoms-isom-via}). 

\begin{rem}
Our axiomatic formulation is designed for geometric situations in which the moduli spaces have been rigidified, i.e. where one counts elements of virtual dimension $0$, with relations arising from those of virtual dimension $1$. In such a setting, the key difficulty of understanding the relationship between virtual chains associated to a pair of moduli spaces and the virtual chains associated to a fibre product over a manifold does not arise because fibre products over $0$-dimensional manifolds can be alternatively described as disjoint unions of products of subsets. In practice, this means that the approach to Floer theory which we will describe relies on Morse theory to rigidify moduli spaces. This approach has its own drawbacks, especially in geometric setting where there is a non-trivial symmetry group.
  
\end{rem}

\subsection{The category of Kuranishi presentations}
\label{sec:categ-kuran-pres-1}

The starting point of our formulation of the notion of virtual chains is the category of Kuranishi charts, given in Definition \ref{def:Kuranishi-charts}. The objects of this category are a generalisation of Fukaya and Ono's original finite-dimensional approach to the construction of virtual chains; they consist of a compact Lie group $G$ of automorphisms, of a locally linear $G$-manifold $X$ with finite stabilisers, of a $G$-vector bundle $V$ over $X$, and of a $G$-equivariant section $\fs$ of $V$. We shall denote by $Z$ the zero locus of $\fs$, and the basic idea is that the quotient $Z/G$ should form a local model for the moduli space being studied.
\begin{rem}
The fact that we extend the formalism to allow for arbitrary compact Lie groups (rather than only finite ones), is largely motivated by our joint work with McLean and Smith \cite{AbouzaidMcleanSmith2021}, which realises moduli spaces of stable genus $0$ maps, in compact symplectic manifolds, as such a quotient $Z/G$ for a single Kuranishi chart, which we call a \emph{global chart}.  
\end{rem}

The next thing to note about our notion of Kuranishi charts is that, following Pardon \cite{Pardon2016}, we shall not impose any smoothness assumption on the manifold $X$. In applications, this greatly simplifies the analytic work that is required in order to prove that moduli spaces of pseudo-holomorphic curves fit within the framework, though this has the obvious disadvantage of preventing the use of differential forms as a model for the cochains. The next thing to note is that our Kuranishi charts are equipped with an additional datum, namely a global stratification indexed by a category. This is in some sense a generalisation for the assumption, in the smooth setting, that the total space is a smooth manifold with corners, but, even in that context, the additional datum of the stratification would be important for an axiomatic formulation of a theory of virtual chains.
\begin{rem}
  The fact that we consider stratifications by categories rather than by partially ordered sets is a consequence of our desire to ultimately apply the results of this paper to Lagrangian Floer theory, and the simpler setup would have sufficed for the Hamiltonian theory that we provide as a test of the applicability of our results.
\end{rem}

From a practical point of view, the key new point, following \cite{AbouzaidBlumberg2021}, is that we make explicit the fact that we construct a \emph{category} of Kuranishi charts denoted  $\Chart$, which then allows us to build later structures using formal arguments. In other approaches, this categorical structure is either absent (often because the very definition of composition requires an additional choice), or is present but left implicit (e.g. in \cite{Ishikawa2018}), so this is not in itself any kind of major advance, but it forms the cornerstone of our approach to the subject: all constructions should be done functorially at the level of the category of Kuranishi charts, which allows for one to apply to them arguments and constructions of formal nature. We note here that Joyce \cite{Joyce2019} has done substantial work, in the smooth setting, building a more sophisticated formalism of Kuranishi charts (which he calls Kuranishi neighbourhoods), which form a $2$-category in his context, but we shall not use the higher categorical structure.

One difference to note with \cite{AbouzaidBlumberg2021} at this stage is that we take care to construct $\Chart$ as a symmetric monoidal category, i.e. that we define the product of Kuranishi charts, and ensure that this product does not depend on the order up to natural isomorphism (this is essentially obvious as long as one takes care that the notion of stratification be sufficiently flexible). This additional structure is crucial when formulating a theory of virtual chains which is adequate for applications to the study of Fukaya categories.

In any finite-dimensional approach to virtual fundamental chains, Kuranishi charts are assembled as the building blocks of global structures variously called Kuranishi structures, Kuranishi spaces, or Implicit atlases depending on the specific implementation. To distinguish our approach from others, we introduced the terminology of \emph{Kuranishi presentation} in  \cite{AbouzaidBlumberg2021} to refer to a category $\Kur$ whose objects are diagrams in the category of Kuranishi charts (i.e. a functor from an indexing category of $\Chart$) satisfying certain assumptions. We revisit this notion in Definition \ref{def:Kuranishi-presentation}, again making sure to construct a symmetric monoidal structure on this category.

At this stage, the datum of the stratification becomes important, as our definition of a Kuranishi presentation really consists of compatible diagrams in the categories of stratified charts associated to each substratum of the given space. Thinking of Kuranishi spaces as generalisations of manifolds, this is certainly an oddity as we would be requiring that an atlas of a manifold with boundary includes the datum of charts for the boundary. What the reader should have in mind, however, is that all our constructions proceed by globalising functors defined on charts, so that if we want an invariant associated to a manifold to functorially map to the associated invariant associated to a bounding manifold, we must either have canonical ways of associating to charts of the boundary a chart of the interior or we must proceed as we do by incorporating these charts in the atlas of the bounding manifold.

We note at this stage that the category $\Kur$ that we construct has far fewer morphisms than the corresponding category studied by Joyce \cite{Joyce2019}. We are singularly focused on the notion of virtual counts as it is used in Floer theory, and for this notion, the only morphisms that are required are the (inclusions) of boundary strata. The key reason that more general morphisms are required in Joyce's approach is the need to ultimately consider fibre products of Kuranishi spaces with respect to evaluation maps to manifolds. As discussed earlier, our own planned approach to applications eschews such fibre products by using Morse theory to rigidify moduli spaces.

\subsection{Theories of virtual counts}
\label{sec:theor-virt-counts-1}

The main notion that we introduce in this paper will be formulated in terms of a natural transformation between two functors from the category of Kuranishi presentations to the category of cochain complexes over a ring $\Bbbk$ (more precisely, the domain is a category $ \Kur^{\Bbbk}$ of graded and $\Bbbk$-oriented Kuranishi charts discussed in Section \ref{sec:absol-orient-kuran}). The target of this natural transformation is the constant functor with value the ring $\Bbbk$ considered as a chain complex supported in degree $0$ with trivial differential. The domain of this natural transformation is the cochain functor of strata $C^*_{st}$ introduced in Section \ref{sec:chain-functor-strata}: this is a functor which assigns to a Kuranishi presentation a complex freely generated by the objects of the stratifying category whose associated strata are non-empty. More precisely, the complex is the direct sum of orientations lines associated to the strata, and the differential is given by a restriction map on these orientation lines.

To understand the idea behind the stratum cochain complex, we consider an oriented manifold $\cM$ with corners: the boundary of each stratum of codimension $k$ is a union of codimension $k-1$ strata, and each such stratum has a fundamental class in relative homology, whose image under the boundary homomorphism is the sum of the relative fundamental classes of its boundary strata. We can thus form a chain complex generated by the relative fundamental classes of all strata, with differential given by the boundary homomorphism. Since we are using cohomological conventions, this complex is supported in degrees greater than minus the virtual dimension, with a unique generator up to sign in this degree corresponding to the top stratum.

At this stage, it would be tempting to define a theory of virtual counts to be a symmetric monoidal natural transformation (we review this notion in Appendix \ref{sec:symm-mono-categ-1}) from the stratum functor to the constant functor:
  \begin{equation} \label{eq:virtual_counts}
\begin{tikzcd}[column sep=huge]
\Kur^{\Bbbk}
  \arrow[bend left=50]{r}[name=U,label=above:$C^*_{st}$]{}
  \arrow[bend right=50]{r}[name=D,label=below:$\Bbbk$]{} &
\Ch.
\arrow[from=U.south-|D,to=D,Rightarrow,shorten=5pt]{r}{\cV}
\end{tikzcd}   
   \end{equation}
   We shall presently explain why we are unable to prove that such a monoidal transformation exists, but it may be useful to first explain why such a natural transformation should be called a theory of virtual counts: note first of all that, since $\Bbbk$ is supported in degree $0$, the above natural transformation consists of a map
   \begin{equation}
          C^{0}_{st}(\bX) \to \Bbbk
   \end{equation}
   for each Kuranishi presentation.  As discussed above, $ C^{k}_{st}(\bX)$ is generated by the strata of dimension $-k$, so the above map is then an assignment of an element of $\Bbbk$ for each stratum of virtual dimension $0$. More precisely, the $\Bbbk$-module $ C^{0}_{st}(\bX) $ is generated by the image of the fundamental classes of the Kuranishi presentations associated to these strata, so that naturality implies that all the data is encoded by assigning a number to each Kuranishi presentations of virtual dimension $0$. This is the \emph{virtual count} that is required to construct algebraic structures on Kuranishi presentations.

       This assignment is not arbitrary. The understand the conditions that are imposed, we use again functoriality, now applied to Kuranishi presentations of virtual dimension $1$. The stratum functor for such presentation has a unique generator in degree $-1$ (corresponding to the (relative) fundamental chain), generators in degree $0$ corresponding to codimension $1$ strata (i.e. those of virtual dimension $0$), as well as generators in positive degree corresponding to strata of negative virtual dimension. The differential maps the generator in degree $-1$ to the sum of the generators in degree $0$, but the image of the degree $-1$ generator under the natural transformation must vanish (because the target is trivial is this degree), so we conclude that the virtual counts of boundaries of $1$-dimensional Kuranishi presentations vanishes.

       We pause at this stage to point out that this data is exactly what is needed to prove the invariance of various enumerative constructions: for example, if one imagines that the $0$-dimensional Kuranishi presentations are moduli spaces of stable holomorphic spheres for a given almost complex structure, and the $1$-dimensional Kuranishi presentation is a parametrised moduli space for a family of tame almost complex structures, then we would be recovering the statement that Gromov-Witten invariants do not depends on the choice of tame almost complex structure.

       To continue the discussion, we observe that, in the construction of Floer complexes, the above structure is not sufficient, and the proof that the Floer differential squares to $0$ arises from the fact that each codimension-$1$ boundary of a  moduli space of Floer trajectories of virtual dimension $n$ is a product of moduli spaces the sum of whose virtual dimension is $n-1$. Specialising to $n=1$, we get a decomposition of the codimension $1$ strata as a product of moduli spaces the sum of whose virtual dimension vanishes. The equation $d^2=0$ then arises by expressing the virtual counts of these boundary strata as a product of the virtual counts of the factors; in particular, this implies that the only contributions arise from the case in which both factors have vanishing virtual dimension.  In our formalism, the fact that the virtual count of a product is the product of the virtual counts of its factors would follow immediately from equipping the domain and target of $\cV$ with the structure of monoidal functors, and imposing the condition that $\cV$ be a monoidal natural transformation.

We now reach the point where we explain why we do not formulate our results in terms of Diagram \eqref{eq:virtual_counts}: the basic issue is that all constructions that we are aware of involve choices, and it is not clear how to make such choices independent of all possible automorphisms of a Kuranishi presentations. This is particularly difficult when one desires to have a monoidal natural transformation, since a $0$-dimension Kuranishi presentation may admit inequivalent factorisations as products of $0$-dimensional presentations, and there seems to be no way of ensuring that the virtual counts associated to different factorisations agree.
       
       We expect that one possible solution to this problem is to formulate the notion of virtual count in an $\infty$-categorical setting, but give a more concrete solution: in Section \ref{sec:categ-kuran-pres}, we shall introduce a rigidification $\widetilde{\Kur}^{\Bbbk}$ of the category of oriented Kuranishi presentations, which has the property that product decompositions are unique up to reordering.

       \begin{rem}
        While it is difficult to justify introducing $\widetilde{\Kur}^{\Bbbk}$ from an abstract point of view, it use does not hinder the study of Floer theory, because all moduli spaces that arise in concrete applications are equipped with canonical product decompositions, hence lift to this category.
       \end{rem}

       For our main definition, we shall write $C^*_{st}$ for the composition of the stratum functor with the forgetful map from $\widetilde{\Kur}^{\Bbbk}$ to $\Kur^{\Bbbk}$. 
   
\begin{defin} \label{def:virtual_count}
  A \emph{multiplicative theory of virtual counts with coefficients in $\Bbbk$} is a monoidal natural transformation from $C^*_{st}$ to the constant functor $\Bbbk$, considered as symmetric monoidal functors from  $\widetilde{\Kur}^{\Bbbk}$ to $\Ch_{\Bbbk}$:
   \begin{equation} \label{eq:virtual_counts-strict}
\begin{tikzcd}[column sep=huge]
\widetilde{\Kur}^{\Bbbk}
  \arrow[bend left=50]{r}[name=U,label=above:$C^*_{st}$]{}
  \arrow[bend right=50]{r}[name=D,label=below:$\Bbbk$]{} &
\Ch.
\arrow[from=U.south-|D,to=D,Rightarrow,shorten=5pt]{r}{\cV}
\end{tikzcd}   
   \end{equation}
\end{defin}
\begin{rem} \label{rem:theor-virt-counts}
  For the purpose of extracting numbers from moduli spaces of holomorphic curve, one can altogether drop the monoidal requirement, and any natural transformation from $ C^*_{st}$ to the constant functor $\Bbbk$ is sufficient; this can be done via a much simpler construction than the one we require. In order to construct chain complexes from moduli spaces of holomorphic curves, it is sufficient to work with monoidal functors; this would be done by replacing $\widetilde{\Kur}^{\Bbbk}$ with a coarser monoidal category in which factorisations are unique, and the task of constructing such a natural transformation is also substantially simpler than constructing a symmetric monoidal one. The symmetric condition is required when studying multiplicative structures, such as Fukaya categories.
\end{rem}

Observing that 
the cohomology $ H^{*}_{st}(\bX)$ is canonically trivial whenever $\bX$ is stratified by a singleton, we have:
\begin{defin}
  If $\Bbbk$ is a field of characteristic $0$, we say that a theory of virtual counts satisfies the \emph{standard normalisation} if on any Kuranishi presentation of $\pt/\Gamma$ (with empty boundary strata), the image of the canonical generator of $H^0_{st}(\bX)$ is $1/|\Gamma|$.
\end{defin}
\begin{rem}
  We conjecture the existence of theories which do not satisfy the standard normalisation, but none have yet been exhibited in the literature.  We do not currently expect that all theories which satisfy the standard nomalisation are equivalent: to understand the kind of issue that may arise, consider a finite group $\Gamma$, a $\Gamma$-representation $V$, and the vector bundle $V \times X \to X$ over a compact manifold $X$  with trivial $\Gamma$-action of dimension equal to that of $V$. Choosing $X$ appropriately, these data should represent a non-trivial Kuranishi presentation of virtual dimension $0$, but we expect that one can choose what a theory of virtual counts assigns to this Kuranishi presentation independently of what is assigned to $\pt/\Gamma$.
\end{rem}

The main result of this paper is the following:

\begin{thm} \label{thm:virtual_counts_exist}
If $\Bbbk$ is a field of characteristic $0$, there exists a symmetric theory of virtual counts which satisfies the standard normalisation.
\end{thm}
\begin{rem}
 If one restricts attention to the smooth setting and to usual manifolds with corners, we expect that one can formulate a notion of equivalence of multiplicative theories, in terms of homotopies of natural transformations, so that the three methods for constructing virtual counts (via singular cochains \cite{FukayaOhOhtaOno2009}, de Rham cochains \cite{FOOO2016,Hofer2006}, or \v{C}ech cochains \cite{Pardon2016,McDuff2017}), all lead to equivalent theories of virtual counts satisfying the standard normalisation. We reiterate, however, that the formalism developed in this note does not incorporate the notion of virtual chains on fibre products, and hence one cannot recast the known applications of virtual fundamental chains to symplectic topology (in particular those due to Fukaya, Oh, Ohta, and Ono), in our framework, without doing the additional work of rigidifying moduli spaces.
\end{rem}

At this stage, it may be helpful to provide a brief outline of the proof of Theorem \ref{thm:virtual_counts_exist}. There are essentially two steps:
\begin{enumerate}
\item Construct a morphism from the stratum functor to the constant functor, in the homotopy category of symmetric monoidal functors from $\Kur^{\Bbbk}$ to $\Ch$.
\item Show that the pullback of the stratum functor to $\widetilde{\Kur}^{\Bbbk} $ satisfies a lifting property with respect to monoidal weak equivalence of symmetric monoidal functors.
\end{enumerate}
Concretely, the first step entails the construction of zig-zag
\begin{equation}
 C^*_{st} \Leftarrow F_1 \Rightarrow \cdots \Leftarrow F_k \Rightarrow \Bbbk   
\end{equation}
consisting of a finite collection of symmetric monoidal functors, and of monoidal natural transformations between them, in which the left pointing arrows are equivalences. Such a zig-zag is sufficient for an $\infty$-categorical approach to algebraic structures in Floer theory, but a more standard approach requires replacing the zig-zag of natural transformations appearing in our definition by a single monoidal natural transformation. This is the outcome of the second step, which entails proving that we can construct the dashed arrow in a diagram
  \begin{equation}
    \begin{tikzcd}
      & F_1 \ar[d,Rightarrow] \\
      C^*_{st} \ar[ur,Rightarrow,dashed] \ar[r,Rightarrow] & F_2,
    \end{tikzcd}
  \end{equation}
  consisting of symmetric monoidal functors on $\widetilde{\Kur}^{\Bbbk} $  and natural transformations between them. We warn the reader at this stage that, when constructing a lift, we only ensure its commutativity on the cohomology level, since this is sufficient for our purpose.

  We also prove two variants of Theorem \ref{thm:virtual_counts_exist}. Both concern settings in which one can construct a theory of virtual counts over an arbitrary ring $\Bbbk$. The first is when one restricts to the category of Kuranishi charts with trivial isotropy, and the proof in this case is a minor variant of Theorem \ref{thm:virtual_counts_exist}, so that we do not separate the discussion of this case. The second arises when one consider Kuranishi charts which have non-negative dimension, and which have the property that those charts of dimension $0$ and $1$ have trivial isotropy. This requires a slightly different approach, whereby one notices that Floer-theoretic applications only require information about moduli spaces of dimension $0$ and $1$, so that one can truncate the virtual cochains to only keep information in this degree. We discuss this in Section \ref{sec:monot-kuran-pres}.

  \subsection{The category of flow categories}
\label{sec:categ-flow-categ}

In the second part of this paper, we explain how to build algebraic structures from the datum of a theory of virtual counts in the sense of Definition \ref{def:virtual_count}, keeping in mind that we shall apply this framework to Hamiltonian Floer theory in the third part; this implementation is built upon the work developed with Groman and Varolgunes in \cite{AbouzaidGromanVarolgunes2021}.

Inspired as in \cite{Pardon2016} by Cohen-Jones-Segal's notion of a flow category \cite{CohenJonesSegal1995}, we formulate the notion of a Kuranishi flow category, which should be thought of as the primary output of Morse and Floer theory, in the sense that we should have objects corresponding to critical points of an action functional, and morphisms given by moduli spaces of solutions to Floer's equation. The reader will again find that our formulation is different from \cite{Pardon2016}, but this difference is familiar to practitioners of Floer theory: in non-exact situations, one can describe generators of Floer complexes either as critical points of a $1$-form on a space of loops or paths, in which case, as in Novikov-Morse theory, there are usually finitely many generators, with a possibly non-compact space of trajectories connecting them, or one can pass to a cover on which the action functional is defined, leading to a possibly infinite set of generators, with a compact space of generators. These two perspectives are equivalent once one records the action of the deck transformations of the covering space in the second case (and the relative homotopy class of the trajectories in the first case), and account for the differences between the two approaches.
 
We choose, however, to precede giving the definition of Kuranishi flow categories by introducing a formal enlargement of the Kuranishi presentations, which we call the category of \emph{Kuranishi diagrams.} Such structure appear repeatedly when trying to use Morse theory to rigidify Floer theory (c.f. \cite{Fukaya1993,BiranCornea2009}): the moduli spaces that have to be counted in order to define operations are given as a union of top-dimensional strata, indexed by topological configurations of Riemann surfaces and metric intervals, which are glued pairwise along codimension $1$ boundary strata. Our notion of Kuranishi diagram formalises this idea by including the data of a stratifying category, and an assignment of a Kuranishi presentation to this category. We also include a condition for the stratifying category, that essentially amounts to requiring that the link of each codimension $k$ stratum be a sphere of dimension $k$, as in the definition of a combinatorial manifold.
\begin{rem}
We could significantly weaken the condition that we impose on Kuranishi diagrams, requiring only that codimension $1$ strata are matched, as in the definition of a pseudo-manifold, but we do not know a situation where this weaker condition naturally occurs.
\end{rem}

In the next step, we try to construct a category $\Flow$ whose objects are Kuranishi flow categories, and whose morphisms are bimodules; we are motivated here by the notion of continuation maps in Morse (and Floer) theory which provide a way to define maps associated to a $1$-parameter family of functions whose endpoints are Morse. This is the place where the notion of Kuranishi diagram is essential in our approach, as the composition of two bimodules built from Kuranishi presentations fails to remain within this class. In the geometric context, this amounts to the fact that a composition of continuation maps is a \emph{broken continuation map.}

However, this basic idea runs into a pesky difficulty which arises from the failure of the natural isomorphism between products of sets to be equalities. The issue can be seen already for endomorphisms of the empty flow category: such endomorphisms are given by Kuranishi diagrams without boundary, and composition is given by products thereof. But this composition fails to be associative for the simple reason that the product sets $X \times (Y \times Z)$ and $(X \times Y) \times Z$ are not equal. We resolve the problem by incorporating factorisations into compositions as
part of the datum of morphisms in $\Flow$.

\begin{rem}
  It is tempting to dismiss this issue as pure pedantry, and we suspect that it is possible to resolve the problem at once by appealing to Isbell's rectification result \cite{Isbell1969} that replaces every monoidal category with an equivalent category in which the associator isomorphisms are equalities. However, the fact that one cannot in general identify the products $X \times Y$ and $Y \times X$ is witnessed both by Steenrod operations, and by Moore's result that connected commutative monoids are products of Eilenberg-MacLane spaces.  
\end{rem}

Implementing the idea of including the choice of factorisation in the datum of a morphism is made technically more difficult by needing to combine this with constructing $\Flow$ as a category whose morphisms carry some remnant of a topology: the latter structure is required because any serious application of Floer/Morse theory, e.g. to prove the invariance of the homology with respect to the choice of Morse function, relies on the study of \emph{parametrised moduli spaces.} However, while the topology on unbroken continuation maps is straightforward to describe, trying to define an actual topology that incorporates breaking (and that is adequate for applications) is a matter that we prefer to avoid. Instead, we model the underlying homotopy theory using the theory of symmetric cubical sets \cite{Grandis2009,Isaacson2011}; which ultimately amounts to studying only families parametrised by cubes. 
\begin{rem}
We expect that our approach here is related to the notion introduced in \cite{Pardon2016} of a flow category parametrised by a simplicial set, under the usual equivalence between simplicial and cubical sets as models for spaces. In particular, it should not be too difficult to prove that the nerve of the category $\Flow$ is a universal base for a flow category diagram in the sense of  \cite{Pardon2016}.  
\end{rem}
At this stage, we can state a key consequence of the combined results of the second two parts of the paper, which rely on the fact that the category of chain complexes may be considered as cubically enriched category via
\begin{equation}
  \Hom_{n}(C,D) \equiv \Hom(C \otimes C_*([0,1]^n), D).  
\end{equation}
\begin{thm}
  \label{thm:floer-functor-and-Ham}
  A theory of virtual counts determines a cubically enriched functor
\begin{equation}
    \Flow \to \Ch,
\end{equation}
which can be applied to flow categories of Hamiltonian Floer trajectories and to continuation maps between them.
\end{thm}
In fact, the results of Section \ref{sec:flow-mult-hamilt} lead to a sharper statement, from which we conclude the invariance of Floer homology in Corollary \ref{cor:Hamiltonian-CF-well-defined}.
\begin{rem}
  We could have replaced the third part of the paper with an account of Lagrangian Floer theory in the exact or monotone situation, but we plan to return to this subject, in a generality that incorporates curvature, in subsequent work.
\end{rem}
  \subsection{The multicategory of flow categories}
\label{sec:mult-flow-categ-2}

We take an additional step towards to the construction of algebraic operations in Floer theory by equipping $\Flow$ with the structure of a (symmetric) multicategory (see Appendix \ref{sec:algebr-constr-symm}): extending the idea that a morphism between flow categories is given by a bimodule, we assign to a sequence of input flow categories $\vec{\bX} = (\bX_1, \ldots, \bX_n)$, and an output flow category $\bY$ a space (more precisely, a symmetric cubical set)
\begin{equation}
  \Flow(\bX_1, \ldots, \bX_n; \bY)  
\end{equation}
of multimodules, which consists of assignments of Kuranishi diagrams to each sequence of objects of $\vec{\bX}$ and of $\bY$, with an appropriate action the morphisms of $\vec{\bX}$ and $\bY$. In Hamiltonian Floer theory, we obtain such multimodules by considering Riemann surfaces with $n+1$ punctures, of which $n$ are labelled by the Hamiltonians giving rise to the flow categories $\bX_i$, and the last is labelled by the Hamiltonian corresponding to $\bY$.

\begin{rem}
  One way to think about the structure of multimorphisms is that it is a replacement for the product of flow categories, in the sense that if we were to introduce a monoidal structure in this category, it would be required to come with a natural isomorphism
  \begin{equation}
       \Flow(\bX_1 \otimes \cdots \otimes \bX_n, \bY)   \cong   \Flow(\bX_1, \ldots, \bX_n; \bY)
     \end{equation}
     for all possible targets $\bY$. The subtlety of defining this product can be seen when thinking about the example of Morse theory: the compactified moduli spaces of Morse flow lines for a function on $M_1 \times M_2$ given as a sum of functions on the individual factors can be expressed in terms of the Morse flow lines in the two factors, but the expression is not given simply by taking the product, because the moduli space of flow lines is a quotient by $\bR$ of the space of gradient trajectories, which affects the description of the compactification.
\end{rem}

\begin{thm}
  \label{thm:floer-multfunctor-and-Ham}
  A theory of virtual counts determines an extension of the functor
\begin{equation}
    \Flow \to \Ch,
\end{equation}
to a multifunctor, where $\Ch$ is considered as a multicategory via it usual symmetric monoidal structure:
\begin{equation}
  \Hom_{\bullet}(C_1, \ldots, C_n; D) \equiv \Hom_{\bullet} (C_1 \otimes \cdots \otimes C_n, D).
\end{equation}
\end{thm}

In Section \ref{sec:flow-mult-hamilt}, we apply this result to Hamiltonian Floer theory.

\subsection*{Acknowledgments} This paper began partly from the following sources: (i) a joint project with John Pardon, now largely abandoned, to use the formalism of \cite{Pardon2016} in order give an alternative construction of Fukaya categories in a general context, (ii) the work with Yoel Groman and Umut Varolgunes on local Floer theory \cite{AbouzaidGromanVarolgunes2021}, (iii) an attempt to put in writing concrete thoughts about Kuranishi spaces coming from joint work with Andrew Blumberg on Floer homotopy theory \cite{AbouzaidBlumberg2021}, in a context that does not require the full machinery of stable homotopy theory, and (iv) the desire to unify the traditional perspective which describes moduli spaces of pseudo-holomorphic curves in terms of local Kuranishi charts with joint work \cite{AbouzaidMcleanSmith2021} performed with Ivan Smith and Mark McLean that provides geometric constructions of global Kuranishi charts for moduli spaces of stable genus $0$ holomorphic curves. I am grateful to them for discussions about various points. I would also like to thank Nate Bottman and Semon Rezchikov for comments about early drafts.

This research was supported by the NSF through grants DMS-1609148, DMS-1564172, and DMS-2103805, the Simons Collaboration on Homological Mirror Symmetry, and a Poincar\'e visiting professorship at Stanford University.

\part{Constructing a theory of virtual counts}
\label{part:constr-theory-virt}
 
\section{Kuranishi presentations}
\label{sec:symm-mono-categ}

\subsection{Stratified spaces}
\label{sec:stratified-spaces}

Let $\cQ$ be a small category, equipped with a functor to the natural numbers with their usual ordering
\begin{equation}
  \codim \co \cQ \to \bN = \{ 0 \to 1 \to \cdots \} 
\end{equation}
which we call the \emph{codimension}.

Given an arrow $\alpha$ in  $\cQ$ (with domain $P_0$ and target $P_1$), we write $\cQ^{\alpha}$ for the category whose objects are factorisation $P_0 \to P \to P_1$ of $\alpha$, and whose morphisms are arrows from $P$ to $P'$ which fit in a commutative diagram
\begin{equation}
  \begin{tikzcd}[row sep = 10, column sep = 10]
    & P \ar[dd] \ar[dr] & \\
    P_0 \ar[ur] \ar[dr] & & P_1. \\
    & P' \ar[ur] &
  \end{tikzcd}
\end{equation}
Note that $ \cQ^{\alpha}$ has an initial object arising from the identity of $P_0$ and a terminal object associated to the identity on $P_1$. We abuse notation and write $P_0$ and $P_1$ for these objects.
\begin{defin} \label{def:model_for_corners}
  A category $\cQ$ is a \emph{model for manifolds with generalised corners} if  (i) $\cQ$ has an initial element and, (ii) for each arrow $\alpha$, the category $\cQ^{\alpha} \setminus \{ P_0, P_1\}$ is a partially ordered set  whose geometric realisation
  is a sphere of dimension
  \begin{equation}
    \codim P_1 - \codim P_0 - 2.
  \end{equation}
  A map of models is an embedding $\cQ_1 \to \cQ_2$ inducing an isomorphism $ \cQ^{\alpha}_1 \to \cQ^{f(\alpha)}_2$ for each arrow $\alpha$ in $\cQ_1$.
\end{defin}
\begin{rem}
  In Section \ref{sec:extend-categ-kuran}, we shall consider a more general setup where the condition that $\cQ$ have an initial element is dropped, allowing us to have more than one top dimensional stratum. 
\end{rem}

The geometric realisation of a category is the simplicial complex whose non-degenerate $n$-simplices are composable sequences consisting of $n$ morphisms. Our definition thus implies that there are no morphisms in $\cQ$ (other than identities) between objects of equal codimension. Since adding a terminal or initial element to a category corresponds to taking the cone, we see that the triple
\begin{equation}
  ( |\cQ^{\alpha}|, | \cQ^{\alpha} \setminus \{P_0\}|, | \cQ^{\alpha} \setminus \{P_1\}|)
\end{equation}
is homeomorphic to $(D^{n}, D_+^{n-1}, D_-^{n-1})$. We shall use the notation
\begin{equation}
  \kappa^{\alpha} \equiv |\cQ^{\alpha}|,
\end{equation}
and in the special case where $\alpha_P$ is the unique arrow from the initial object to a given object $P$, we write
\begin{equation}
    \kappa^{P} \equiv \kappa^{\alpha_P}.
\end{equation}
\begin{rem}
  The canonical example to have in mind in the case $\cQ$ is the partially ordered set of subsets of a set $Q$: in that case, there is a unique arrow $P_0 \to P_1$ if and only if $P_1$ includes $P_0$, and $ |\cQ^{P_0 \to P_1}|$ is a cube of dimension given by the number of element of the complement. This level of generality is sufficient to understand the open-closed field theory associated to Lagrangians in a symplectic manifold, in the absence of curvature. An intermediate notion, which appears when studying Lagrangian Floer theory in the presence of curvature, would consist of requiring that $\cQ^{\alpha} $ be a cube of the appropriate dimension. The setting we consider arises in the construction the Fukaya $2$-category of symplectic manifolds \cite{MauWehrheimWoodward2018,Bottman2019}, as well as in the study of symplectic field theory \cite{Pardon2019}. 
\end{rem}

For the next definition, we shall need to appeal to the following construction:  given an arrow $\alpha \co P_0 \to P_1$, we associate to each object $P_0 \to P \to P_1$ of $\cQ^{\alpha}$ the partially ordered set $\cQ^{P \to P_1}$. This construction is functorial in $\cQ^{\alpha} $ because an arrow  from $P \to P'$ induces a map of partially ordered set
\begin{equation}
 \cQ^{P \to P_1} \to\cQ^{P' \to P_1}. 
\end{equation}

\begin{defin}
  A \emph{$ \cQ $-manifold} with an action of a compact Lie group $G$ is a contravariant functor $P \mapsto \partial^P M$  on $ \cQ^{op}$ taking value in locally linear topological $G$-manifolds with boundary, and satisfying the following two properties for each $P_0 \in \cQ$:
  \begin{enumerate}
  \item the natural map
    \begin{equation}
      \colim_{\alpha \co P_0 \to P_1}   \partial^{P_1} M \to \partial^{P_0} M   
    \end{equation}
    factors through a homeomorphism of the colimit with the boundary of $ \partial^{P_0} M$.
    \item given an arrow $\alpha \co P_0 \to P_1$, and a point $x$ in the interior of $\partial^{P_1} M$, there is a $G$-manifold $U$, and a natural transformation
     \begin{equation} \label{eq:local_model}
              U \times |\cQ^{P \to P_1} | \Rightarrow \partial^{P} M
            \end{equation}
of functors on $\cQ^\alpha$, which is $G$-equivariant for the trivial $G$-action on second factor of the source, and whose restriction to the product of $U$ with a neighbourhood of $P_1 \in |\cQ^{P \to P_1} |$ is an open neighbourhood of the image of $x$ in $\partial^{P} M$.   
  \end{enumerate}

\end{defin}
We note that Equation \ref{eq:local_model} specialises to the statement that $U$ is a $G$-invariant open neighbourhood of $x$ in $\partial^{P_1} M $, because $P_1 \in |\cQ^{P \to P_1} |$ is a boundary point. 
\begin{rem}
The local linearity condition is the requirement that each orbit admit a $G$-invariant neighbourhood which is $G$-equivariantly homeomorphic to the total space of a vector bundle on the orbit. This condition is automatically satisfied in the smooth setting. In its absence, there is essentially no control on the topology of the fixed point loci of $G$ or of the quotient space \cite{Bing1952}.
\end{rem}

We often abuse notation, and write $M$ for the manifold associated to the minimal element of $\cQ$; indeed, we can describe the data of a $\cQ$-manifold as a stratification on $M$ in the manner of \cite[Definition 2.15]{Pardon2019}. However, our point of view allows us to functorially associate to each element $P \in \cQ$ the $ \partial^{P} \cQ $-manifold $\partial^P M$, where $\partial^{P} \cQ$ is the under category of $P$ (i.e. the category of arrows in $\cQ$ with source $P$).

When we specialise the local model in Equation \eqref{eq:local_model} to the case that $P$ and $P_0$ are both given by the minimal element, we find a local homeomorphism
\begin{equation} \label{eq:local_model_stratum_open}
    \partial^{P_1} M  \times \kappa^{P_1} \dashrightarrow M.
\end{equation}

\begin{rem}
The above definition is a version of Pardon's notion of cell-like stratification \cite[Section 3.1]{Pardon2019}. 
It includes  as a special case the stratification induced on the topological manifold with boundary underlying  \emph{smooth manifolds with generalised corners}  \cite{Joyce2016,KottkeMelrose2015}; in this setting, the stratification is modeled after an integral affine polytope. 
\end{rem}

\subsection{Kuranishi charts}
\label{sec:kuranishi-charts}

We now arrive at the basic building block of all models of Kuranishi spaces; our definition is slightly more general than the one that is usually considered, in that we allow $G$ to be a compact Lie group:
\begin{defin} \label{def:Kuranishi-charts}
 A \emph{Kuranishi chart} is a quintuple $\bX = (G,\cQ,X,V,s)$, where:
  \begin{enumerate}
  \item $G$ is a compact Lie group,
  \item $\cQ$ is a model for manifolds with generalised corners,
  \item $X$ is a $ \cQ $-manifold with an action of $G$ that has finite stabilisers,
  \item $V$ is a real $G$-vector bundle over $X$, and
    \item $\fs$ is a $G$-equivariant section of $V$.
  \end{enumerate}
\end{defin}
The datum of the vector bundle can be expressed most simply as a vector bundle on the space $X$ associated to the minimal element of $\cQ$; this determines a vector bundle over $\partial^P X$, for every $P \in \cQ$, by pullback.

Kuranishi charts are the objects of a category which we denote $\Chart$, and whose morphisms are given by:
\begin{defin}
A morphism $f$ of  Kuranishi charts  from $\bX_1 = (G_1,\cQ_1,X_1,V_1,\fs_1)$ to $\bX_2 = (G_2,\cQ_2,X_2,V_2,\fs_2)$ consists of the following data:
  \begin{enumerate}
  \item A homomorphism $G_1 \to G_2$ whose kernel acts freely on $X_1$,
  \item A map of $f \co \cQ_1 \to \cQ_2$ of models in the sense of Definition \ref{def:model_for_corners}. We denote the image of the minimal element of $\cQ_1$ by $P_f$.
  \item A $G_1$-equivariant map $f \co X_1 \to \partial^{P_f} X_2$ (i.e. a natural transformation from $X_1$ to the pullback of $\partial^{P_f} X_2$).
  \item A $G_1$-equivariant embedding $V_1 \to f^* V_2$ of vector bundles over $X_1$.
  \end{enumerate}
  We require that the following diagram commute
  \begin{equation}
    \begin{tikzcd}
        V_1 \ar[r] & V_2\\
        X_1 \ar[r] \ar[u,"\fs_1"] &  \partial^{P_f} X_2 \ar[u,"\fs_2"],
    \end{tikzcd}
  \end{equation}
  and that it be transverse.
\end{defin}

Since we are in the setting of topological manifolds, we need to elaborate on our notion of transversality: letting
  \begin{equation}
 \pi_f \co  V^{\perp}_{f} \to X_1      
  \end{equation}
  denote the quotient by $V_1$ of the pullback of $V_2$ to $X_1$, we require the existence, for each point in $X_1$, of a $G_1$-invariant neighbourhood $U_f^\perp$ in $V_f^\perp$ (using the embedding of $V_1$ as the zero-section), and of a $G_2$-equivariant  open embedding
  \begin{equation} \label{eq:quotient_has_open_embedding}
         U^{\perp}_f \times_{G_1} G_2 \to \partial^{P_f} X_2,
       \end{equation}
respecting the stratification, so that the following diagram commutes:
    \begin{equation} \label{eq:diagram_charts_Kuranishi_map}
    \begin{tikzcd}
      X_1 \ar[r] &  \partial^{P_f} X_2 \ar[r,"\fs_2"] & V_2 \\
     U_f^\perp \cap X_1 \ar[r,hookrightarrow] \ar[u,hookrightarrow] &  U^{\perp}_f \ar[r,"\fs_1"] \ar[u]  & \pi_f^* V_1\ar[u].
    \end{tikzcd}
  \end{equation}

To further clarify this notion, note that the datum of the map in Equation \eqref{eq:quotient_has_open_embedding} is equivalent to a $G_1$-equivariant map
\begin{equation}
  U^{\perp}_f \to \partial^{P_f} X_2;
\end{equation}
we formulate it in terms of the fibre product in order to state the condition of being an open embedding. Next, the left square of Diagram \eqref{eq:diagram_charts_Kuranishi_map} asserts that the above map extends the map on $X_1$ that is part of the data of the morphism of charts. Finally, the right square Diagram \eqref{eq:diagram_charts_Kuranishi_map} is where the transversality constraint is imposed, for it implies that, if we split $f^* V_2$ at the direct sum of $V_1$ and $V_f^\perp$, then the pullback of $\fs_2$ agrees with the direct sum of $\fs_1$ with the identity section of $V_f^\perp$.  

The product of Kuranishi charts is given by
\begin{multline}
  (G_1,\cQ_1,X_1,V_1,\fs_1) \times  (G_2,\cQ_2,X_2,V_2,\fs_2) \\ \equiv  (G_1 \times G_2,\cQ_1 \times \cQ_2,X_1 \times X_2,V_1 \times V_2,\fs_1 \times \fs_2).
\end{multline}
It is straightforward to check that there is a well-defined induced map of products, which is associative and commutative, i.e. we have:
\begin{lem}
 The category of Kuranishi charts is symmetric monoidal. \qed 
\end{lem}

 We denote the subcategory of charts stratified by $\cQ$ by $\Chart_{\cQ} $, and  by $\Chart_{\downarrow \cQ}$ the category whose objects are pairs $(\cQ' \to \cQ, \bX')$, consisting of a chart stratified by a category $\cQ'$ equipped with a map of models $\cQ' \to \cQ$. Morphisms from $(\cQ'' \to \cQ, \bX'')$ to $(\cQ' \to \cQ, \bX')$ are given by a map $\bX'' \to \bX$ of Kuranishi charts, whose underlying functor $\cQ'' \to \cQ' $  fits in a commutative diagram
\begin{equation}
  \begin{tikzcd}
    \cQ'' \ar[r] \ar[dr] & \cQ' \ar[d] \\
    & \cQ.
  \end{tikzcd}
\end{equation}
Note that we are imposing equality of compositions here, rather than the existence of a natural transformation.

\subsection{Kuranishi presentations}
\label{sec:kuran-pres}

To introduce the notion of a Kuranishi presentation, it is useful to note that the category of Kuranishi charts is equipped with a forgetful functor
\begin{equation} \label{eq:footprint_functor}
\cM \co \Chart \to \Top,  
\end{equation}
which assigns to a Kuranishi chart the quotient of the zero-locus $\fs^{-1}(0)$ by $G$ (this is the footprint functor in the terminology of McDuff and Wehrheim \cite{McDuff2017}). 

\begin{defin} \label{def:Kuranishi-presentation}
  A \emph{Kuranishi presentation} consists of a category $\cQ$ (which is a model for manifolds with generalised corners), a category $A$, and a functor
  \begin{equation}
    \bX \co A \to \Chart_{\downarrow \cQ}
  \end{equation}
  such that the (ordinary) colimit $\cM$ of the stratified spaces $\cM(\bX_\alpha) $ over $\alpha$ in $A$ is a compact Hausdorff space with finitely many non-empty strata, and such that  \begin{equation} \label{eq:contractible_nerve_presentation}
    \parbox{31em}{for each point $u \in \cM$, the nerve of the category $A_{[u]}$ of charts which cover $u$ is contractible.}
  \end{equation}
\end{defin}
\begin{rem}
While Condition \eqref{eq:contractible_nerve_presentation} seems technical, some version of it is necessary for our approach: we shall construct various invariants of Kuranishi presentations as (homotopy) colimits over the indexing category $A$, and if we do not impose this condition, then the answers we would obtain would include spurious contributions. 
\end{rem}
In applications, one often encounters a stratified space $\cM$, which is \emph{equipped with a Kuranishi presentation} in the sense that we a functor $\bX$ as in Definition \ref{def:Kuranishi-presentation}, together with a homeomorphism
\begin{equation}
  \colim_{\alpha \in A}   \cM(\bX_\alpha) \cong \cM.
\end{equation}
In this case, we say that $\bX$ is a Kuranishi presentation of $\cM$, omitting the choice of homeomorphism from the notation.

It is convenient to introduce the notation 
\begin{equation} \label{eq:boundary_of_Kuranishi_presentation}
  \partial^{P} \bX \co \partial^{P} A \to \Chart_{\downarrow \partial^P \cQ}
\end{equation}
for the Kuranishi presentation associated to $\bX$ and to each element $P \in \cQ$, where $\partial^{P} A$ is the (categorical) fibre over $P$ of the functor $A \to \cQ$, which assigns to each chart $\alpha$ the image of the minimal element of the category $\cQ_\alpha$ in $\cQ$. Explicitly, an object of $\partial^{P} A$ consists of an arrow $P \to P'$ in $\cQ$, and an object $\alpha$ of $A$ over $P'$. The functor $\partial^{P} \bX $ assigns to such a pair the chart $\bX_{\alpha}$.

\begin{defin} \label{def:map_Kuranishi-presentation}
  A \emph{ map of Kuranishi presentations} from $(\cQ_1, A_1, \bX_1)$ to $(\cQ_2, A_2, \bX_2)$  consists of a map of models $\cQ_1 \to \cQ_2$, a functor $A_1 \to A_2$, and a natural transformation in the following diagram
\begin{equation}  \label{eq:commutative_diagram_map_of_presentations}
  \begin{tikzcd}
    A_1 \ar[r] \ar[d]  & \Chart_{\downarrow \cQ_1} \ar[d,""{name=fromhere}] \\
    A_2 \ar[r,""{name=tohere}] &  \Chart_{\downarrow \cQ_2}.
    \arrow[Rightarrow,from=fromhere, to=tohere, start anchor={east},
    end anchor={south}, shorten >=3pt, shorten <=6pt,bend right=30]
  \end{tikzcd}
\end{equation}
 If $P \in \cQ_2$ is the image of the minimal element of $\cQ_1$, we require the above map to induce a homeomorphism from $\cM_{1}$ to $\partial^{P} \cM_2$. 
\end{defin}
\begin{rem}
  If $P \in \cQ_2$ is the image of the minimal element of $\cQ_1$, then a morphism $\bX_1 \to \bX_2$ canonically factors through $\partial^{P} \bX_2$.
\end{rem}
To be more explicit, and omitting any notation for the map $A_1 \to A_2$, we recall that the datum of a natural transformation is a map of Kuranishi charts from  $\bX_{1,\alpha_1}$ to $ \bX_{2,\alpha_1}$ so that the following diagram commutes for each arrow $\alpha_1 \to \alpha'_1$ in $A_1$:
\begin{equation} \label{eq:natural_isomorphism_commutative_square}
  \begin{tikzcd}
    \bX_{1,\alpha_1} \ar[r] \ar[d] &   \bX_{2,\alpha_1} \ar[d] \\
    \bX_{1,\alpha'_1} \ar[r]  &   \bX_{2,\alpha'_1}.
  \end{tikzcd}
\end{equation}

We define the composition of maps of Kuranishi presentations in a straightforward way, noting that the composite natural transformation is given by horizontally stacking Diagram \eqref{eq:natural_isomorphism_commutative_square}. We denote the corresponding category by $\Kur$, and shall often abuse notation by denoting its objects $\bX$, leaving the indexing category $A$ and the stratifying category $\cQ$ implicit in the notation. 
\begin{prop}
  The category $\Kur$ is symmetric monoidal, with monoidal product given by
  \begin{equation}
    \bX_1 \times \bX_2  \co A_1 \times A_2 \to \Chart_{\downarrow \left(\cQ_1 \times \cQ_2\right)}.
  \end{equation}
  \qed
\end{prop}

\begin{rem}
 If we drop the natural isomorphism from Definition \ref{def:map_Kuranishi-presentation}, in the sense that we require that it agree with the identity, the resulting category fails to be monoidal, since the associator isomorphisms $(\bX \times \bY) \times \bZ \cong \bX \times (\bY \times \bZ)$ are absent.
\end{rem}

\subsection{Relatively oriented Kuranishi presentations}
\label{sec:absol-orient-kuran}

To pass from geometry to algebra, one often needs to prescribe (i) the dimension of Kuranishi presentations, as well as (ii) a choice of orientation. To this end, we fix a commutative ring $\Bbbk$, and an integer $d \in \bN$ which is required to be even unless $\Bbbk$ has characteristic $2$ (otherwise, we will not obtain a symmetric monoidal category).

To proceed with the definition, we need two basic notions: first, the (virtual) dimension of a chart $\bX$ refers to the integer
\begin{equation}
\dim \bX \equiv  \dim X - \dim V - \dim G.
\end{equation}
Next, we formulate the notion of relative orientation as follows: the groups $H_{\dim X}(X,X \setminus \{x\})$ form a local system over $X$, as do the groups $H_{\dim V}(V_x, V_{x} \setminus \{0\})$, where $V_x$ is the fibre of $V$ at $x \in X$. We obtain a local system of free abelian groups
\begin{equation} \label{eq:orientation_line_X}
  H_{\dim X}(X,X \setminus \{x\}) \otimes (H_{\dim V}(V_x, V_{x} \setminus \{0\}) \otimes H_{\dim G}(\fg, \fg \setminus \{ 0 \})^{-1}, 
\end{equation}
which is graded in degree $-\dim \bX$ (the negative sign is because we shall use cohomological conventions for homological algebra). An orientation of $\bX$ relative to a graded line $\ro_{\bX}$ is then an isomorphism from this local system to $\ro_{\bX}$.
\begin{defin}
  The category $\Chart^{\Bbbk,d}$ of \emph{ $\bZ/d$-graded and relatively $\Bbbk$-oriented Kuranishi charts} is the category whose objects are pairs $(\bX, \ro_{\bX})$, where $\ro_{\bX}$ is a $\bZ/d$-graded line, and  $\bX$ is a  Kuranishi chart such that all components of the underlying manifold have the same dimension modulo $d$, and which is equipped with a $\Bbbk$-orientation relative to $\ro_{\bX}$.

  A morphism in $\Chart^{\Bbbk, d}$ from $(\bX_0, \ro_{\bX_0}) $ to $(\bX_1, \ro_{\bX_1}) $ consists of a morphism $f \co \bX_0 \to \bX_1$ of the underlying charts and an isomorphism
  \begin{equation}
      \ro_{\bX_1} \cong   \ro_{\bX_0} \otimes \ro_{P_f},
  \end{equation}
where $\ro_{P_f}$ is the orientation line of the ball $\kappa^{P_f} $  associated to $f$,  with the property that Equation \eqref{eq:quotient_has_open_embedding} preserves relative orientations (using the fact that an orientation of $\kappa^{P_f}$ and a relative orientation of $X$ induces a relative orientation of $\partial^{P_f} X$ via Equation \eqref{eq:local_model_stratum_open}).
\end{defin}    
The most important cases are the category $\Chart^{\bF_2, 1}$, which is isomorphic to $\Chart$ since every manifold is oriented with respect to $\bF_2$ coefficients, and the category $\Chart^{\bZ,0}$ of oriented charts of pure dimension. We shall henceforth leave the parameter $d$ hidden from the notation. Moreover, since the graded line $\ro_{\bX}$ is isomorphic to the local system in Equation \eqref{eq:orientation_line_X}, which in an intrinsic invariant of $\bX$, we often omit it from the notation.

Note that, for an arrow $f \co \bX_1 \to \bX_2$, we have
\begin{equation}
  \dim \bX_1 + \dim \kappa^{P_f} = \dim \bX_2 \in  \bZ/ d \cdot \bZ,
\end{equation}
where $P_f$ is, as before, the image in $\cQ_2$ of the minimal element in $\cQ_1$. We define  $\Chart^{\Bbbk}_{\downarrow \cQ}$ to be the category of oriented charts whose stratifying category is equipped with a map of models  $f \co \cQ' \to \cQ$, and an orientation of the product of $\bX$ with $\kappa^{P_f}$. Given a pair $(\cQ' \to \cQ, \bX)$ in this category, we define
\begin{equation}
  \dim (\cQ' \to \cQ, \bX) \equiv \dim \bX + \dim \kappa^{P_f}  \in  \bZ/ d \cdot \bZ.
\end{equation}
We have an evident forgetful functor
\begin{equation}
   \Chart^{\Bbbk}_{\downarrow \cQ} \to  \Chart_{\downarrow \cQ}.
\end{equation}

\begin{defin}
  A \emph{$\Bbbk$-oriented and $\bZ/d$-graded Kuranishi presentation} is a Kuranishi presentation $(\cQ, A, \bX)$ equipped with a lift of $\bX$ to $\Chart^{\Bbbk}_{\downarrow \cQ}$.
\end{defin}

We define a map of $\Bbbk$-oriented, $\bZ/d$-graded Kuranishi presentations from $(\cQ_1, A_1, \bX_1)$ to $(\cQ_2, A_2, \bX_2)$ to be a map of the underlying Kuranishi presentations, and a lift of the natural transformation in Diagram \eqref{eq:commutative_diagram_map_of_presentations} to a natural transformation in the following diagram
\begin{equation}
  \begin{tikzcd}
    A_1 \ar[r] \ar[d]  & \Chart^{\Bbbk}_{\downarrow \cQ_1} \ar[d,""{name=fromhere}] \\
    A_2 \ar[r,""{name=tohere}] &  \Chart^{\Bbbk}_{\downarrow \cQ_2}.   \arrow[Rightarrow,from=fromhere, to=tohere, start anchor={east},
    end anchor={south}, shorten >=3pt, shorten <=6pt,bend right=30]
  \end{tikzcd}
\end{equation}

We denote the corresponding category by $\Kur^{\Bbbk}$, and shall often abuse notation by denoting its objects $\bX$, leaving the indexing category $A$, the stratifying set $\cQ$, and the graded line $\ro_{\bX}$ implicit in the notation.
\begin{prop}
  The category $\Kur^{\Bbbk}$ is symmetric monoidal, with monoidal structure given by
  \begin{equation}
    \bX_1 \times \bX_2  \co A_1 \times A_2 \to \Chart^{\Bbbk}_{\downarrow \left(\cQ_1 \times \cQ_2\right)}.
  \end{equation}\qed
\end{prop}

\subsection{The cochain functor of strata}
\label{sec:chain-functor-strata}

In this section, we introduce a $\bZ/d$-graded cochain complex $ C^*_{st}(\bX)$ of \emph{strata} associated to each oriented Kuranishi presentation.  We define
\begin{equation} \label{eq:relative_cochains}
  C^{*}_{st}(\bX) \equiv \bigoplus_{\substack{P \subset \cQ(\bX)}} \ro_{\partial^P \bX},
\end{equation}
where we recall that $\cQ(\bX)$ is the subcategory of $\cQ$ consisting of non-empty strata, and that $\ro_{\partial^P \bX}$ is supported in degree $- \dim \partial^P \bX$.  In particular, assuming that $d$ vanishes, so that $ C^*_{st}(\bX)$ is an ordinary co-chain complex, the graded component $  C^k_{st}(\bX)$ vanishes for $k < - \dim \bX$, and if $\cM$ is non-empty, the component $  C^{-\dim \bX}_{st}(\bX)$ is of rank-$1$ corresponding to the minimal element of $\cQ$; the choice of orientation of $\bX$ determines a generator for this group. The differential on $C^k_{st}(\bX) $  is the sum
\begin{equation}
  d [\partial^P \bX] =  \sum_{\substack{\alpha \co P \to P' \\ \codim P +1  =  \codim P'}} [\partial^{P'} \bX]
\end{equation}
where we use the fact that whenever the codimension of $P$ and $P'$ differ by $1$, the interior of $\partial^{P} X$ is embedded as an open subset in the boundary of $\partial^{P'} X$, and hence that the product of an orientation of $\partial^{P} X$ with the canonical orientation of the normal direction determines an orientation of $\partial^{P'} X$. The fact that $d^2=0$ follows by consider an arrow $\alpha \co P_0 \to P_1$ between objects whose codimension differs by $2$. In this case, $|\cQ^{\alpha} |$ is $2$-dimensional, so it is a polygon with a vertex labelled by $P_0$, and edges adjacent to this vertex labelled by factorisations $P_0 \to P \to P_1$ of $\alpha$; there must therefore be exactly two such factorisations.

Let $\Ch$ be the symmetric monoidal category of $\bZ/d$-graded cochain complexes over the commutative ring $\Bbbk$.
\begin{lem}
  The assignment $\bX \mapsto  C^{*}_{st}(\bX)$  extends to a strong symmetric monoidal functor
\begin{equation}
  \begin{tikzcd}[column sep=large]
  \Kur^{\Bbbk}
  \ar[r,"C^*_{st}"] & \Ch.
  \end{tikzcd}
\end{equation}
\end{lem}
\begin{proof}
We assign to $f \co \bX_1 \to \bX_2$ the natural map
  \begin{equation}
  C^*_{st}(\bX_1) \to  C^*_{st}(\bX_2)
  \end{equation}
  associated to the orientation of $\kappa^{P_f}$ which makes the middle vertical arrow in Diagram \eqref{eq:diagram_charts_Kuranishi_map} orientation preserving. The functor is a strong monoidal functor because an orientation of $\partial^{P_1} \bX_1$ and $\partial^{P_2} \bX_2$ induces an orientation of their product, which is identified with $\partial^{P_1 \times P_2} (\bX_1 \times \bX_2)$, and it symmetric monoidal because this assignment is symmetric with respect to reversing the order of $\bX_1$ and $\bX_2$ (keeping the Koszul signs in mind).
\end{proof}

\subsection{The category of Kuranishi presentations with factorised strata}
\label{sec:categ-kuran-pres}

The difficulty with the category $\Kur$ is that there are too many morphisms for inductive arguments to be possible. We shall resolve this by constructing a category in which each Kuranishi presentation which is a product is a product in a unique way up to reordering factors.

\begin{rem}
  Our construction is in the spirit of the Isbell construction \cite{Isbell1969} for strictifying every symmetric monoidal category to a permutative one, and it could be modified to yield a permutative category mapping to $\Kur$, though this functor is not going to be an equivalence as in Isbell's work.
\end{rem}

\begin{defin}
  A \emph{Kuranishi presentation with factorised strata} consists of the following data:
  \begin{enumerate}
  \item a model $\cQ$ for manifolds with generalised corners,
    \item for each object $P$ of $\cQ$, an integer $m_P$, a collection of Kuranishi presentations $ \bX_i(P) $ indexed by $i \in \{1, \ldots, m_P\}$, and a map of models 
\begin{equation} \label{eq:factorisation_model-boundary}
\prod_{i=1}^{m_P} \cQ_i(P) \to  \cQ
\end{equation}
of the product of the underlying stratifying categories of $\bX_i(P)$ to $\cQ$, mapping the minimal element to $P$,
\item for each arrow $P \to P'$ in $\cQ$, a surjective map
  \begin{equation}
   \pi_{P}^{P'} \co \{1, \ldots, m_{P'}\} \to  \{1, \ldots, m_{P}\} 
  \end{equation}
  and a map of Kuranishi presentations
  \begin{equation} \label{eq:compatibility_factorisations_tkur}
   \prod_{j \in  (\pi_{P}^{P'})^{-1} (i)} \bX_{j}(P')  \to \bX_i(P)
 \end{equation}
whose underlying map of stratifying categories is compatible with the maps to $\cQ$, and 
\item For each composition $P \to P' \to P''$, we require that the following diagram commute
  \begin{equation} \label{eq:decompositions_boundary_strata_commute}
    \begin{tikzcd}
\displaystyle{    \prod_{k \in (\pi_{P}^{P''})^{-1} (i)}} \bX_{k}(P'') \ar[dr] \ar[r] & \displaystyle{  \prod_{j \in(\pi_{P}^{P'})^{-1} (i) } } \bX_{j}(P') \ar[d] \\
 &  \bX_i(P) . 
    \end{tikzcd}
     \end{equation}
\end{enumerate}
  \end{defin}

  We write $\tilde{\bX}$ be the totality of such data, which we call a Kuranishi presentation with factorised strata. We also write
  \begin{equation} \label{eq:functor_Kurt_to_Kur}
    \bX \equiv  \prod_{i \in \{1, \ldots,m_{\min \cQ}\} } \bX_i(\min \cQ),
  \end{equation}
  which is a Kuranishi presentation stratified by $\cQ$.

     Morphisms between such data are given by inclusions of boundary strata equipped with the chosen factorisation of the boundary sub-strata:
     \begin{defin}
       A morphism  $ \tilde{\bX} \to \tilde{\bY}$ of Kuranishi presentations with factorisations of boundary strata consists of
       \begin{enumerate}
\item an object $P_f$ of $\cQ_{\tilde{\bY}} $, and an isomorphism  $f \co \cQ_{\tilde{\bX}} \cong \partial^{P_f} \cQ_{\tilde{\bY}}$, and
\item for each object $P$ of $\cQ_{\tilde{\bX}}$, a permutation $f \co \{1, \ldots, m_{P}\} \cong \{1, \ldots, m_{f(P)}\}$, so that $\bX_{i}(P) = \bY_{f(i)}(f(P))$ for each $1 \leq i \leq m_P$.
\end{enumerate}
Note that the data associated to $P_f$ induces a map $\bX \to \bY$ of Kuranishi presentation. We require the following conditions to hold:
\begin{enumerate}
\item  for each object $P$ of $\cQ_{\tilde{\bX}}$, the following diagram commutes
  \begin{equation} \label{eq:commuting_diagram_Kuranishi-space-factored}
     \begin{tikzcd}
  \displaystyle{     \prod_{i=1}^{m_P}  \bX_{i}(P')} \ar[d] \ar[r] &  \displaystyle{  \prod_{j=1}^{m_{f(P)}}  \bY_{j}(P') }  \ar[d] \\
  \bX \ar[r]  &   \bY .
    \end{tikzcd}   
  \end{equation}
  \item for each arrow $P \to P'$, the isomorphisms of sets labelling the factors of the strata of $\tilde{\bX}$ and $\tilde{\bY}$ fit in a commutative diagram
  \begin{equation}\label{eq:commutative_diagram_map_factorisation}
    \begin{tikzcd}
      \{1 , \ldots, m_{P'} \} \ar[r] \ar[d] & \{ 1, \ldots, m_{f(P')} \} \ar[d] \\
       \{1 , \ldots, m_{P} \}  \ar[r] & \{ 1, \ldots, m_{f(P)} \}. 
     \end{tikzcd}
   \end{equation}
\end{enumerate}
 \end{defin}

 This construction defines a category $\widetilde{\Kur}$ of \emph{Kuranishi presentations with factorisation of boundary strata}. We stress that our definition of morphisms in this category imposes the condition that  $\bX_{i}(P)$ and $ \bY_{f(i)}(f(P))$ be equal, rather than the (categorically) more natural notion of isomorphism.

 Next, we consider the multiplicative structure:
    \begin{lem}
      The category $\widetilde{\Kur}$ is a symmetric monoidal category, which is equipped with a strong symmetric monoidal functor
$\widetilde{\bX} \mapsto \bX$   to $\Kur$.
    \end{lem}
    \begin{proof}
      We define $\tilde{\bX} \times \tilde{\bY}$ to be the product $\cQ_{\bX} \times \cQ_{\bY}$, and the assignment of the sets  $\{ \bX_i(P)\}_{i}$ and $\{  \bY_j(P')\}_{j} $ of Kuranishi presentations to the element  $(P, P')$, which we  index by the $\{ 1, \ldots, m_{P} + m_{P'} \}$ under the order preserving isomorphism
  \begin{equation}
    \{ 1, \ldots, m_P \} \amalg \{1, \ldots, m_{P'} \} \cong  \{ 1, \ldots, m_{P}, m_{P} +1, \ldots, m_{P} + m_{P'}\}.
  \end{equation}
The rest of the data is determined by this isomorphism.  The symmetry isomorphism $\tilde{\bX} \times \tilde{\bY} \cong  \tilde{\bY} \times \tilde{\bX}$ is given by permutation of factors and the reordering of $  \{ 1, \ldots,  m_{P} + m_{P'}\}$.    The forgetful map to $\Kur$ assigns to $\widetilde{\bX} $ the product $\bX$ of the factors assigned to the minimal element of $\cQ$, and is evidently symmetric monoidal.
       \end{proof}

       At this stage, we define $\widetilde{\Kur}^{\Bbbk}$ to be the fibre product of $\widetilde{\Kur}$ and $\Kur^{\Bbbk}$ over $\Kur$, i.e. an object is a Kuranishi presentation $\tilde{\bX}$ with a factorisation of its boundary strata, equipped with a relative orientation of the underlying presentation $\bX$, and morphisms are maps of Kuranishi presentation with factorised strata, together with a choice of an orientation of the corresponding normal orientation line. This is again a symmetric monoidal category, equipped with a strong symmetric monoidal functor to $\Kur^{\Bbbk}$. 
       \subsection{Lifting symmetric monoidal natural transformations}
\label{sec:strict-theor-virt}

We can now formulate the key technical result which will allow us to obtain a theory of virtual counts from a zig-zag of monoidal transformations between symmetric monoidal functors:
\begin{prop}
  \label{prop:cofibrant_property_of_lift_Kur}
Let $F$ and $G$ be symmetric monoidal functors from $\widetilde{\Kur}^{\Bbbk}$ to the category of chain complexes, which factor through $\Kur^{\Bbbk} $.  Given a monoidal weak equivalence $\pi \co G \Rightarrow F$, and a monoidal natural transformation $\eta \co  C^*_{st} \Rightarrow F$, there exists a monoidal natural transformation $C^*_{st} \to F $ so that the following homology-level diagram commutes:
  \begin{equation}
    \begin{tikzcd}
      & H^* G \ar[d,Rightarrow, "H^* \pi"] \\
      H^*_{st} \ar[ur,Rightarrow] \ar[r,Rightarrow, "H^* \eta"] & H^* F.
    \end{tikzcd}
  \end{equation}
\end{prop}

\begin{rem}
If the ring $\Bbbk$ has characteristic $0$, the statement of Proposition \ref{prop:cofibrant_property_of_lift_Kur} can be strengthened to assert the existence of a unique chain-level lift, up to contractible choice. The proof would require introducing more tools of homotopical algebra (namely, the notion of homotopy of symmetric monoidal functors) than is justified by the applications. The restriction to characteristic $0$ appears because the proof known to the author uses Sullivan's PL differential forms on the interval as a model for a path object for the category of commutative differential graded algebras. 
\end{rem}

\begin{proof}[Proof  of Proposition \ref{prop:cofibrant_property_of_lift_Kur}]
\label{sec:proof-proposition-}

The construction of the lift in Proposition \ref{prop:cofibrant_property_of_lift_Kur} proceeds by induction on the \emph{depth} of objects of $\widetilde{\Kur}^{\Bbbk}$, i.e. by the maximal codimension of a non-empty stratum. We write $\widetilde{\Kur}^{\Bbbk}_{i}$ for the subcategory of objects of depth $i$, and note that:
\begin{enumerate}
    \item There are no morphisms from objects of  $\widetilde{\Kur}^{\Bbbk}_{i}$ to those of  $\widetilde{\Kur}^{\Bbbk}_{j}$ unless $i \leq j$.
\item Any morphism in $\widetilde{\Kur}^{\Bbbk}_{i}$ projects to reordering of factors in $\Kur^{\Bbbk}$ .
\item The monoidal structure maps $\widetilde{\Kur}^{\Bbbk}_{i} \times \widetilde{\Kur}^{\Bbbk}_{j}$ to $\widetilde{\Kur}^{\Bbbk}_{i+j}$.
\end{enumerate}

The base case of the lifting process is to consider the restriction to $ \widetilde{\Kur}^{\Bbbk}_{0}$. 
\begin{lem}
  The restriction of $ C^*_{st}$ to $ \widetilde{\Kur}^{\Bbbk}_{0}$ admits a monoidal natural transformation to $G$ which lifts $\eta$ on homology. 
\end{lem}
\begin{proof}
  Consider the equivalence relation on the set indecomposable objects of $\widetilde{\Kur}^{\Bbbk}_{0} $ given by the relation of the existence of a morphism between them.  We pick a representative $\tilde{\bX}$ of such an equivalence class, and choose degree $-n$ cycle in $G(\tilde{\bX})$ which represents, in $H^* G(\tilde{\bX}) \cong H^* F(\tilde{\bX})$, the class induced by the image of the canonical generator of $ C^{-n}_{st}(\tilde{\bX}) \cong \Bbbk$ in $F(\tilde{\bX})$. Since all morphisms among indecomposable objects of $\widetilde{\Kur}^{\Bbbk}_{0} $ map to identity morphisms in $\Kur^{\Bbbk}$, this choice determines the lift on this full subcategory.

  The monoidal condition determines a unique extension of this choice to a natural transformation on all of $ \widetilde{\Kur}^{\Bbbk}_{0} $. Functoriality on morphisms follows from the fact that morphisms in $ \widetilde{\Kur}^{\Bbbk}_{0} $ project to reordering factors.
\end{proof}

Having completed the proof of the Lemma, we return to the proof of Proposition \ref{prop:cofibrant_property_of_lift_Kur}: let us write $\widetilde{\Kur}^{\Bbbk}_{\leq m} $ for the full subcategory of $\widetilde{\Kur}^{\Bbbk}$ consisting of objects of depth bounded by $m$. Our inductive assumption is then the existence of a natural transformation
\begin{equation}
 C^*_{st}| \widetilde{\Kur}^{\Bbbk}_{\leq m}  \Rightarrow G| \widetilde{\Kur}^{\Bbbk}_{\leq m},
\end{equation}
with the following properties:
\begin{enumerate}
\item (Lifting) The following diagram of homology level natural transformations commutes:
  \begin{equation}
    \begin{tikzcd}
      &   H^* G| \widetilde{\Kur}^{\Bbbk}_{\leq m} \ar[Rightarrow, d] \\
      H^{*}_{st}| \widetilde{\Kur}^{\Bbbk}_{\leq m}  \ar[Rightarrow, r] \ar[Rightarrow, ur] & H^* F| \widetilde{\Kur}^{\Bbbk}_{\leq m}
    \end{tikzcd}
  \end{equation}
\item (Monoidal structure) If $\tilde{\bX}$ and $\tilde{\bY}$ are objects of $\widetilde{\Kur}^{\Bbbk}_{i}$ and $\widetilde{\Kur}^{\Bbbk}_{j}$ with $i + j \leq m$, then the following diagram commutes
  \begin{equation} \label{eq:lift_is_monoidal}
  \begin{tikzcd}
     C^*_{st}(\tilde{\bX}) \otimes C^*_{st}(\tilde{\bY}) \ar[r] \ar[d] & C^*_{st}( \tilde{\bX} \times \tilde{\bY}) \ar[d] \\
     G(\tilde{\bX}) \otimes G(\tilde{\bY}) \ar[r]  & G(\tilde{\bX} \times \tilde{\bY})
  \end{tikzcd}
  \end{equation}
  \end{enumerate}

  We now proceed with the induction step:  since the functor $C^*_{st} $ is strongly symmetric monoidal, the value of the desired extension on every object of $\widetilde{\Kur}^{\Bbbk}_{m+1}  $ which decomposes as a product $\tilde{\bX} \times \tilde{\bY}$, is determined by Diagram \eqref{eq:lift_is_monoidal}. The resulting map is then functorial on the subcategory of decomposable objects of $\widetilde{\Kur}^{\Bbbk}_{m+1}  $ again because all morphisms project to reordering factors.

  Finally, we consider the equivalence relation among objects of $\widetilde{\Kur}^{\Bbbk}_{m+1}$ which are not products, given as above by the existence of a morphism. Choosing a representative $\tilde{\bX}$ of each such class, observe that $\partial^{P} \tilde{\bX}$ lies in $ \widetilde{\Kur}^{\Bbbk}_{\leq m} $ whenever $P$ is a non-minimal element of $\cQ_{\tilde{\bX}}$. The desired functoriality then determines the value of the natural transformation on the subcomplex of $  C^*_{st}(\tilde{\bX})$ consisting of all elements of degree strictly larger than $-\dim \tilde{\bX}$. It thus remains to extend the natural transformation to $  C^{- \dim \tilde{\bX}}_{st}(\tilde{\bX})$, which is a rank-$1$ free abelian group. We may pick any element of $G(\tilde{\bX}) $ of degree $-\dim \tilde{\bX}$ whose boundary is the image of the differential of this generator under the previously chosen lift; such an element must exist because of the isomorphism $H^* G(\tilde{\bX}) \cong H^* F(\tilde{\bX})$. The fact that all morphisms among irreducible objects map to identities
implies that this assignment is functorial.  This completes the inductive step in the construction of the lift, and hence the proof of the Proposition.
  
\end{proof}

\section{Virtual counts for charts and global quotients}
\label{sec:virt-counts-global}

In this section, we consider the special case of global quotients, and construct a theory of virtual counts using a variant of Pardon's notion of virtual cochains, which incorporates a twist inspired by \cite{AbouzaidBlumberg2021}. In order to avoid unnecessary repetition, we introduce most of our constructions for arbitrary Kuranishi charts.

\subsection{Virtual counts for global quotients I}
\label{sec:virt-counts-glob}

We start by constructing the desired zig-zag for a subcategory of Kuranishi charts, which we call the \emph{naive category of global charts,} whose objects are quintuples $(\cQ, G, X, V, \fs)$ so that $\fs^{-1}(0)$ is compact, and morphisms are those morphisms of Kuranishi charts as in Definition \ref{def:Kuranishi-charts} with the property that the map of vector bundles is an isomorphism.  For applications, this is an unreasonably small category because there are too few morphisms, but the restrictions allow us to explain some basic ideas with the minimal technical baggage. We will discuss a more reasonable category of global charts in Section \ref{sec:virt-counts-glob-II} below.

We begin by considering a general object of the category of Kuranishi charts: denote by $\ro_V^{-1}$ the inverse of the orientation line of $V$ as a local system on $X$, with associated twisted { singular} chains $ C_*(X; \ro_V^{-1}) $ (here, $\ro_V$ is supported in degree $n$ because we are temporarily using homological conventions). Recalling that $Z$ is the zero locus of $s$, we use the following notation for the twisted symmetric cubical chains (c.f. Appendix \ref{sec:symm-cubic-sets}) of the pair $(X,X \setminus Z)$:
\begin{equation}
     C_p(X|Z;\ro_V^{-1} ) \equiv \coker\left(C_p(X \setminus Z;\ro_V^{-1} ) \to   C_p(X;\ro_V^{-1} ) \right).
\end{equation}

\begin{rem}
  The reader may wonder at this stage why we are using symmetric cubical chains instead of singular ones. We refer to Remark \ref{rem:why-cubes} below for the answer.
\end{rem}

Since $X$, $V$, and $Z$ cary compatible actions of $G$, these relative cochains naturally form a module over the chains $C_* G$.  We write $ C_{*}(E G)$ for the standard bar resolution of $\Bbbk$ as a free $C_*G$-module, which is given in degree $n$ by $C_*G^{ \otimes n+1}$, with differential given by the alternating sum of the product of successive elements. 
\begin{defin} \label{def:twisted_compact_support_relative}
  The \emph{twisted compactly supported relative virtual cochains} of a Kuranishi chart is the complex
  \begin{equation}
    C^{-*,c}_{\rel \partial}(\bX) \equiv C_*(X|Z;\ro_V^{-1})   \otimes_{C_* G}  C_{*}(E G).
  \end{equation}
\end{defin}
\begin{rem}
Our terminology follows Pardon's in \cite{Pardon2016}, who studies an untwisted version. The notation is supposed to suggest that the cohomology of the relative cochains is isomorphic (up to shift by the virtual dimension) to the relative (twisted) Borel equivariant \v{C}ech cohomology of the pair $(Z,\partial Z)$ with compact support. If we restrict attention to the case where $\Bbbk$ is a field of characteristic $0$, the fact that the isotropy groups are finite implies that the cohomology of this complex is isomorphic to the compactly supported relative cohomology of the quotient of the zero locus by $G$. 
\end{rem}
While the input for the virtual cochains is a Kuranishi chart, this is not a functorial notion in this general context, and we will resolve this problem in Section \ref{sec:homot-norm-bundl}.  It is however straightforward to check that Definition \ref{def:twisted_compact_support_relative} yields a symmetric monoidal functor on the naive category of global Kuranishi charts.

Next, we observe that the section $\fs$ defines a natural $G$-equivariant map of pairs
\begin{equation}
  X|Z \to V|0,  
\end{equation}
which yields upon passing to twisted chains the map
\begin{equation}
  C_*(X|Z;\ro_V^{-1})   \to    C_*(V|0;\ro_V^{-1}).
\end{equation}
The right hand side has a quasi-isomorphic subcomplex
\begin{equation}
      C^{\pitchfork}_*(V|0;\ro_V^{-1}) \subset C_*(V|0;\ro_V^{-1}).
\end{equation}
consisting of maps from cubes whose strata of dimension smaller than the rank of $V$ are disjoint from the $0$ section. This subcomplex  admits a canonical map to the ground ring, given by the pairing between chains of degree $\dim V$ and the orientation line. Since this pairing is equivariant, we conclude:
\begin{lem} \label{lem:evaluation_map_naive}
  The twisted virtual cochains define a symmetric monoidal functor from the naive category of global charts to the category of cochain complexes, which is equipped with a morphism to the constant functor, in the homotopy category of symmetric monoidal functors. \qed
\end{lem}
Our next task is to compare this construction with the stratum functor. To this end, we consider the complex
\begin{equation} \label{eq:compactly_supported_cochains_chart}
  C^{*,c}(\bX) \equiv  \left(\bigoplus_{P \in \Ob \cQ}   C^{*,c}_{\rel \partial}(\partial^P \bX) \otimes \ro_P, d \right)
\end{equation}
where $\ro_P$ is the orientation line of $\kappa^{P}$ (which has degree $-\codim P$ because we are back to using cohomological conventions), and the differential is the sum of the internal differentials of each summand with the tensor product, over all arrows $P \to P'$ with $\codim P' = \codim P + 1$, of the induced map of virtual cochains with the map
\begin{equation}
    \ro_{P'} \to \ro_{P}
  \end{equation}
  associated to the inclusion of $\kappa^{P}$ as a codimension-$1$ boundary stratum of $\kappa^{P'}$, and the standard normal orientation.

  The following is a consequence of the fact that the quotient $X/G$ is an orbifold with boundary, together with the compactness of $Z$, and can be proved by the same methods used to prove Poincar\'e duality for compact topological manifolds (c.f. \cite[Theorem 3.30]{Hatcher2002}). For a statement of a result directly applicable to the case $G$ is a finite group, the reader may consult  Section 4.3 of \cite{Pardon2016}.
\begin{lem} \label{lem:absolute_cochians_have_fundamental_class}
 If $\bX$ is a global Kuranishi chart, then assuming that either
  \begin{enumerate}
  \item the action of $G$ on $Z$ is free, or
    \item the ring $\Bbbk$ is a field of characteristic $0$,
    \end{enumerate}
the cohomology $H^{*,c}(\bX)$ is supported in degrees greater than or equal to $- \dim \bX$ and it is canonically isomorphic, in this degree, to a direct sum of copies of $\ro_{\bX}$, indexed by the components of $Z/G$. \qed
\end{lem}

To relate this to our previous construction, we consider the direct sum
\begin{equation}
  \bigoplus_{P \in \Ob \cQ}   C^{*,c}( \partial^P \bX).
\end{equation}
We shall define a differential which is the sum of the differential on each factor and of terms associated to each arrow in $\cQ$. To this end, we consider an arrow $\alpha \co P_0 \to P_1$ with $\codim P_1 = \codim P_1 + 1$, and we shall define a map
\begin{multline} \label{eq:differential_reconstructed_relative_chains}
  C^{*,c}( \partial^{P_0} \bX)  \cong \bigoplus_{\beta \co P_0 \to P} C^{*,c}_{\rel \partial}(\partial^{P} \bX) \otimes \ro_\beta  \to  \\  C^{*,c}( \partial^{P_1} \bX) \cong \bigoplus_{\gamma \co P_1 \to P} C^{*,c}_{\rel \partial}(\partial^{P} \bX) \otimes \ro_\gamma, 
\end{multline}
where $\ro_\alpha$ is the orientation line of $\kappa^\alpha$. The desired map vanishes except on morphisms which factor through $\alpha$, in which case it is given by the tensor product of the identity with the degree-$1$ map from the orientation line of $\kappa^{\gamma \circ \alpha}$ to the orientation line of $\kappa^{\gamma}$ associated to the corresponding codimension $1$ boundary inclusion.
\begin{lem} \label{lem:relative_into_reconstructed_quasi-iso}
  The inclusion
  \begin{equation}
     C^{*,c}_{\rel \partial}(\bX) \to  \left(  \bigoplus_{P \in \cQ}   C^{*,c}( \partial^P \bX), d \right) 
     \end{equation}
     induces an isomorphism on cohomology.
\end{lem}
\begin{proof}
  The filtration of each chain complex $ C^{*,c}( \partial^P \bX) $ indexed by $\cQ$ induces a filtration on $ C^{*,c}_{\rel \partial}(\bX)$. The associated graded chain complex associated to $P$ is naturally isomorphic to the tensor product of $C^{*,c}_{\rel \partial}(\partial^P \bX) $  with the cellular chain complex of the pair $(\kappa^P, \partial_{-} \kappa^P)$, where $\partial_{-} \kappa^P$ is the complement of the union of faces which contain the vertex associated to $P$. Since the pair $(\kappa^P, \partial_{-} \kappa^P) $ is contractible unless $P$ is the minimum, the result follows.
\end{proof}

\begin{lem} \label{lem:morphism_to_constant_functor_global}
  If $\Bbbk$ has characteristic $0$ there is a morphism from the stratum functor to the virtual cochain functor, in the homotopy category of symmetric monoidal functors with domain the naive category of $\Bbbk$-oriented global charts. If $\Bbbk$ is an arbitrary ring, such a morphism exists on the subcategory of charts with trivial isotropy. 
\end{lem}
\begin{proof}
  Consider the canonical truncation
\begin{equation}
  C^{*,c}( \partial^P \bX)_{(-\infty, - \dim \partial^P \bX]}    \subset   C^{*,c}( \partial^P \bX)
\end{equation}
given by all element of degree strictly smaller than $- \dim \partial^P \bX$, together with the kernel of the differential
\begin{equation}
   C^{- \dim \partial^P \bX,c}( \partial^P \bX) \to  C^{1- \dim \partial^P \bX,c}( \partial^P \bX).
 \end{equation}
 Taking the direct sum over all $P \in \cQ$, we obtain a subcomplex 
\begin{equation}
 \left(\bigoplus_{P \in \cQ}   C^{*,c}( \partial^P \bX), d \right)_{(-\infty, - \dim \partial^P \bX]}  \subset \left( \bigoplus_{P \in \cQ}   C^{*,c}( \partial^P \bX), d \right),
\end{equation}
where we use the differential from Equation \eqref{eq:differential_reconstructed_relative_chains}. This truncation is preserved by the monoidal structure, and hence yields a symmetric monoidal functor.

Next, we use Lemma \ref{lem:relative_into_reconstructed_quasi-iso} to see that the projection map
\begin{equation}
 C^{*,c}( \partial^P \bX)_{(-\infty, - \dim \partial^P \bX]}  \to  H^{- \dim \partial^P \bX,c}( \partial^P \bX)  \cong \ro_{\partial^P \bX}.
  \end{equation}
  is a quasi-isomorphism, which is functorial and monoidal. We thus obtain a zig-zag
  \begin{equation}
    \begin{tikzcd}
      C^*_{st}(\bX) & \ar[l] \displaystyle{ \left(\bigoplus_{P \in \cQ}   C^{*,c}( \partial^P \bX)_{(-\infty, - \dim \partial^P \bX]} , d \right)} \ar[d] \\
C^{*,c}_{\rel \partial}(\bX)   \ar[r]  & \displaystyle{ \left( \bigoplus_{P \in \cQ}   C^{*,c}( \partial^P \bX), d \right)},    
    \end{tikzcd}
  \end{equation}
  where the two horizontal arrows are quasi-isomorphisms. 
\end{proof}

\begin{rem}
  At this stage, one can construct an analogue of the category $\widetilde{\Kur}^{\Bbbk}$, lifting the naive category of global charts, and which satisfies Proposition \ref{prop:cofibrant_property_of_lift_Kur}, so that the above zig-zag gives rise to a symmetric monoidal functor, yielding a multiplicative theory of virtual counts. We shall not do this because, as noted at the beginning of this section, our goal is to explain a basic idea in a toy situation which we do not expect to be sufficiently flexible for substantial applications.
\end{rem}

\subsection{The Thom isomorphism via homotopy normal bundles}
\label{sec:thoms-isom-via}

There are two fundamental problems in extending Lemma \ref{lem:morphism_to_constant_functor_global} to Kuranishi presentations: the first is that a map of Kuranishi charts does not induce a map on virtual cochains, and the second is that the non-compactness of the local charts implies that there are two local notions of virtual cochains, depending on whether a condition of compact support is imposed. In this section, we explain how to resolve the first problem using homotopy normal bundles, which we discuss in more generality in the next section.

It is easiest to understand our strategy starting with a simpler situation: let $X$ be a topological manifold, and $V$ a vector bundle on $X$ equipped as before with a section whose zero-locus we denote $Z$. Let $Y$ be a manifold equipped with an open embedding in $Z$, and assume that $\fs$ is transverse along the image of $Y$; in particular, the composite $Y \to X$ is a locally flat embedding. Our goal is to construct a map
\begin{equation}
  C_*(Y) \dashrightarrow C_*(X|Z; \ro_{V}^{-1})  
\end{equation}
in the homotopy category, which we will represent by an explicit zig-zag involving one middle term.

We denote by $N X$ the homotopy normal bundle of $Z$ in $X$: this is the space over $Z$ with fibre at $z$ given by
\begin{equation}
    N_{z} X = X \setminus \left(Z \setminus z \right).
\end{equation}
We call the point $z \in N_{z} X$ the origin. Since we have not imposed any global transversality assumption, the local homotopy type of the fibre $N_zX$ near the origin may depend on the point $z$. However, the restriction to $Y$ is a fibre bundle; we shall abuse notation and write $NX$ for this restriction. The  projection map $NX \to Y$ has a canonical section, which we call the $0$-section.  Note as well that we have a natural map
\begin{equation}
  N X \to X  
\end{equation}
which is the inclusion of each fibre. A key point which we shall use later is that this map takes the complement of the zero section in $NX$ to the complement of $Z$; this would fail if we were to
use the homotopy normal bundle of $Y$ in $X$.

Let us denote by
\begin{equation}
    C_{*}(Y; C_*(NX) \otimes \ro_{V}^{-1}) 
\end{equation}
the  twisted  symmetric bi-cubical chain complex of the fibre bundle $NX \to Y$. This is the direct sum of the pullback of the line $\ro_{V}^{-1}$ by maps
\begin{equation}
\square^p \times \square^q \to NX   
\end{equation}
whose projection to $Y$ is independent of the second variable (i.e. factors through $\square^p$), modulo those which are degenerate, and the ideal generated by a cube and its image under a permutation of any two factors of $\square^p$ or of $\square^q$. The differential is the sum of two (sets of) terms, corresponding to the boundaries of the factors of $\square^p \times \square^q$.

We write $NX|0$ for the pair $(NX, NX \setminus Y)$, and obtain a corresponding relative chain complex
\begin{equation}
    C_{*}(Y; C_*(NX|0) \otimes \ro_{V}^{-1}) 
\end{equation}
by taking the quotient by the subcomplex generated by chains which lie in the complement of the zero section. Finally, we consider the subcomplex
\begin{equation}
     C^{\pitchfork}_{*}(Y; C_*(NX|0) \otimes \ro_{V}^{-1}) 
   \end{equation}
   generated by those chains with domain $\square^p \times \square^q$ satisfying the following condition:
   \begin{equation}
\parbox{31em}{the image of each stratum $\square^p \times \square^i $ for $i < \dim V$ is  disjoint from $Y$.}          
   \end{equation}
\begin{lem} \label{lem:pairing_orientation_line_simplex}
  There is a natural equivalence
  \begin{equation}
     C^{\pitchfork}_{*}(Y; C_*(NX|0) \otimes \ro_{V}^{-1}) \to C_*(Y).   
  \end{equation}
\end{lem}
\begin{proof}[Sketch of proof:]
The section $\fs$ identifies the homology of $N_y X$ relative the complement of the point $y$ with the orientation line of $V_y$.   The map vanishes except on chains whose domain is $\square^p \times \square^{\dim V}$, where it is given by the pairing of the fibre with the orientation line, which is well-defined because the image of $ \square^p \times \partial \square^{\dim V}$ is disjoint from the $0$-section. Excision and Mayer-Vietoris reduce the proof that the map is an equivalence to the corresponding statement for a fibre, which follows from the fact that the pair $N_yX|0$ has the homotopy type of a sphere, hence has (reduced) homology supported in a single degree.
\end{proof}
On the other hand, the map $NX \to X$, and the identification $\square^p \times \square^q \cong  \square^{p+q}$ induce a map
\begin{equation} \label{eq:evaluation_map_normal_to_total}
     C^{\pitchfork}_{*}(Y; C_*(NX|0) \otimes \ro_{V}^{-1}) \to C_*(X|Z;\ro_{V}^{-1}),
\end{equation}
which together with Lemma \ref{lem:pairing_orientation_line_simplex} yields the desired zig-zag:
\begin{equation}
    C_*(Y) \leftarrow  C^{\pitchfork}_{*}(Y; C_*(NX|0) \otimes \ro_{V}^{-1}) \to C_*(X|Z;\ro_{V}^{-1}).
\end{equation}

\subsection{Homotopy normal bundles and multisimplicial chains I: fixed automorphism groups}
\label{sec:homot-norm-bundl}

We now generalise the discussion of the previous section to the setting of Kuranishi charts. Since the map we constructed is represented by a zig-zag, it is natural to expect that its functoriality will be formulated as a homotopy coherent functor on the category of charts. We thus start by considering a sequence of maps
\begin{equation}
  \begin{tikzcd}
    \bX_0 \ar[r,"f^1"] &  \bX_1 \ar[r,"f^2"] & \cdots \ar[r,"f^n"] & \bX_n
  \end{tikzcd}
\end{equation}
of Kuranishi charts; we write $f_i^{j}$ for the composite map from $\bX_i$ to $\bX_j$. While we shall consider the general case later, we begin by assuming that the underlying maps of groups are isomorphisms; we write $G$ for this group. We thus obtain a commutative diagram of $G$-equivariant embeddings
\begin{equation}
    \begin{tikzcd}
      X_0 \ar[r,"f^1"] &  X_1 \ar[r,"f^2"] & \cdots \ar[r,"f^n"] & X_n \\
       Z_0 \ar[r] \ar[u] &  Z_1 \ar[r]  \ar[u]  & \cdots \ar[r]  \ar[u]  & Z_n . \ar[u] 
  \end{tikzcd}  
\end{equation}
We associate to each of these maps the bundle $N f^i$ over $X_{i-1}$, whose fibre at a point $x$ is
\begin{equation}
 \partial^{P_{f_i}} X_{i} \setminus \left( X_{i-1} \cup Z_i\right) \cup \{x\}  \subset  \partial^{P_{f_i}} X_i,
\end{equation}
where we recall that $\partial^{P_{f_i}} X_{i} $ is the stratum of $X_i$ receiving the map from $X_{i-1}$. Note that this is not quite the homotopy normal bundle of $f^i$, because we also  use the data of the sets $Z_i$ in its definition. The fact that this is a bundle follows from the transversality assumption in the definition of maps of Kuranishi charts, which implies that there is a neighbourhood of the image of $X_{i-1}$ in $ \partial^{P_{f_i}} X_i$ with the property that its intersection with $Z_i$ agrees with $Z_{i-1}$ (and in particular is contained in $X_{i-1}$).

We have a natural map
\begin{equation}
\iota_i \co N f^i \to X_{i},
\end{equation}
and we inductively define $N^{\vec{f}} f^{i+1}$ to be the fibration over $N^{\vec{f}} f^i$ obtained by pulling back $N f^{i+1}$. This construction yields a sequence of fibre bundles
\begin{equation}
  \begin{tikzcd}
      X_0  &  N^{\vec{f}} f^1 \ar[l] &   N^{\vec{f}} f^2 \ar[l] & \cdots \ar[l] & N^{\vec{f}} f^n \ar[l]
  \end{tikzcd}
\end{equation}
over $X_0$. We shall consider maps
  \begin{equation}
    \square^{k_0} \times \cdots \times \square^{k_n} \to N^{\vec{f}} f^n ,  
  \end{equation}
  such that the projection to $N^{\vec{f}} f^i$ factors through the projection of the domain to $\square^{k_0} \times \cdots \times \square^{k_i}$, i.e. so that we have a commutative diagram
  \begin{equation}
    \begin{tikzcd}[column sep=15]
    \square^{k_0} \ar[d] & \square^{k_0} \times \square^{k_1} \ar[d] \ar[l] &  \square^{k_0} \times \square^{k_1} \times \square^{k_2} \ar[l] \ar[d]  &  \ar[l] \ar[d]  \cdots  &  \ar[l]  \square^{k_0} \times \cdots \times \square^{k_n} \ar[d] \\
      X_0  &  N^{\vec{f}} f^1 \ar[l] &   N^{\vec{f}} f^2 \ar[l] & \cdots \ar[l] & N^{\vec{f}} f^n \ar[l].
  \end{tikzcd}
\end{equation}
We further restrict attention to those maps which are transverse to the zero section
\begin{equation}
   N^{\vec{f}} f^{i-1} \to    N^{\vec{f}} f^i
\end{equation}
in the sense that, for each point in $ \square^{k_0} \times \cdots \times  \square^{k_{i-1}}$,
  \begin{equation}
    \label{eq:chains_transverse_to_o-section}
    \parbox{32em}{the image  in $ N^{\vec{f}} f^i$  of each $d$-dimensional stratum of $\square^{k_{i}} \times \cdots \times \square^{k_j}$ lies in the complement of the zero section if $d < \dim V_{i}/V_{i-1}$.}
  \end{equation}

  \begin{rem} \label{rem:why-cubes}
   In the singular theory, one would consider multi-simplices, i.e. diagrams of the form
    \begin{equation}
    \begin{tikzcd}[column sep=15]
    \Delta^{k_0} \ar[d] & \Delta^{k_0} \times \Delta^{k_1} \ar[d] \ar[l] &  \Delta^{k_0} \times \Delta^{k_1} \times \Delta^{k_2} \ar[l] \ar[d]  &  \ar[l] \ar[d]  \cdots  &  \ar[l]  \Delta^{k_0} \times \cdots \times \Delta^{k_n} \ar[d] \\
      X_0  &  N^{\vec{f}} f^1 \ar[l] &   N^{\vec{f}} f^2 \ar[l] & \cdots \ar[l] & N^{\vec{f}} f^n \ar[l].
  \end{tikzcd}
\end{equation}
It would seem straightforward to define transversality as above in terms of the assumption that strata of the form $ \Delta^{k_0} \times \cdots \times \Delta^{d} \subset  \Delta^{k_0} \times \cdots \times \Delta^{k_i} $ are disjoint from the zero section of $ N^{\vec{f}} f^i$ whenever $d < \dim V_{i}/V_{i-1}$. However, this condition is not compatible with the maps associated to forgetting an element of the sequence $\vec{f}$, because these maps involve applying the prismatic subdivision, which introduces new strata. One can in fact restore functoriality at this stage by imposing the appropriate transversality for all prismatic subvisions, which ultimately defines a functor on the category of simplices of $\Chart$, but this functor fails to be multiplicative because there are new transversality conditions that have to be imposed whenever one considers a product of sequences.
  \end{rem}

This leads us to our main definition:
\begin{defin} \label{def:twisted_vcochains-constant-G}
  The \emph{twisted virtual cochains} of a sequence $\vec{f} = \{f^i\}$ of composable morphisms of Kuranishi charts with constant automorphism group $G$ is the co-chain complex
  \begin{equation}
    C^{-*,c}_{\rel \partial}(\vec{f}) \equiv C^{\pitchfork}_*(X_0|Z_0; C_*(N^{\vec{f}}f^1|0) \otimes \cdots \otimes C_*(N^{\vec{f}}f^n|0)  \otimes \ro_{V_n}^{-1}) \otimes_{C_* G} C_* (EG),
  \end{equation}
  generated by the pullback of $ \ro_{V_n}^{-1}$ by multi-cubes satisfying Condition \eqref{eq:chains_transverse_to_o-section}, modulo those whose projection to $N^{\vec{f}} f^i$ lies in the complement of the origin (the case $i=0$ is implicitly included in these conditions if we use the convention that $f^0=\id_{\bX_0}$). The differential is given by the alternating sum of the restriction to the boundary facets of the domain.
\end{defin}

The following result states the key structural properties of this construction; we write $\partial^i \vec{f}$ for the sequence obtained by (i) omitting $f^1$ if $i=0$, (ii) composing $f^i$ and $f^{i+1}$ if $1 \leq i < n $, (iii) omitting $f^n$ if $i=n$. Note that if $0 \leq i < j \leq n$ we have $\partial^i \partial^j \vec{f}  = \partial^{j-1} \partial^i \vec{f}$.
\begin{lem}
  There are natural maps
  \begin{equation} \label{eq:restriction_map}
    \partial^i \co   C^{*,c}_{\rel \partial}(\vec{f}) \to     C^{*,c}_{\rel \partial}(\partial^i \vec{f}),
  \end{equation}
which are quasi-isomorphisms except possibly for $i=0$, and which yield a commutative diagram 
           \begin{equation}
             \begin{tikzcd}
               C^{*,c}_{\rel \partial}(\vec{f})  \ar[r] \ar[d] & C^{*,c}_{\rel \partial}(\partial^i \vec{f})   \ar[d] \\
   C^{*,c}_{\rel \partial}(\partial^j \vec{f})              \ar[r] &
     C^{*,c}_{\rel \partial}( \partial^{j-1} \partial^i \vec{f}) 
     =   C^{*,c}_{\rel \partial}( \partial^i \partial^j \vec{f} )  .
             \end{tikzcd}
           \end{equation}
\end{lem}
\begin{proof}
  The maps are defined in each of the three cases as follows, where we implicitly use the identification $ \square^{k} \times  \square^{k'}  \cong \square^{k+k'}$ in the first two cases:
  \begin{enumerate}
  \item For $i=0$, we project from $N^{\vec{f}}f^1$ to $X_1$.
  \item For $1 \leq i < n$, we observe that we have a commutative diagram
    \begin{equation}
      \begin{tikzcd}
     \iota_i^*    N f^{i+1} \ar[r] \ar[d]  & N \left( f^i \circ f^{i+1}\right) \ar[d] \\
          N f^i \ar[r] & X_{i-1}.
      \end{tikzcd}
    \end{equation}
    Pulling back to $N^{\vec{f}} f^{i-1}$, we obtain the map of towers
\begin{equation}
  \begin{tikzcd}[column sep=8]
    X_0  \ar[d] & \cdots \ar[l] \ar[d]  &   N^{\vec{f}} f^{i-1} \ar[l] \ar[d,"="]  & N^{\vec{f}} f^{i} \ar[l]   &  N^{\vec{f}} f^{i+1} \ar[d] \ar[l] & \cdots \ar[l] \ar[d]  & N^{\vec{f}} f^n \ar[l]\ar[d]   \\
    X_0  & \cdots \ar[l] &  N^{\partial^i \vec{f}} f^{i-1} \ar[l] &   &  N^{\partial^i \vec{f}} \left( f^{i+1} \circ f^i \right) \ar[ll] & \cdots \ar[l] & N^{\partial^i \vec{f}} f^n \ar[l].
  \end{tikzcd}
\end{equation}
Composing each multi-cube with target $N^{\vec{f}} f^n $ with the rightmost vertical map, we obtain a multi-cube with the desired target.
\item For $i=n$, we use the canonical isomorphism $\ro_{V_n} \cong \ro_{V_{n-1}} \otimes \ro_{V_n/V_{n-1}} $ of local systems over $X_n$. As in Lemma \ref{lem:pairing_orientation_line_simplex}, the map is then induced by the pairing of $ \ro_{V_n/V_{n-1}}$ with cubes in the fibre of $N f^n$ of dimension $ \dim V_{n}/V_{n-1}$, whose boundary maps away from the origin. 
  \end{enumerate}
\end{proof}
The above result can be restated as the construction of a semi-simplicial chain complex, i.e. a functor from the category of non-degenerate simplices of $\Delta^n$ to the category of cochain complexes. It will be convenient later to know that this extends to a simplicial chain complex; to this end, we write $s^i \vec{f}$ for the sequence $(f^1, \ldots, f^{i}, \id_{\bX_i}, f^{i+1}, \ldots f^n )$. Since the homotopy normal bundle of the identity map is the trivial bundle with fibre a point, we have: 
\begin{lem}
  There is a natural isomorphism
  \begin{equation}
        C^{*,c}_{\rel \partial}(\vec{f}) \cong C^{*,c}_{\rel \partial}(s^i \vec{f})
      \end{equation}
      which extend $C^{*,c}_{\rel \partial}(\vec{f})$ to a functor on the category of simplices of $\Delta^n$. \qed 
\end{lem}

\subsection{Homotopy normal bundles and multisimplicial chains II: the general case}
\label{sec:homot-norm-bundl-2}

We now consider an $n$-simplex of Kuranishi charts, i.e. a functor
\begin{equation}
  \sigma \co \bfn \to \Chart,
\end{equation}
whose domain is the category $\bfn = \{ 0 < 1 < \cdots < n \}$. We write $f^i_{\sigma} \co \bX_{\sigma_{i-1}} \to \bX_{\sigma_{i}}$ for the corresponding maps, and do not assume that the underlying maps $f^i_{\sigma} \co G_{i-1} \to G_i$ are isomorphisms. We denote the induced maps of charts with automorphism group $G_{\sigma_n}$ by
\begin{equation}
  f^i_{\sigma} \times_{G_{\sigma_i}} G_{\sigma_n} \co \bX_{\sigma_{i-1}} \times_{G_{\sigma_{i-1}}} G_{\sigma_n} \to \bX_{\sigma_i} \times_{G_{\sigma_i}} G_{\sigma_n},
\end{equation}
which define a functor
\begin{equation}
  \sigma \times_{G_\sigma} G_{\sigma_n} \co \bfn \to   \Chart.
\end{equation}
More generally given a commutative diagram
\begin{equation}
  \begin{tikzcd}
    \mathbf{k} \ar[r,"\tau"] \ar[d] & \cC . \\
    \mathbf{m} \ar[ur,"\sigma"] & 
  \end{tikzcd}
\end{equation}
we obtain a map
\begin{equation}
  \tau \times_{G_\tau} G_{\sigma_n} \co  \Bbbk \to \Chart.
\end{equation}
We write $\tau \subset \sigma$ when the vertical arrow above is an inclusion.
\begin{defin} \label{def:virtual_cochains_chart}
  The \emph{twisted virtual cochains} of a simplex $\sigma \co \bfn \to \Chart$ is the co-chain complex
  \begin{equation} \label{eq:virtual_cochain_simplex}
    C^{*,c}_{\rel \partial}(\sigma) \equiv \bigoplus_{\tau \subset \sigma}  C^{*,c}_{\rel \partial}(\tau \times_{G_{\tau}} G_{\sigma_n} )[\dim \tau]
  \end{equation}
 with differential given by the alternating sum of the restriction maps in Equation \eqref{eq:restriction_map} from a stratum to its boundary facets ($\dim \tau$ refers to the dimension of the simplex corresponding to the domain of $\tau$).
\end{defin}
The following result will be essentially used, in later arguments, to control the quasi-isomorphism type of the virtual cochains that we assign to Kuranishi presentation:
\begin{lem} \label{lem:compute_virtual_cochains_simplex}
  The inclusion of the last vertex induces a quasi-isomorphism
  \begin{equation}
     C^{*,c}_{\rel \partial}(\bX_{\sigma_n}) \to   C^{*,c}_{\rel \partial}(\sigma).
   \end{equation}
\end{lem}
\begin{proof}
  Filter the complex in Equation \eqref{eq:virtual_cochain_simplex} by the difference between the dimension of $\tau$ and the maximal element of $\bfn$ appearing in $\tau$. The associated graded group is the direct sum of $C^{*,c}_{\rel \partial}(\bX_{\sigma_n})$ with (a shift) of the cone of each map
  \begin{equation}
      C^{*,c}_{\rel \partial}(\tau \times_{G_{\tau}} G_{\sigma_n} ) \to  C^{*,c}_{\rel \partial}(\partial^n \tau \times_{G_{\partial^n \tau}} G_{\sigma_n} ).
    \end{equation}
    Since this map is a quasi-isomorphism, the result follows. 
\end{proof}

There are natural maps of twisted virtual cochains associated to maps of simplices: the key point is that, if $\vec{f}$ is a sequence of composable morphisms of Kuranishi charts with underlying group $G$, and  $G \to G'$ is a map of compact Lie groups whose kernel acts freely on the sequence of objects associated to $\vec{f}$, there is a natural equivalence
  \begin{equation}
    C^{*,c}_{\rel \partial}(\vec{f})    \to   C^{*,c}_{\rel \partial}(\vec{f} \times_{G} G').
  \end{equation}  

We formulate this naturality in terms of the category $  \Delta  \Chart$ of simplices in $\Chart $ as described in Appendix \ref{sec:mult-colim-via}:
\begin{lem} \label{lem:functoriality_twisted_cochains}
  The twisted virtual cochains define a functor
  \begin{equation}
 \Delta  \Chart \to \Ch.
\end{equation} \qed
\end{lem}

\subsection{Multiplicativity of the virtual cochains}
\label{sec:mult-virt-coch}

The main purpose of this section is to prove that the twisted virtual cochains are externally multiplicative, with commutative multiplication. The basic ingredient of the proof is the following observation: let $\vec{f}$ and $\vec{g}$ be sequences of composable morphisms of the same length, with underlying maps of groups that are isomorphisms.  Let $\vec{f} \times \vec{g}$ denote the product of these morphisms, i.e. the sequence $(f^1 \times g^1, f^2 \times g^2, \cdots, f^n \times g^n)$:
\begin{lem}
  There is a natural map
  \begin{equation} \label{eq:simplicial_map_virtual_cochains}
        C^{*,c}_{\rel \partial}(\vec{f})  \otimes C^{*,c}_{\rel \partial}(\vec{g}) \to  C^{*,c}_{\rel \partial}(\vec{f} \times \vec{g}). 
      \end{equation}
      This product is associative and commutative in the sense that the next two diagrams commute.
     \begin{equation}\label{eq:associative_simplicial_map_virtual}
        \begin{tikzcd}[column sep = 10]
          C^{*,c}_{\rel \partial}( \vec{f})  \otimes \left( C^{*,c}_{\rel \partial}(\vec{g})  \otimes  C^{*,c}_{\rel \partial}(\vec{h})\right) \ar[r,"\cong"] \ar[d]  & \left( C^{*,c}_{\rel \partial}(\vec{f})  \otimes  C^{*,c}_{\rel \partial}(\vec{g})\right) \otimes C^{*,c}_{\rel \partial}( \vec{h})  \ar[d]  \\
         C^{*,c}_{\rel \partial}( \vec{f})  \otimes  C^{*,c}_{\rel \partial}(\vec{g} \times \vec{h}) \ar[d]  & C^{*,c}_{\rel \partial}(\vec{f} \times \vec{g})  \otimes  C^{*,c}_{\rel \partial}(\vec{h})  \ar[d] \\
     C^{*,c}_{\rel \partial}\left( \vec{f} \times ( \vec{g} \times \vec{h} )\right)  \ar[r,"\cong"]    &   C^{*,c}_{\rel \partial}\left( (\vec{f} \times \vec{g}) \times \vec{h}\right)
        \end{tikzcd}
      \end{equation}
      \begin{equation} \label{eq:commutative_simplicial_map_virtual}
        \begin{tikzcd}
          C^{*,c}_{\rel \partial}(\vec{f})  \otimes C^{*,c}_{\rel \partial}(\vec{g}) \ar[r] \ar[d] &  C^{*,c}_{\rel \partial}(\vec{f} \times \vec{g}) \ar[d] \\
          C^{*,c}_{\rel \partial}(\vec{g})  \otimes C^{*,c}_{\rel \partial}(\vec{f}) \ar[r] & C^{*,c}_{\rel \partial}(\vec{g} \times \vec{f})
        \end{tikzcd}
      \end{equation}
\end{lem}
\begin{proof}
  This is a straightforward generalisation of the product map on symmetric cubical chains. Let us write $f^i \co \bX_{i-1} \to \bX_i$ and $g^i \co \bY_{i-1} \to \bY_i$. A product of multi-simplices gives rise to a map
  \begin{equation}
\square^{k_1} \times \square^{\ell_1} \times \cdots \times \square^{k_n} \times \square^{\ell_n} \to N^{\vec{f} \times \vec{g}} \left(f^n \times g^n \right).  
\end{equation}
Using the identification $\square^{k} \times \square^{\ell} \cong \square^{k+\ell}$, we obtain a map
 \begin{equation}
\square^{k_1 + \ell_1} \times \cdots \times \square^{k_n + \ell_n} \to N^{\vec{f} \times \vec{g}} \left(f^n \times g^n \right).
\end{equation}

This construction is invariant under the permutation of factors, and so the result follows by checking that Condition \eqref{eq:chains_transverse_to_o-section} is preserved, and using the multiplicativity of orientation lines.
\end{proof}

We now consider a pair $\sigma^1 \co \bfn_1 \to  \Chart $ and $ \sigma^2 \co \bfn_2 \to  \Chart$ of simplices of the nerve of the category of Kuranishi charts. Their product is the functor
\begin{equation}
  \sigma^1 \times \sigma^2 \co   \bfn_1  \times \bfn_2 \to  \Chart
\end{equation}
obtained by using the monoidal structure on $\Chart$.

Given a simplex $\iota \co \bfk \to \bfn_1  \times \bfn_2 $ of the prismatic subdivision of the product $\Delta^{n_1} \times \Delta^{n_2} $, we obtain a composite functor
\begin{equation}
 \left( \sigma^1 \times \sigma^2 \right) \circ \iota \co \Bbbk \to \Chart,
\end{equation}
to which we associate the virtual cochains
\begin{equation}
     C^{*,c}_{\rel \partial}(  \left( \sigma^1 \times \sigma^2 \right) \circ \iota  ).
\end{equation}
This construction is functorial with respect to inclusions of strata of the product $\Delta^{n_1} \times \Delta^{n_2} $. This leads us to define the cochains of the product $\sigma^1 \times \sigma^2$ as a colimit
\begin{equation}
   C^{*,c}_{\rel \partial}( \sigma^1 \times \sigma^2) \equiv \colim_{\iota \co \bfk \to  \bfn_1  \times \bfn_2 }  C^{*,c}_{\rel \partial}( ( \sigma^1 \times \sigma^2 ) \circ \iota  ).
 \end{equation}
We note that we are using the ordinary colimit here (rather than a homotopy colimit). It is straightforward to see that this construction defines a functor
 \begin{equation}
    C^{*,c}_{\rel \partial} \co   \Delta \Chart \times \Delta \Chart \to \Ch.
 \end{equation}

\begin{rem}
   In the language of Appendix \ref{sec:mult-colim-via}, the functor $C^{*,c}_{\rel \partial}$ is naturally isomorphic to the functor $\left(C^{*,c}_{\rel \partial} \right)^{\RS{lr}}$, which is obtained from the virtual cochains (on simplices) as a left Kan extension in the diagram
     \begin{equation}
     \begin{tikzcd}
       \Delta^{\RS{lr} }\Chart^2 \ar[r] \ar[d] &\Delta \left( \Chart \times \Chart \right) \ar[r] &  \Delta \Chart  \ar[r]   & \Ch, \\
      \Delta  \Chart \times \Delta \Chart \ar[urrr,dashed] & & &
    \end{tikzcd}
  \end{equation}
where $ \Delta^{\RS{lr} }\Chart^2$ is the category whose objects are compositions $  \mathbf{k} \to \mathbf{n}_1 \times \mathbf{n}_2 \to \Chart \times \Chart$. 
 \end{rem}
 
\begin{lem} \label{lem:virtual_cochains_product_commutative}
  The Eilenberg-Zilber shuffle maps define a natural transformation
  \begin{equation}
         C^{*,c}_{\rel \partial} \otimes C^{*,c}_{\rel \partial} \Rightarrow C^{*,c}_{\rel \partial}
  \end{equation}
  of functors on $\Delta  \Chart \times \Delta \Chart  $, which is commutative in the sense that the following diagram commutes for each pair $\sigma^1$ and $\sigma^2$ of simplices:
  \begin{equation}
    \begin{tikzcd}
    C^{*,c}_{\rel \partial}( \sigma^1) \otimes C^{*,c}_{\rel \partial}( \sigma^2) \ar[r] \ar[d] &   C^{*,c}_{\rel \partial}( \sigma^1 \times \sigma^2) \ar[d]  \\
     C^{*,c}_{\rel \partial}( \sigma^2) \otimes C^{*,c}_{\rel \partial}( \sigma^1) \ar[r] & C^{*,c}_{\rel \partial}( \sigma^2 \times \sigma^1). 
   \end{tikzcd}
 \end{equation}
\end{lem}
\begin{proof}
Returning to the definition of the virtual cochains, the source of the desired map is 
  \begin{equation}
         \bigoplus_{\tau^1 \times \tau^2 \subset \sigma^1 \times \sigma^2}  C^{*,c}_{\rel \partial}(\tau^1 \times_{G_{\tau^1}} G_{\sigma^1_n} ) \otimes C^{*,c}_{\rel \partial}(\tau^2 \times_{G_{\tau^2}} G_{\sigma^2_n} )  [\dim \tau^1 + \dim \tau^2].
  \end{equation}
  On the other hand, the target of the map has the following explicit description
  \begin{equation}
       C^{*,c}_{\rel \partial}( \sigma^1 \times \sigma^2) \cong \bigoplus_{\rho \subset \sigma^1 \times \sigma^2}  C^{*,c}_{\rel \partial}(\rho \times_{G_{\rho}} G_{\sigma^1_{n_1} \times \sigma^2_{n_2}})[\dim \rho] .
     \end{equation}
The desired map will be given on each factor labelled by the product $\tau^1 \times \tau^2  $  by a sum of maps with target labelled by the non-degenerate simplices $\rho$ of $\tau^1 \times \tau^2$. Letting $\bfk_1$ and $\bfk_2$ denote the domains of $\tau^1$ and $\tau^2$, each such simplex is determined by a map
   \begin{equation}
       \bfk_{12} \to \bfk_1 \times \bfk_2,
     \end{equation}
     where $k_{12} = k_1 + k_2$, i.e. by a pair consisting of (possibly degenerate) simplices in the domains of $\Delta^{k_1}$ and $\Delta^{k_2}$, which we write as $\rho_1$ and $\rho_2$. The functoriality of virtual cochains (i.e. Lemma \ref{lem:functoriality_twisted_cochains}), together with the product of virtual cochains for morphisms with constant groups yield a composite map
                    \begin{align}
                      &  C^{*,c}_{\rel \partial}(\tau^1 \times_{G_{\tau^1}} G_{\sigma^1_n} ) \otimes C^{*,c}_{\rel \partial}(\tau^2 \times_{G_{\tau^2}} G_{\sigma^2_n} )  \\
                       & \to  C^{*,c}_{\rel \partial}(\rho^1 \times_{G_{\rho^1}} G_{\sigma^1_n} ) \otimes C^{*,c}_{\rel \partial}(\rho^2 \times_{G_{\tau^2}} G_{\sigma^2_n} )  \\ 
      &  \to  C^{*,c}_{\rel \partial}(\rho \times_{G_{\rho}} G_{\sigma^1_{n_1} \times \sigma^2_{n_2}}).
     \end{align}
     The desired map is the sum of these map, multiplied by the standard shuffle sign which we now describe: because we have set $k_{12} = k_1 + k_2$, and $\rho$ is non-degenerate, the repeated elements associated to the maps $\bfk_{12} \to \bfk_1$ and $\bfk_{12} \to \bfk_{2}$ are distinct. This determines a partition of $\bfk_{12} \setminus \{0\}  $ into disjoint sets with $k_1$ and $k_2$ elements, which is called a $(k_1, k_2)$-shuffle. The shuffle sign is the sign of the permutation that maps this partition into the partition of $\bfk_{12} \setminus \{0\}$ consisting of the first $k_1$ elements and their complement.

     We omit the proof of the functoriality of this construction, and note that its commutativity is a consequence of the commutativity of Diagram \eqref{eq:commutative_simplicial_map_virtual} and of the Eilenberg-Zilber map.
\end{proof}

In order to formulate the associativity of our construction, we consider a tri-simplex in $\Chart$, and define
\begin{equation}
   C^{*,c}_{\rel \partial}( \sigma^1 \times \sigma^2 \times \sigma^3) \equiv \colim_{\iota \co \bfk \to  \bfn_1  \times \bfn_2 \times \bfn_3}  C^{*,c}_{\rel \partial}( ( \sigma^1 \times \sigma^2 \times \sigma^3 ) \circ \iota  ).
 \end{equation}
This construction gives a functor on $\left(\Delta \cC\right)^3  $ and  Lemma \ref{lem:multiply_twice}, which is the key result of Appendix \ref{sec:mult-colim-via}, implies that the product which we have constructed determines natural transformations
    \begin{align}
      C^{*,c}_{\rel \partial} \otimes ( C^{*,c}_{\rel \partial} \otimes C^{*,c}_{\rel \partial}) \Rightarrow &  C^{*,c}_{\rel \partial} \\
      ( C^{*,c}_{\rel \partial}  \otimes C^{*,c}_{\rel \partial} )  \otimes  C^{*,c}_{\rel \partial} \Rightarrow & C^{*,c}_{\rel \partial}.
    \end{align}
    We omit the proof of the following result, as it is a slightly tedious imitation of the proof of associativity of the Eilenberg-Zilber product, combined with the commutativity of Diagram \eqref{eq:associative_simplicial_map_virtual}.
    \begin{lem}
      \label{lem:assocativity_virtual_cochains_product}
      The following diagram is commutative
      \begin{equation}
        \begin{tikzcd}[column sep = 10]
          C^{*,c}_{\rel \partial}( \sigma^1)  \otimes \left( C^{*,c}_{\rel \partial}(\sigma^2)  \otimes  C^{*,c}_{\rel \partial}(\sigma^3)\right) \ar[r,"\cong"] \ar[d]   & \left( C^{*,c}_{\rel \partial}(\sigma^1)  \otimes  C^{*,c}_{\rel \partial}(\sigma^2)\right) \otimes C^{*,c}_{\rel \partial}( \sigma^3)  \ar[d]  \\
         C^{*,c}_{\rel \partial}( \sigma^1)  \otimes  C^{*,c}_{\rel \partial}(\sigma^2 \times \sigma^3) \ar[d] & C^{*,c}_{\rel \partial}(\sigma^1 \times \sigma^2)  \otimes  C^{*,c}_{\rel \partial}(\sigma^3)  \ar[d] \\
      C^{*,c}_{\rel \partial}( \sigma^1 \times (\sigma^2 \times \sigma^3))  \ar[r,"\cong"]  &  C^{*,c}_{\rel \partial}( (\sigma^1 \times \sigma^2) \times \sigma^3) .
        \end{tikzcd}
      \end{equation} \qed
    \end{lem}

\subsection{Evaluation map on virtual cochains}
\label{sec:eval-map-virt}

In Lemma \ref{lem:evaluation_map_naive}, we stated the existence of a morphism in the homotopy category from the virtual cochains of naive charts to the constant functor. The goal of this section is to formulate the corresponding statement for the virtual cochains associated to arbitrary simplices in $\Chart$, i.e. to drop the condition that the map of vector bundles be an isomorphism.

The starting point is to consider, for each sequence $\vec{f}$ of morphisms of Kuranishi charts with isomorphic automorphism groups, the projections of vector bundles
\begin{equation}
  \begin{tikzcd}
       V_0 &  \left(f^1\right)^* V_1 \ar[l] & \left(f^2_0 \right)^* V_2 \ar[l] & \cdots \ar[l] & \left(f^n_0 \right)^* V_n \ar[l]
  \end{tikzcd}
\end{equation}
over $X_0$. We consider as before maps from a multi-cube to $ \left(f^n_0 \right)^* V_n $, whose projection to $ \left(f^i_0 \right)^* V_i $ factors through the projection to the $i$\th  factor, i.e. so that we have a commutative diagram
\begin{equation}
  \begin{tikzcd}
 \square^{k_0} \ar[d] & \square^{k_0}  \times  \square^{k_1} \ar[l] \ar[d] & \cdots \ar[l] \ar[d] &  \square^{k_0} \times \cdots \times \square^{k_n} \ar[d] \ar[l] \\
       V_0 &  \left(f^1\right)^* V_1 \ar[l]  & \cdots \ar[l] & \left(f^n_0 \right)^* V_n. \ar[l]
  \end{tikzcd}
\end{equation}
We define a twisted cochain complex of multi-cubes
\begin{equation}\label{eq:multi-simplex-sequence-of-vector-bundles}
  C_{-*}(\vec{V}|0; \ro_{V}^{-1})  
\end{equation}
associated to the pull back of the orientation line of $V_n$ by the rightmost vertical map, by taking the quotient by those cubes which project to the complement of the $0$-section of each of the above vector bundle, and restricting to those which satisfy the following property for each $i \in \{1, \ldots, n\}$:
\begin{equation} \label{eq:transverse_chains_positive}
  \parbox{31em}{the image  in $\left(f^i_0 \right)^* V_i $  of each stratum  $\square^{k_0} \times \cdots \times \square^{d} \subset \square^{k_0} \times \cdots \times \square^{k_i}$ lies in the complement of $ \left(f^{i-1}_0 \right)^* V_{i-1} $ if $d < \dim V_{i}/V_{i-1}$.}      
 \end{equation}
 At this stage, we note that the section $\fs_n$ defines a natural map
 \begin{equation}
     C^{*,c}_{\rel \partial}(\vec{f}) \to    C_{-*}(\vec{V}|0; \ro_{V}^{-1})  .
 \end{equation}

We would like to compose the above map with an evaluation map from the right handside to the ground ring, but this is only possible after restricting to a subcomplex. Indeed, observe that Condition \eqref{eq:transverse_chains_positive} does not include a constraint on the map $\square^{k_0} \to    V_0$. We write
\begin{equation} \label{eq:inclusion_transverse_chains_0}
 C_{-*}^{\pitchfork_0}(\vec{V}|0; \ro_{V}^{-1})   \subset  C_{-*}(\vec{V}|0; \ro_{V}^{-1})  
\end{equation} 
for the subcomplex of where we impose in addition the condition that
\begin{equation}
  \parbox{31em}{the image  in $ V_0 $  of each stratum  $ \square^{d} \subset \square^{k_0}$ lies in the complement of the zero-section if $d < \dim V_0$.}      
 \end{equation}
 \begin{lem}
   The inclusion map in Equation \eqref{eq:inclusion_transverse_chains_0} is a quasi-isomorphism, and its domain is equipped with a natural evaluation map
   \begin{equation}
      C_{-*}^{\pitchfork_0}(\vec{V}|0; \ro_{V}^{-1})  \to \Bbbk.     
   \end{equation} \qed
 \end{lem}
 The functoriality of these constructions is the same as what we considered in Section \ref{sec:homot-norm-bundl}:
 \begin{lem}
   The assignments $\vec{f} \mapsto C_{-*}(\vec{V}|0; \ro_{V}^{-1})  $ and  $\vec{f} \mapsto C_{-*}^{\pitchfork_0}(\vec{V}|0; \ro_{V}^{-1})  $ extend to simplicial chain complexes associated to each sequence of morphisms of Kuranishi charts with constant automorphism groups. The maps
   \begin{equation}
     \begin{tikzcd}
       C^{*,c}_{\rel \partial}(\vec{f}) \ar[r] & C_{-*}(\vec{V}|0; \ro_{V}^{-1})  & \ar[l]  C_{-*}^{\pitchfork_0}(\vec{V}|0; \ro_{V}^{-1})  \ar[r] & \Bbbk
     \end{tikzcd}
   \end{equation}
   are morphisms of simplicial chain complexes. \qed
 \end{lem}

 Next, we consider a simplex $  \sigma \co \bfn \to \Chart$, and we write 
 \begin{equation}
   C_{-*}^{\pitchfork_0}(\vec{V}_{\tau} \times_{G_\tau} G_{\sigma_n}|0; \ro_{V}^{-1})  \subset C_{-*}(\vec{V}_{\tau} \times_{G_\tau} G_{\sigma_n}|0; \ro_{V}^{-1})
 \end{equation}
 for the result of applying the construction of Equation \eqref{eq:multi-simplex-sequence-of-vector-bundles} to the sequence of vector bundles over $X_{\tau_0} \times_{G_{\tau_0}} G_{\sigma_n}$ obtained by pullback from $X_{\tau_i}$. We obtain chain complexes
 \begin{align}
   C_{-*}(\vec{V}_{\sigma}|0; \ro_{V}^{-1}) & \equiv \bigoplus_{\tau \subset \sigma} C_{-*}(\vec{V}_{\tau} \times_{G_\tau} G_{\sigma_n}|0; \ro_{V}^{-1})  [\dim \tau] \\
  C_{-*}^{\pitchfork_0}(\vec{V}_{\sigma}|0; \ro_{V}^{-1}) & \equiv \bigoplus_{\tau \subset \sigma}  C_{-*}^{\pitchfork_0}(\vec{V}_{\tau} \times_{G_\tau} G_{\sigma_n}|0; \ro_{V}^{-1}) [\dim \tau]
 \end{align}
that are both equipped with the differential given by the sum of the internal differential of each summand with the alternating sum of the restriction maps from each stratum of $\Delta^n$ to its boundary. These constructions are functorial on the category of simplices, and the methods of the previous section show that they admit associative and commutative multiplicative structures in the sense of Appendix \ref{sec:mult-colim-via}.  The comparison maps with the virtual cochains and the constant functor are functorial and multiplicative:
 \begin{lem} \label{lem:zig-zag-multiplicative-functors}
   There are natural transformations
   \begin{equation}
     \begin{tikzcd}
       C^{*,c}_{\rel \partial} \ar[r,Rightarrow] & C_{-*}(\vec{V}_{\_}|0; \ro_{V}^{-1})   & \ar[l,Rightarrow]  C_{-*}^{\pitchfork_0}(\vec{V}_{\_}|0; \ro_{V}^{-1})  \ar[r,Rightarrow] & \Bbbk
     \end{tikzcd}
   \end{equation}
   of commutative multiplicative functors on $\Delta \Chart $. \qed
 \end{lem}

 \subsection{Virtual counts for global quotients II}
\label{sec:virt-counts-glob-II}

We now return to the setting of Section \ref{sec:virt-counts-glob}, providing a more geometrically meaningful discussion of global charts:
\begin{defin}\label{def:global_Kuranishi} 
  The category of \emph{Kuranishi global quotients} $\Kur^{\gl}$ is the full subcategory of $\Kur$ with objects Kuranishi presentation $\bX \co A \to \Chart$ with the property that the natural induced map from $A$ to $\cQ^{op}$ is an embedding with image the non-empty boundary strata $\cQ(\bX)$. 
\end{defin}
Explicitly, this means that an object of $\Kur^{\gl}$ consists of a Kuranishi chart $\bX^{P}$ for each $P \in \cQ$ with the property that this stratum is non-empty, and a morphism of Kuranishi $\bX^{P} \to \bX^{P'}$ for each arrow in $\cQ$. These maps are required to be compatible with compositions, and the underlying map of stratifying categories are given by the induced map of undercategories in $\cQ$.

By abuse of notation, we  denote by $C^{*,c}_{\rel \partial}(\sigma) $ the virtual cochains  associated by Definition \ref{def:virtual_cochains_chart} to  the composition with $\bX$ of each simplex $\sigma \co \bfn \to A$.
\begin{defin}
  The \emph{virtual cochains of a Kuranishi global quotient} is the functor
\begin{align}
  C^{*,c}_{\rel \partial} \co  \Kur^{\gl} & \to \Ch   \\
  \bX & \mapsto \colim_{\sigma \in \Delta A}  C^{*,c}_{\rel \partial}(\sigma).
\end{align}
\end{defin}

  This definition conflicts with the one which we shall introduce later for general Kuranishi presentations, and which will consist of taking the homotopy colimit over $A$ of the colimit of $C^{*,c}_{\rel \partial}(\sigma) $ over the fibre of $\Delta A \to A$. The essential point is that, in the global quotient case, the category $A$ has a terminal element, which allows us to better control the ordinary colimit, which is the point of the next result:
  \begin{lem}
If $\bX^{\min \cQ} $ is the chart associated by the presentation $\bX$ to the minimal element, the natural map
    \begin{equation}
        C^{*,c}_{\rel \partial}(\bX^{\min \cQ}) \to    C^{*,c}_{\rel \partial}(\bX)  
      \end{equation}
      is a chain equivalence. \qed
  \end{lem}

For the next result, note that the category of Kuranishi global quotients inherits a symmetric monoidal structure from the category of Kuranishi presentations. Given a pair $\bX_1$ and $\bX_2$, we have a natural map
\begin{align}
  C^{*,c}_{\rel \partial} (\bX_1) \otimes    C^{*,c}_{\rel \partial} (\bX_2)   & \cong \colim_{\sigma_1 \in \Delta A_1}  C^{*,c}_{\rel \partial}(\sigma_1) \otimes  \colim_{\sigma_2 \in \Delta A_2}  C^{*,c}_{\rel \partial}(\sigma_2) \\
  & \to \colim_{\sigma_1 \times \sigma_2 \in \Delta A_1 \times \Delta A_2}   C^{*,c}_{\rel \partial}(\sigma_1 \times \sigma_2) \\
  & \cong \colim_{\sigma_1 \times \sigma_2 \in \Delta A_1 \times \Delta A_2}  \colim_{\tau \to \sigma_1 \times \sigma_2 }  C^{*,c}_{\rel \partial}(\tau) \\
  & \to \colim_{\tau \in  \Delta \left(A_1 \times A_2\right)}  C^{*,c}_{\rel \partial}(\tau) \\
  & \cong C^{*,c}_{\rel \partial} (\bX_1 \times \bX_2). 
\end{align}
Applying Lemma \ref{lem:virtual_cochains_product_commutative} and \ref{lem:assocativity_virtual_cochains_product}, implies that this product is associative and commutative. Moreover, Lemma \ref{lem:zig-zag-multiplicative-functors} together with Lemma \ref{lem:multiplicative-map-hocolim-functorial} relate this functor to the constant functor:
\begin{lem}
  The virtual cochains define a symmetric monoidal functor on the category of Kuranishi global quotients, which is equipped with a morphism to the constant functor, in the homotopy category of symmetric monoidal functors. \qed
\end{lem}

Next, we formulate the analogue of Lemma \ref{lem:morphism_to_constant_functor_global}:
\begin{lem}
  If $\Bbbk$ has characteristic $0$, there is a natural transformation from the stratum functor to the constant functor, in the homotopy category of symmetric monoidal functors on the category of Kuranishi global quotients. If $\Bbbk$ is an arbitrary ring, such a natural transformation exists on the subcategory of global quotients with trivial isotropy.
\end{lem}
\begin{proof}
  As in Equation \eqref{eq:compactly_supported_cochains_chart}, we define
  \begin{equation}
     C^{*,c}(\bX)   \equiv \left(   \bigoplus_{P \in \cQ}   C^{*,c}_{\rel \partial} (\partial^P \bX) \otimes \ro_{P}   , d \right),
  \end{equation}
  with differential that incorporates the map $\ro_{P'} \to \ro_{P}$ associated to each codimension $1$ inclusion. Lemma \ref{lem:absolute_cochians_have_fundamental_class} implies that the homology of this complex is supported in degree greater than or equal to $- \dim \bX$ and it is canonically isomorphic in this degree to a direct sum of copies of $\ro_{\bX}$.
  The functoriality of this construction can be formulated in terms of a contravariant symmetric monoidal functor
  \begin{align}
    \Kur^{\gl}  & \to \Ch \\
    \bX & \mapsto C^{*,c}(\bX) \otimes \ro^{-1}_{\bX}.
  \end{align}

      The inclusion into the factor corresponding to the minimal element of $\cQ$ defines a quasi-isomorphism of complexes
       \begin{equation} \label{eq:map_virtual_cochains_global_to_filtered}
           C^{*,c}_{\rel \partial} (\bX) \cong  \left(   \bigoplus_{P \in \cQ}   C^{*,c} (\partial^P \bX)   , d \right)
       \end{equation}
where the differential in the right hand side is the sum of the internal differential of $C^{*,c} (\partial^P \bX)$ with the map
       \begin{equation}
         C^{*,c}(\partial^P \bX) \to    C^{*,c}(\partial^{P'} \bX)
       \end{equation}
       for arrows $\alpha \co P \to P'$ with $\dim \kappa^{\alpha} = 1$, which is associated to the co-normal orientation of the boundary of the interval. The symmetric monoidal structure on the functor $ C^{*,c}$ induces a map
       \begin{multline}
         \bigoplus_{P_1 \in \cQ_1}   C^{*,c} (\partial^{P_1} \bX_1) \otimes   \bigoplus_{P_2 \in \cQ_2}   C^{*,c} (\partial^{P_2} \bX_2)  \to \\
         \bigoplus_{P_1 \times P_2 \in \cQ_1 \times \cQ_2}  C^{*,c} (\partial^{P_1} \bX_1) \otimes  C^{*,c} (\partial^{P_2} \bX_2) \to  \\
         \bigoplus_{P_1 \times P_2 \in \cQ_1 \times \cQ_2}  C^{*,c} \left( \partial^{P_1 \times P_2} \left(\bX_1 \times \bX_2\right) \right),
       \end{multline}
       which is associative and commutative, so that Equation \eqref{eq:map_virtual_cochains_global_to_filtered} becomes a monoidal natural transformation of symmetric monoidal functors.

       The remainder of the proof amounts to the observation that the filtration of chain complexes by degree is evidently symmetric monoidal, hence that the zig-zag constructed in Lemma \ref{lem:morphism_to_constant_functor_global} satisfies the desired properties.
\end{proof}

We now combine this result with Proposition \ref{prop:cofibrant_property_of_lift_Kur}. We define the category of \emph{Kuranishi global quotients with factorised boundary strata} $\widetilde{\Kur}^{\gl}$ to be the full subcategory of $\widetilde{\Kur}$ with objects those Kuranishi presentations with factorised strata whose image under the forgetful functor to $\Kur$ is a Kuranishi global quotient. 
\begin{thm}
  If $\Bbbk$ has characteristic $0$, the category  $\widetilde{\Kur}^{\gl}$ is equipped with a multiplicative theory of virtual counts, with coefficients in $\Bbbk$, which is moreover commutative.  If $\Bbbk$ is an arbitrary ring, such a natural transformation exists after restriction to the subcategory of quotients with trivial isotropy. \qed
\end{thm}

\section{Virtual cochains for Kuranishi presentations}
\label{sec:virt-coch-kuran-1}

We shall prove Theorem \ref{thm:virtual_counts_exist} in this section, following the strategy given in the previous section for global Kuranishi quotients. The key point is that we shall define the virtual cochains of such a Kuranishi presentation, and show that it satisfies the analogue of Lemma \ref{lem:absolute_cochians_have_fundamental_class}.

\subsection{Relative virtual cochains}
\label{sec:relat-virt-coch-1}

Let $\bX \co A \to \Chart_{\downarrow \cQ}$ be a Kuranishi presentation. Passing to the category of simplices, and using the twisted relative cochains from Definition \ref{def:virtual_cochains_chart}, we obtain a composite functor
\begin{equation}
    C^{*,c}_{\rel \partial} \co  \Delta A \to \Delta \Chart_{\downarrow \cQ} \to \Ch.
\end{equation}
Using the projection map $\Delta A \to A$, associated to the last element of a simplex, we obtain the left Kan extension
\begin{equation}
  L    C^{*,c}_{\rel \partial} \co  A \to \Ch.
\end{equation}
In our setting, this Kan extension is given by
\begin{equation}
    L    C^{*,c}_{\rel \partial} (\alpha) \cong \colim_{\substack{ \sigma \co \bfn \to A \\ \sigma_n = \alpha} } C^{*,c}_{\rel \partial}(\bX_{\sigma}).
\end{equation}
\begin{defin}
  The \emph{relative virtual cochains of $\bX$} is the homotopy colimit:
  \begin{equation}
  C^{*,c}_{\rel \partial} \bX \equiv \hocolim_{A} L C^{*,c}_{\rel \partial}.    
  \end{equation}
\end{defin}

Implicit in the above definition is the fact that we model the homotopy colimit by a bar construction, as discussed in Appendix \ref{sec:mult-colim-via}.
\begin{lem}
  The relative virtual cochains define a symmetric monoidal functor
  \begin{equation}
        C^{*,c}_{\rel \partial} \co \Kur \to \Ch.
  \end{equation}
\end{lem}
\begin{proof}
  The monoidal structure is induced by Lemma \ref{lem:virtual_cochains_product_commutative}
  which defines a multiplicative structure in the sense of Definition \ref{def:mult_structure_functor}. The commutativity and associativity of the Eilenberg-Zilber map imply that these properties hold for this multiplicative structure hence that the result follows from Lemma \ref{lem:commutative_map_hocolim} and \ref{lem:associative_map_hocolim}.
\end{proof}

Next, we observe that Lemma \ref{lem:zig-zag-multiplicative-functors}, together with Lemma \ref{lem:multiplicative-map-hocolim-functorial}, implies:
\begin{lem}
  There is a morphism from the virtual cochains to the constant functor, in the homotopy category of symmetric monoidal functors on $\Kur $. \qed
\end{lem}

\subsection{Virtual cochains and virtual counts}
\label{sec:absol-virt-coch}

We continue to follow the strategies developed in Section \ref{sec:virt-counts-glob} and \ref{sec:virt-counts-glob-II} for the construction of a theory of virtual counts. The first step is to assign to each Kuranishi presentation the absolute cochains
\begin{equation} \label{eq:absolute_virtual_cochains-presentatoin}
    C^{*,c}(\bX) \equiv  \left(   \bigoplus_{P \in \cQ}   C^{*,c}_{\rel \partial} (\partial^P \bX) \otimes \ro_{P}   , d \right),
\end{equation}
with differential as in Equation \eqref{eq:compactly_supported_cochains_chart}. We shall postpone the proof of the following result, which is the analogue of Lemma \ref{lem:absolute_cochians_have_fundamental_class} to the next section; the proof is more involved:
\begin{lem} \label{lem:absolute_cochians_have_fundamental_class-local-to-global}
 The cohomology $H^{*,c}(\bX)$ of a $\Bbbk$-oriented Kuranishi presentation is supported in degrees greater than or equal to $- \dim \bX$ and it is canonically isomorphic to a direct sum of copies of $\ro_{\bX}$, indexed by the components of $Z/G$, in this degree, whenever either of the following conditions holds:
  \begin{enumerate}
  \item the action of $G$ on $Z$ is free, or
    \item the ring $\Bbbk$ is a field of characteristic $0$,
    \end{enumerate}
\end{lem}

Given the above result, the proof of the main result of this paper is now a reiteration of the techniques developed in the previous section:
\begin{proof}[Proof of Theorem \ref{thm:virtual_counts_exist}]
 By Proposition \ref{prop:cofibrant_property_of_lift_Kur}, it suffices to construct a morphism in the homotopy category of symmetric monoidal functors; this is done in exactly the same way as in Lemma \ref{lem:morphism_to_constant_functor_global}:

  We introduce the complex
  \begin{equation}
 \left( \bigoplus_{P \in \cQ}   C^{*,c}( \partial^P \bX), d \right)
\end{equation}
with differential given by the projection maps from the cochains associated to a stratum to those associated to its codimension $1$ strata (this map has degree $1$ because it entails trivialising the conormal orientation of such a stratum). This is a symmetric monoidal functor on $\Kur$, equipped with a monoidal equivalence
\begin{equation}
    C^{*,c}_{\rel \partial}(\bX) \to  \left( \bigoplus_{P \in \cQ}   C^{*,c}( \partial^P \bX), d \right).
\end{equation}
Lemma \ref{lem:absolute_cochians_have_fundamental_class} then yields, by truncation, a monoidal zig-zag of natural transformations from $C^*_{st}$ to the right hand side.
\end{proof}

\subsection{Virtual cochains with support}
\label{sec:virt-coch-with}

Our goal in the remaining part of Section \ref{sec:virt-coch-kuran-1} is to prove Lemma \ref{lem:absolute_cochians_have_fundamental_class}. We note that, as we formulated it, this is a statement about a single Kuranishi presentation, and we do not have to make sure that the way that we prove this result is in any sense functorial or multiplicative. As will be clear below, Condition \eqref{eq:contractible_nerve_presentation} will play an essential role.

Given a Kuranishi presentation $\bX$ of $\cM$, with domain $A$, and a subset $S$ of $\cM$, we define for each object $\alpha \in A$ the subset
\begin{equation}
  Z_{\alpha}(S) \subset Z_{\alpha}  
\end{equation}
to be the inverse image of $S$ under the projection to $\cM$. Given a sequence $\vec{f}$ of composable morphisms in $A$ whose underlying maps of groups are isomorphisms, we define $N f^i(S)$ to be the bundle over $X_{i-1}$, whose fibre at a point $x$ is
\begin{equation}
\partial^{P_{f_i}}  X_{i} \setminus \left( X_{i-1} \cup Z_i(S)\right) \cup \{x\}  \subset \partial^{P_{f_i}} X_i. 
\end{equation}
As in Section \ref{sec:homot-norm-bundl}, by iteratively pulling back $Nf^{i+1}(S)$ under the evaluation map
\begin{equation}
    N f^i(S) \to X_i,
\end{equation}
we obtain an associated sequence
\begin{equation}
  \begin{tikzcd}
      X_0  &  N^{\vec{f}} f^1(S) \ar[l] &   N^{\vec{f}} f^2(S) \ar[l] & \cdots \ar[l] & N^{\vec{f}} f^n(S) \ar[l]
  \end{tikzcd}
\end{equation}
of fibrations. We have the following variant of Definition \ref{def:twisted_vcochains-constant-G}:
\begin{defin}
  The \emph{relative virtual cochains with support in $S$} of a sequence $\vec{f} = \{f^i\}$ of composable morphisms of Kuranishi charts with constant automorphism group $G$ is the co-chain complex
  \begin{multline}
    C^{-*,c}_{\rel \partial,S}(\vec{f}) \equiv \\ C^{\pitchfork}_*(X_0|Z_0; C_*(N^{\vec{f}}f^1(S)|0) \otimes \cdots \otimes C_*(N^{\vec{f}}f^n(S)|0)  \otimes \ro_{V_n}^{-1}) \otimes_{C_* G} C_* (EG),
  \end{multline}
of symmetric multi-cubical chains, generated by the pullback of $ \ro_{V_n}^{-1}$ under maps
  \begin{equation}
    \square^{k_0} \times \cdots \times \square^{k_n} \to N^{\vec{f}} f^n(S) ,  
  \end{equation}
  such that the projection to $N^{\vec{f}} f^i(S)$ factors through the projection of the domain to $\square^{k_0} \times \cdots \times \square^{k_i}$, modulo (i) degenerate cubes, (ii) the ideal generated by a cube and its image under a permutation of any of the factors, and (iii) those cubes whose projection to $N^{\vec{f}} f^i(S)$ lies in the complement of the origin (the case $i=0$ is implicitly included in these conditions if we use the convention that $f^0=\id_{\bX_0}$).   We assume as well that each generator satisfies the following property:
  \begin{equation}
    \label{eq:chains_transverse_to_o-section-support-S}
    \parbox{30em}{the image  in $Nf^i(S)$  of each stratum  $\square^{k_0} \times \cdots \times \square^{d} \subset \square^{k_0} \times \cdots \times \square^{k_i}$ lies in the complement of the zero section if $d < \dim V_{i}/V_{i-1}$.}
  \end{equation}
\end{defin}

Considering a general simplex $\sigma$ in $A$, and returning to the construction of Section \ref{sec:homot-norm-bundl-2}, we then obtain a cochain complex
 \begin{equation}
 C^{*,c}_{\rel \partial,S}(\sigma) \equiv \bigoplus_{\tau \subset \sigma}  C^{*,c}_{\rel \partial,S}(\tau \times_{G_{\tau}} G_{\sigma_n} )[\dim \tau],
  \end{equation}
by imposing in Definition  \ref{def:virtual_cochains_chart} the support condition in Equation \eqref{eq:virtual_cochain_simplex}. This leads us to introduce the absolute cochains
  \begin{equation}
       C^{*,c}_{S}(\sigma) =   \left(   \bigoplus_{P \in \cQ}   C^{*,c}_{\rel \partial} (\partial^P \sigma) \otimes \ro_{P}   , d \right).
  \end{equation}

This construction is contravariantly functorial in $S$, in the sense that an inclusion $S' \subset S$ induces a restriction map
\begin{equation}
    C^{*,c}_{S}(\sigma) \to   C^{*,c}_{S'}(\sigma). 
\end{equation}

This construction is also functorial with respect to simplices in $A$, and hence descends to Kuranishi presentations:
\begin{defin} \label{def:rel_compact_virt_support}
  The \emph{compactly supported virtual cochain of a Kuranishi presentation $\bX$ with support in $S$} is the complex
  \begin{equation}
        C^{*,c}_{S}(\bX) \equiv \hocolim_{ \alpha \in A}  \left(LC^{*,c}_{S}\right)(\alpha).
      \end{equation}
\end{defin}
Having passed from simplices to presentations, the construction remains functorial with respect to inclusions $S' \subset S$, in the sense that they induce maps
\begin{equation}
  C^{*,c}_{S}(\bX) \to   C^{*,c}_{S'}(\bX)
\end{equation}
For the terminal element $S = \cM$, we have
\begin{equation}
  C^{*,c}_{\cM}(\bX) \cong C^{*,c}(\bX).
 \end{equation}

   \subsection{Computing the cohomology of Kuranishi presentations}
\label{sec:constr-non-mult}

We now specialise the above discussion to the case of compact subsets. We use Pardon's terminology from \cite{Pardon2016}
\begin{lem}
  The assignment of the cochain complex $ C^{*,c}_{K}(\bX) $ to each compact subset $K$ of $\cM$  defines a homotopy $\cK$-sheaf on $\cM$.
\end{lem}
\begin{proof}
  The result entails checking three properties: (i) $C^{*,c}_{\emptyset}(\bX) $ is acyclic because $ C^{*,c}_{\emptyset}(\sigma)$ is the trivial chain complex for every simplex, (ii) given a cover $K_1 \cup K_2$ of $K$, the Mayer-Vietoris property for (relative) singular chains implies that the total complex
  \begin{equation}
    C^{*,c}_{K}(\bX) \to C^{*,c}_{K_1}(\bX) \oplus  C^{*,c}_{K_2}(\bX) \to C^{*,c}_{K_1 \cap K_2}(\bX)
  \end{equation}
  is acyclic, (iii) the fact that singular chains are compactly supported implies that the map
  \begin{equation}
    \colim_{K \Subset K'}    C^{*,c}_{K'}(\bX) \to   C^{*,c}_{K}(\bX)   
  \end{equation}
  is an isomorphism of chain complexes.
\end{proof}

Next, we check that the virtual cochains define a pure homotopy sheaf. This entails computing the stalks at a point:
 \begin{lem} \label{lem:compute_virtual-cochains-chart}
   If $\sigma$ is a simplex of charts in $A$ whose footprints include a point $z$, the cohomology of $ H^{*,c}_{\{z\}}(\sigma) $ is naturally isomorphic to $\ro_{\bX}$ whenever either of the following conditions hold:
  \begin{enumerate}
  \item the action of $G$ on $Z$ is free, or
    \item the ring $\Bbbk$ is a field of characteristic $0$,
    \end{enumerate} 
  \end{lem}
  \begin{proof}
Let $K=\{z\}$. Recall that $X_{\sigma}$ is a manifold of dimension $\dim \bX_{\sigma} + \dim V_\sigma + \dim G_\sigma $, equipped with an action of $G_\sigma$ with finite isotropy. This implies that the $G_\sigma$-equivariant relative homology of the pair $X_\sigma| Z_\sigma(K)$ is supported in degree less than or equal to $\dim X_\sigma + \dim V_\sigma$. Twisting by the bundle $V_\sigma$, and using the fact that $X_\sigma$ is oriented relative this bundle implies, by Lemma \ref{lem:absolute_cochians_have_fundamental_class}, that the cohomology of $ C^{*}_{K}(\sigma)$ in degree $\dim \bX_\sigma$ is isomorphic to $\ro_{\bX}$, and is trivial in lower degrees.
\end{proof}
Since $ H^{*,c}_{\{z\}}(\sigma)$ vanishes for simplices of charts one of which fails to include $\{z\}$, and since we have assumed that the nerve of the category of charts covering $\{z\}$ is contractible, we conclude:
\begin{cor}  \label{cor:purity}
Under the hypotheses of Lemma \ref{lem:compute_virtual-cochains-chart}, the stalks $ H^{*,c}_{\{z\}}(\bX)$ define a local system on $\cM$ which is naturally isomorphic to $ \ro_{\bX}$. \qed
\end{cor}

We need to check one more technical condition, which is referred to in \cite{Pardon2016} as the \emph{weak vanishing property}. Our proof will use the fact that $\cM$ is locally a finite quotient of a closed subset of a Euclidean space, hence has finite Lebesgue covering dimension:
\begin{lem} \label{lem:weak_vanishing}
If $\cF$ is a $\cK$-sheaf whose stalks form a local system, then  the cohomology of $\cF(K) $ is bounded below for each compact subset $K$ of $\cM$.
\end{lem}
\begin{proof}
  By exactness, it suffices to prove this result whenever $K$ is sufficiently small. We thus pick a point $z \in \cM$, and choose $K$ so that the stalk at $z$ lifts to $K$. This reduces the problem to the case of a $\cK$-presheaf on $K$ with trivial stalks. We shall prove that such a sheaf identically vanishes (a stronger result than asserted).

  Assume, by contradiction, that $\cF(K)$ is not acyclic, and fix a nontrivial element $\phi$ of the cohomology of $\cF(K)$. Recall that each cover $\{K_i\}$ of $K$ determines a spectral sequence, converging to $0$, associated to the filtration of the extended \v{C}ech complex
  \begin{equation}
        \cF(K) \to \bigoplus_{i} \cF(K_i) \to \bigoplus_{i < j} \cF(K_{ij}) \to \cdots
  \end{equation}
  This is a right plane spectral sequence, with subspaces of $\cF(K)$ sitting along the edge (vertical axis). We define the depth of $x$ with respect to a cover to be the maximal $r$ so that $\phi$ survives to the $E^r$ page but no further (i.e. so that all previous differentials vanish on $\phi$, and $d^r$ is non-zero).

  By the functoriality of the spectral sequence associated to a filtration, we see that the depth can only decrease under a refinement $\{K'_j\}$ of $\{K_i\}$ (i.e. if a differential vanishes on the $r$\th page for a cover, it also vanishes for every refinement). This implies that depth of $\phi$ is no larger than the covering dimension of $K$.

  On the other hand if the restriction maps
  \begin{equation}
    \cF(K_I) \to \cF(K'_J)    
  \end{equation}
are trivial on the image of $d^r$, then the depth of must strictly decrease. The vanishing of the stalks implies that we can always find a refinement for which these restrictions vanish, which contradicts the bound on the depth.
\end{proof}
\begin{rem}
  As pointed out to the author by Pardon, the above result is a consequence of a general result of Lurie about sheaves of spaces \cite[7.2.3.6 and 7.2.1.17]{Lurie2009}, and the enrichment of the category of chain complexes over spaces.
\end{rem}
We complete this section by establishing the last step in the proof of Theorem \ref{thm:virtual_counts_exist}, which was postponed from Section \ref{sec:absol-virt-coch}:
\begin{proof}[Proof of Lemma \ref{lem:absolute_cochians_have_fundamental_class-local-to-global}]
Corollary \ref{cor:purity} and Lemma \ref{lem:weak_vanishing} allows us to appeal to Pardon's result \cite[Proposition A.5.4]{Pardon2016} which implies that we have a natural isomorphism
\begin{equation}
  H^{*,c}(\bX) \cong \check{H}^*(\cM; \ro_\bX).
\end{equation}

\end{proof}

\section{Semipositive Kuranishi presentations}
\label{sec:monot-kuran-pres}

In this section, we adapt the theory developed in the text to the (much simpler) case of semipositive Kuranishi presentations, i.e. to the subcategory
\begin{equation}
  \Kur_{+} \subset \Kur  
\end{equation}
consisting of Kuranishi presentations all of whose components have non-negative virtual dimension (i.e. the dimension at each point of $X$ is larger than the sum of the dimension of $G$ and the rank of $V$), and have the additional property that those components of virtual dimension $0$ and $1$ have trivial stabilisers. We write $\Kur^{\Bbbk}_{+}$ for the subcategory of $\Kur^{\Bbbk}$ consisting of presentations lying over $\Kur_+$, and $\widetilde{\Kur}^{\Bbbk}_+$ for the lift to the category of Kuranishi presentation with factorised strata $\widetilde{\Kur}^{\Bbbk}$. We being by noting that the monoidal structures on $\Kur$, $\Kur^{\Bbbk}$, and $\widetilde{\Kur}^{\Bbbk}$ preserve $\Kur_{+}$, $\Kur^{\Bbbk}_{+}$, and $\widetilde{\Kur}^{\Bbbk}_+$.

As alluded to in the introduction, despite the fact that Floer theory only uses moduli spaces of virtual dimension $0$ and $1$, one cannot work out a general theory of virtual counts without considering moduli spaces of arbitrary virtual dimension. The essential difficulty is that the product of a Kuranishi presentation of positive virtual dimension and one of negative virtual may arise as a boundary of a presentation of virtual dimension $0$ (or $1$), so that there is no way to truncate the stratum functor on $\Kur^{\Bbbk} $ while preserving its monoidal structure.

In the semi-positive case, however, we can define a (lax) symmetric monoidal functor
\begin{equation}
    C^*_{st,+} \co \Kur^{\Bbbk}_+ \to \Bbbk 
\end{equation}
given by
\begin{equation}
C^*_{st,+}(\bX) \equiv
\begin{cases}
  C^*_{st}(\bX) & \textrm{ if } \dim \bX = 0 , 1 \\
  0 & \textrm{ otherwise.}
\end{cases}
\end{equation}
The proof of the next result is a simplified version of the proof of Proposition \ref{prop:cofibrant_property_of_lift_Kur}, as the notion of depth can be replaced by the dimension.
\begin{prop}
  \label{prop:cofibrant_property_of_lift_Kur+}
  Given a monoidal weak equivalence $\pi \co G \Rightarrow F$ of symmetric monoidal functors from $\widetilde{\Kur}^{\Bbbk}_+$ to the category of chain complexes, which factor through $\Kur^{\Bbbk}_+ $, and a monoidal natural transformation $\eta \co  C^*_{st,+} \Rightarrow F$, there exists a symmetric monoidal natural transformation $C^*_{st,+} \Rightarrow F $ so that the following homology-level diagram commutes:
  \begin{equation}
    \begin{tikzcd}
      & H^* G \ar[d,Rightarrow] \\
      H^*_{st,+} \ar[ur,Rightarrow] \ar[r,Rightarrow] & H^* F.
    \end{tikzcd}
  \end{equation} \qed
\end{prop}

We formulate the notion of a semipositive theory of virtual counts in the same way: we define the symmetric monoidal functor
\begin{equation}
  \Bbbk_+ \co  \Kur^{\Bbbk}_+ \to \Bbbk  
\end{equation}
to be the constant functor with value $\Bbbk$ on Kuranishi presentations of dimension $0$ and $1$, and $0$ on presentations of dimension greater than or equal to $2$. 

\begin{defin}
  A \emph{multiplicative theory of semi-positive virtual counts with coefficients in $\Bbbk$} is a monoidal natural transformation from $C^*_{st,+}$ to the functor $\Bbbk_+$, considered as symmetric monoidal functors from  $\widetilde{\Kur}^{\Bbbk}_+$ to $\Ch_{\Bbbk}$:
   \begin{equation} 
\begin{tikzcd}[column sep=huge]
\widetilde{\Kur}^{\Bbbk}_+
  \arrow[bend left=50]{r}[name=U,label=above:$C^*_{st}$]{}
  \arrow[bend right=50]{r}[name=D,label=below:$\Bbbk$]{} &
\Ch.
\arrow[from=U.south-|D,to=D,Rightarrow,shorten=5pt]{r}{\cV}
\end{tikzcd}   
   \end{equation}
\end{defin}

To construct such a theory, we proceed by considering the semipositive virtual cochains
\begin{equation}
    C^{-*,c}_{\rel \partial,+}(\bX) \equiv
    \begin{cases}
      C^{-*,c}_{\rel \partial}(\bX) & \textrm{ if } \dim \bX = 0,1 \\
      0 & \textrm{otherwise.}
    \end{cases}
\end{equation}

Following exactly the same strategy as in the proof of Theorem \ref{thm:virtual_counts_exist}, we build morphisms from the semipositive stratum functor to $C^{-*,c}_{\rel \partial,+} $, and thence to $\Bbbk_+$. Applying Lemma \ref{prop:cofibrant_property_of_lift_Kur+}, we conclude:
\begin{thm}
  A multiplicative theory of semipositive virtual counts exists for any ring $\Bbbk$. \qed
\end{thm}

\part{From virtual counts to algebraic structures}
\label{part:from-virtual-counts}

\section{Extending the category of Kuranishi presentations}
\label{sec:extend-categ-kuran}

\subsection{Diagrams of Kuranishi presentations}
\label{sec:diagr-kuran-pres}
We fix a ground ring $\Bbbk$, an integer $d \in \bN$, and write $\Kur$ for the category of Kuranishi presentations 
which are relatively $\Bbbk$-oriented and have constant dimension modulo $d$ (i.e. we omit the subscript from the notation). Recall that an object of this category consists in part of a category $\cQ$ which is a model for manifolds with generalised corners, which includes the assumption that it has an initial element. 

We shall need a generalisation of this notion: the basic idea is that we want to drop the condition that $\cQ$ have an initial element. Geometrically, the minimal element of $\cQ$ corresponds to the top stratum, and hence dropping this condition amounts to considering spaces which are stratified in such a way that there is more than one top-dimensional stratum.
\begin{rem}
  There are two reasons for considering this generalisation. Most importantly for this paper, the constructions of Section \ref{sec:flow-categ-chain} seem difficult to implement entirely in the context of Kuranishi presentations: the problem is that, after introducing our notion of flow categories, we shall define a notion of flow bimodule, and we will want to define the tensor product of flow bimodules. Unfortunately, even if we start with Kuranishi presentations, the tensor product doesn't assign to a pair of objects a bimodule built from single Kuranishi presentations, but a diagram thereof.

 The second reason for this generalisation arises in applications: when coupling Floer and Morse theory by constructing algebraic structures on Morse complexes which include holomorphic curve corrections, one tends not to encounter a single Kuranishi presentation associated to each operation, but rather a collection indexed by combinatorial data consisting of configurations of discs and gradient flow lines. The formalism we introduce is designed to apply in this setting.
\end{rem}

To proceed with our construction, we consider a more general class of categories than in the definition of Kuranishi presentations, and which will index the relevant class of Kuranishi diagrams which we will later introduce:
\begin{defin}  
  A category $\cP$ equipped with a functor $\codim \cP \to \bN$ is a \emph{model for partitioned manifolds with generalised corners} if (i) for each element $P$ of $\cP$, the geometric realisation of the over category $\cP^{P}$ of $P$ is homeomorphic to a ball of dimension $\codim P$, and (ii) for each arrow $\alpha$, the category $\cP^{\alpha} \setminus \{ P_0, P_1\}$ is a partially ordered set  whose geometric realisation
  is a sphere of dimension
  \begin{equation}
    \codim P_1 - \codim P_0 - 2.
  \end{equation}
\end{defin}
\begin{rem}
  Since the results of this paper are not concerned with Floer homotopical constructions, we expect that only the conditions for codimension $1$ strata are in fact needed.
\end{rem}
We denote the geometric realisation of $\cP^{P}$ by $\kappa^P$ as in Section \ref{sec:stratified-spaces}, and observe that there are two cases to consider: the point of $\kappa^P$ associated to the terminal element $P$ of $\cP^P$ lies either in the interior or in the boundary. We say that $P$ is an interior element in the first case, and a boundary element in the second; these correspond to the condition that the geometric realisation of $\cP^P \setminus \{ P\} $ be a sphere or a ball.   If $\cP$ only has interior elements we say that it is a \emph{model for partitioned manifolds without boundary.}

\begin{lem}
The subcategory $\partial \cP$ of $\cP$ consisting of boundary elements is downward closed, and is a model for partitioned manifolds without boundary.
\end{lem}
\begin{proof}[Sketch of proof]
Given an arrow $\alpha \co P \to P'$, we have an embedding
  \begin{equation} \label{eq:open_embedding_arrow_in_cP}
    |\cP^{P} \setminus \{ P\}| \times |\cP^{\alpha} \setminus \{ P'\}| \to |\cP^{P'} \setminus \{ P'\}|,
  \end{equation}
  which by the existence of the codimension functor is open in a neighbourhood of $ |\cP^{P} \setminus \{ P\}| \times \{P'\}$. By assumption, $|\cP^{\alpha} \setminus \{ P\}| $ is homeomorphic to a ball, with $P'$ mapping to the origin.
  
  For the first part, we assume that $P$ is a boundary element, and hence that $|\cP^P \setminus \{ P\}|$ is a ball.  The source of Equation \eqref{eq:open_embedding_arrow_in_cP} is then a closed ball, to that the right hand side cannot be a closed manifold, which implies that $P'$ is a boundary point as well.

  For the second part, we use Equation \eqref{eq:open_embedding_arrow_in_cP} to show that the embedding
  \begin{equation}
    |\left(\partial \cP\right)^{P'} \setminus \{ P'\}| \to |\cP^{P'} \setminus \{ P'\}| 
  \end{equation}
  is the inclusion of the boundary.   
\end{proof}

In practice, the Kuranishi  diagrams that we consider  are often parametrised by certain manifolds with corners, but with the property that the diagram has a more refined partition than the manifold parametrising it. Instead of developing the theory of such parametrisation at the level of spaces, it shall be sufficient for our purposes to record the corresponding data at the level of the categories labelling the partition. For the next definition, we recall that the right fibre of a functor $\pi \co \cQ \to \cP$ over an object $P$ of $\cP$ is the category whose objects are pairs $(Q, \beta)$, with $Q$ an object of $\cQ$ and $\beta$ a morphism in $\cP$ from $P$ to $\pi(Q)$. We shall denote this category by $\partial^P \cQ$. We shall also need an analogous construction for arrows: given a morphism $\alpha$ in $\cP$, we define $\partial^{\alpha} \cQ$ to consists of pairs $(\alpha',\beta)$, with $\alpha'$ a morphism in $\cP$ and $\beta$ a morphism in $\cQ$, so that $\pi (\alpha') \circ \beta = \alpha$.
\begin{defin} \label{def:refinement-poset}
  A functor $\pi \co \cQ \to \cP$ of categories which model manifolds with partitions is a \emph{refinement} if the following properties hold for each object $P$ of $\cP$:
  \begin{enumerate} 
    \item  the minimal codimension of an object of $\cQ$ lying over $P$ is $\codim P$.
    \item the right fibre $\partial^P \cQ$ is a model for partitioned manifolds,
      \item \label{item:preserve_boundary} the natural functor from the colimit over $\alpha \in \partial^P \cP \setminus \{ P \}$ of the categories $\partial^{\alpha} \cQ$, to $\partial^P \cQ$ is an isomorphism onto the boundary.
     \item \label{item:degree_1} If $\codim Q = \codim \pi Q$, the natural map $\cQ^{Q} \to \cP^{\pi Q}$ induces a degree $1$ map on geometric realisations. 
  \end{enumerate}
\end{defin}
Note that the last condition makes sense because the first condition implies that $\cQ^{Q}$ and $\cP^{\pi Q}$ are balls of equal dimension, and that the map takes boundary to boundary.

The most important case of the above definition is the case where $P \in \cP$ is a boundary element of codimension $1$: the category $\partial^{P} \cQ$ can then be thought of as assembling the codimension $1$ strata indexed by $Q \in \pi^{-1} (P)$ into a single boundary stratum.

Having introduced the requisite notions at the level of stratifying categories, we now consider the desired class of diagrams of Kuranishi presentations:
\begin{defin}
  A \emph{Kuranishi diagram} consists of the following data:
  \begin{enumerate}
  \item a refinement $\pi \co \cQ \to \cP$ of categories modelling manifolds with partitions,
  \item a functor $\bX \co \cQ^{op} \to \Kur$,
  \item for each  $Q \in \cQ$, an isomorphism of the undercategory $\partial^Q \cQ$ with the stratifying category $\cQ_{\bX^Q}$ of the Kuranishi presentation  $\bX^Q$, and
    \item a graded line $\ro_{\bX}$.
  \end{enumerate}
  which satisfy the following properties:
  \begin{enumerate}
  \item The Kuranishi presentation $\bX^{Q}$ is oriented relative $\ro^{-1}_{Q} \otimes \ro_{\bX} $ (here, $\ro_Q$ is the orientation line of the ball $\kappa^Q$).
\item The morphism of Kuranishi presentations induced by an arrow $\alpha \co Q \to Q'$ is given at the level of stratifying sets by the induced map $\partial^{Q'} \cQ \to \partial^Q \cQ$, and at the level of orientation line by the isomorphism
  \begin{equation}
    \ro_{\alpha} \otimes   \ro^{-1}_{Q'} \cong \ro^{-1}_{Q}, 
  \end{equation}
  where $ \ro_{\alpha}  $ is the orientation line of the cell associated to $\cQ^{\alpha}$.
  \item  The presentation $\bX^Q$ is empty for all but finitely many objects $Q$ of $\cQ$.
\end{enumerate}

\end{defin}
Unpacking the data, and omitting the orientation data, we observe that each diagram induces an assignment of a category $A_{Q}$ and a functor
\begin{equation}
  A_{Q} \to \Chart_{\downarrow \partial^Q \cQ},
\end{equation}
 to each object $Q$ of $\cQ$, and to each arrow $Q \to Q'$, of a functor $A_{Q'} \to A_{Q}$ and of a natural transformation in the diagram 
\begin{equation} \label{eq:iso_strata_Kuranishi_diagram}
  \begin{tikzcd}
    A_{Q'} \ar[r] \ar[d] & \Chart_{\downarrow \partial^{Q'} \cQ }\ar[d,""{name=fromhere}] \\
    A_{Q} \ar[r,""{name=tohere}] &  \Chart_{\downarrow \partial^{Q} \cQ}.
      \arrow[Rightarrow,from=fromhere, to=tohere, start anchor={east},
    end anchor={south}, shorten >=3pt, shorten <=6pt,bend right=30]
  \end{tikzcd}
\end{equation}
These data are then required to be compatible with compositions in $\cQ$.
 
We define a morphism of Kuranishi diagrams to consist of a commutative diagram
\begin{equation}
  \begin{tikzcd}
    \cQ \ar[r] \ar[d] \ar[rr,bend left=60,""{name=fromhere}]  & \cQ' \ar[r,""{name=tohere}] \ar[d] & \Kur \\
    \cP \ar[r] &  \cP' &
     \arrow[Rightarrow,from=fromhere, to=tohere, start anchor={west},
    end anchor={south}, shorten >=3pt, shorten <=6pt]
  \end{tikzcd}
\end{equation}
in which the map $\cQ \to \cQ'$ is an embedding inducing an isomorphisms of the categories associated to each arrow in $\cQ$ as in Definition \ref{def:model_for_corners}.  We denote the category of such diagrams by $\cD \Kur$.
\begin{rem}
At this stage, we do not impose any assumption on the map $\cP \to \cP'$, as we will be providing an explicit map in all constructions that use it. We expect that, in general, one should assume that this image is provided with a factorisation as the composition of a refinement, and an embedding inducing an isomorphisms of arrow categories.
\end{rem}
\begin{lem}
  \label{lem:dKur_has_tensor_and_coproduct}
  The category $\cD \Kur$ is symmetric monoidal, and admits finite coproducts. \qed
\end{lem}

Recalling that $\widetilde{\Kur}$ is the category of Kuranishi diagrams with factorised corner strata, we define $\cD \widetilde{\Kur}$ to be the category whose objects and morphisms are lifts of Kuranishi diagrams to $\widetilde{\Kur}$. 
This category is again symmetric monoidal.

\begin{rem}
  The Grothendieck construction provides an alternate way of thinking about Kuranishi diagrams: we can assemble all the categories $A_Q$ indexing the Kuranishi charts associated to objects of $\cQ$ into a single category $A$ equipped with a functor to $\cQ$, and a functor to the category of charts. One can then formulate the notion of Kuranishi diagram directly in terms of properties of these two functors on $A$, without appealing to Kuranishi presentations as an intermediate notion.
\end{rem}

\subsection{Virtual counts for Kuranishi diagrams}
\label{sec:virt-counts-kuran}

Given a Kuranishi diagram $\bX$, we define
\begin{equation} \label{eq:relative_cochains_diagram}
  C^{*}_{st}(\bX) \equiv \bigoplus_{\substack{P \subset \cP_{\bX}(\bX)}} \ro_{\partial^P \bX},
\end{equation}
where $\cP_{\bX}(\bX) $ is the subcategory of $\cP_{\bX}$ consisting of non-empty strata, and $\ro_{\partial^P \bX} $ is the tensor product of $\ro_{\bX} $ with $\ro^{-1}_{P}$. This complex is equipped with a differential induced by restriction to codimension $1$ strata, and is a symmetric monoidal functor in the same way as in Section \ref{sec:chain-functor-strata}.
\begin{lem} \label{lem:virtual_counts_diagrams}
  A theory for virtual counts determines a monoidal natural transformation  \begin{equation} \label{eq:virtual_counts_diagram}
\begin{tikzcd}[column sep=huge]
\cD \widetilde{\Kur}
  \arrow[bend left=50]{r}[name=U,label=above:$C^*_{st}$]{}
  \arrow[bend right=50]{r}[name=D,label=below:$\Bbbk$]{} &
\Ch.
\arrow[from=U.south-|D,to=D,Rightarrow,shorten=5pt]{r}{\cV}
\end{tikzcd}   
   \end{equation} 
 \end{lem}
 \begin{proof}
   A theory of virtual counts determines a monoidal natural transformation from the functor given by
   \begin{equation}
         \bigoplus_{\substack{Q \in \cQ_{\bX}(\bX)}} \ro_{\partial^Q \bX}, 
   \end{equation}
   to the constant functor. Consider the map 
   \begin{equation}
       \ro_{\partial^P \bX} \to    \bigoplus_{\substack{\pi Q = P \\ \codim Q = \codim P }}  \ro_{\partial^Q \bX},
       \end{equation}
       which on each factor is given by the isomorphisms in Part \ref{item:degree_1} of Definition \ref{def:refinement-poset}. Part \ref{item:preserve_boundary} shows that this is a cochain map,  yielding the desired natural transformation.
 \end{proof}
\subsection{Kuranishi presentations indexed by monoids}
\label{sec:kuran-pres-index}

Let $\Gamma$ be an abelian monoid equipped with a homomorphism
\begin{equation}
  \cA \co \Gamma \to \bR,
\end{equation}
which we refer to as the action homomorphism. The canonical example is $\Gamma = \bR$, and the reader may want to perform that substitution throughout. We denote by $\Gamma_+$ the submonoid which is the inverse image of the non-negative real numbers.

We introduce the category ${\cD \Kur}^{\Gamma}$ whose objects consist of a graded line $\ro_{\bX}$, and a collection of objects of $\cD {\Kur}$ indexed by elements of $\Gamma$, of constant dimension modulo $d$, which are oriented relative to $\ro_{\bX}$, and which are required to be empty except for a set which is proper over $\bR$, and whose image is bounded below. This is equivalent to requiring that the number of elements mapping to any interval $(-\infty,E]$ be finite. A morphism consists as before of an isomorphism of graded lines, together with morphisms of Kuranishi diagrams for each element of $\Gamma$. This is again a symmetric monoidal category, with the product of objects $\{ \bX^{\lambda} \}_{\lambda \in \Gamma}$ and $\{ \bY^{\rho}\}_{\rho \in \Gamma}$ given by
\begin{equation}
  \left( \bX \times \bY \right)^{\mu} = \coprod_{\rho + \lambda = \mu} \bX^{\lambda} \times \bY^{\rho}.
\end{equation}
Note that this coproduct is finite, and hence well-defined, because we have assumed that the objects of ${\cD \Kur}^{\Gamma}$ are proper over $\bR$ and have bounded below action.  Note that we have a natural fully faithful embedding
\begin{equation}
    {\cD \Kur}^{\Gamma_+} \to {\cD \Kur}^{\Gamma}
\end{equation}
given by objects for which the presentations for elements of negative action is empty.

It is convenient for later purposes to record a category lying under ${\cD \Kur}^{\Gamma}$: let 
$\Cat^{\Gamma}$ be the category whose objects are assignments of categories to each element of $\Gamma$ (we do not impose any properness conditions here), with morphisms given by collections of functors (natural transformations will play no role here). This is a symmetric monoidal category.  We note that the functor
\begin{align} \label{eq:remember_coarse_stratification}
  \cD \Kur & \to \Cat \\
  (\bX \co \cQ \to \cD \Kur, \pi \co \cQ \to \cP) &  \mapsto \cP
\end{align}
naturally induces a functor
\begin{equation}
    {\cD \Kur}^{\Gamma}  \to \Cat^{\Gamma}  .
\end{equation}

\section{The Morse-Witten-Floer complex of a flow category}
\label{sec:chain-compl-assoc-1}

\subsection{Flow categories}
\label{sec:flow-categories}

Let $\cP$ be a partially ordered set. The key example to have in mind is the set of critical points of a Morse function, chosen for example on a Morse-Bott critical locus. In the next definition, $\Gamma_+$ refers to the \emph{elements of non-negative action} in an abelian monoid as in Section \ref{sec:kuran-pres-index}.
\begin{defin} \label{def:poset-flow-category}
For each pair $(p,p')$ of elements of $\cP$, and each element $\lambda \in \Gamma_+$, the partially ordered set $\cP^{\lambda}(p,p')$  consists of sequences
\begin{equation}
  (p, \lambda_0, p_1, \lambda_1, p_2, \cdots, p_{k}, \lambda_k, p')
\end{equation}
with $\{\lambda_i\}$ a collection of elements of $\Gamma_+$ whose sum is $\lambda$, and $\{p_i\}$ a collection of elements of $\cP$, so that
\begin{equation}
  \parbox{31em}{if $\cA(\lambda_i)=0$, then $p_i < p_{i+1}$.}
\end{equation}
The ordering on $\cP^{\lambda}(p,p') $ is generated by the relation that
\begin{equation}
 (p, \cdots,  \lambda_{i-1} + \lambda_i ,  \cdots, p') <   (p,  \cdots,  \lambda_{i-1}, p_i, \lambda_i ,  \cdots, p') .
\end{equation}  
\end{defin}

It is convenient for later generalisations to have a graphical representation of $\cP^{\lambda}(p,p')$ as shown in Figure \ref{fig:arc-morphism-flow-category}: its elements can be displayed as a directed arc (i.e. a directed tree in which all vertices are bivalent), with internal edges successively labelled by the elements $p_i$, and vertices labelled by the elements $\lambda_i$. The order is generated by collapsing an edge (hence forgetting the corresponding label by an element of $\cP$), and adding the elements of $\Gamma$ associated to its endpoints.

\begin{figure}[h]
  \centering
  \begin{tikzpicture}
    \draw[thick] (-1,0) -- (7,0);
    \node[label=left:{$p$}] at (-1,0) {};
    \node[label=right:{$p'$}] at (7,0) {};
    \foreach \i in {0,...,6}{ \filldraw (\i,0) circle (1pt);
      \node[label=below:{$\lambda_{\i}$}] at (\i,0) {};
     \ifthenelse{\i=0}{}{\node[label=above:{$p_{\i}$}] at (\i-.5,0) {}};
    };
  \end{tikzpicture}
  \caption{Graphical representation of an object of $\cP^{\lambda}(p,p') $}
  \label{fig:arc-flow-category}
\end{figure}
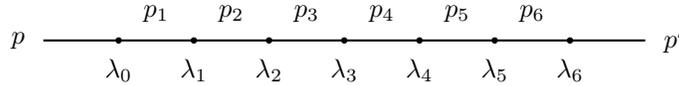

It follows from the definition of $ \cP^{\lambda}(p,p')$ that this set has minimal element the triple $(p, \lambda, p')$, and models manifolds with generalised corners, since the geometric realisation of the partially ordered set of elements lying between any pair of elements is a cube. Moreover, concatenation defines a natural map
  \begin{align}
    \cP^{\lambda}(p,p')  \times \cP^{\mu}(p',p'') & \to \cP^{\lambda + \mu}(p,p'') \\
\{ (p, \lambda_0,  \cdots,  \lambda_k, p') , (p', \mu_0,  \cdots,  \mu_l,p'') \} & \mapsto  (p, \lambda_0,  \cdots,  \lambda_k, p', \mu_0,  \cdots,  \mu_l,p'').
\end{align}
The above data encodes a category $\cP^{\Gamma_+}$ with object set $\cP$, and enriched in $\Cat^{\Gamma}$, with the property that the partially ordered set assigned to elements of negative action is empty. Noting that the minimal generator on the left is the product of $(p, \lambda, p')$ with $(p', \mu , p'')$, and that its image in $  \cP^{\lambda + \mu}(p,p'') $ is the element $(p,\lambda, p', \mu, p'') $ which has codimension $1$, we shall presently write $\ro_{\bR^{p'}}^{-1}  $ for the inverse of the corresponding conormal orientation line, which we consider as a trivial graded line in degree $1$, with the standard trivialisation of $\bR$ corresponding to the inward pointing direction.

\begin{defin} \label{def:flow_category}
  A \emph{Kuranishi flow category} $\bX$ with object set $\cP_{\bX}$ consists of (i) a choice of graded $\Bbbk$-line $\ro_{p}$ for each element of $\cP_{\bX}$, and (ii) a non-unital category enriched in $\cD{\Kur}_{\Gamma_+}$ with object set $\cP_{\bX}$ and lying over $\cP_{\bX}^{\Gamma_+}$. We require that:
  \begin{enumerate}
  \item the morphisms from $p$ to $p'$ are oriented relative 
\begin{equation} \label{eq:orientation-line-flow-category}
 \ro_{p,p'} \equiv  \ro_{p} \otimes \ro^{-1}_{p'} \otimes \ro^{-1}_{\bR^{p'}}.
\end{equation}
\item at the level of graded lines, the composition is given by the isomorphism
 \begin{align}
   \ro_{p,p'} \otimes \ro_{p',p''} & \cong   \ro_{p} \otimes \ro^{-1}_{p'} \otimes \ro^{-1}_{\bR^{p'}} \otimes \ro_{p'} \otimes \ro^{-1}_{p''} \otimes \ro^{-1}_{\bR^{p''}} \\
   & \cong \ro_{p}  \otimes \ro^{-1}_{\bR^{p'} } \otimes \ro^{-1}_{p''} \otimes \ro^{-1}_{\bR^{p''}} \\
   & \cong \ro_{p,p''} \otimes \ro^{-1}_{\bR^{p'}} .
 \end{align}
\item the following finiteness condition holds:
 \begin{equation} \label{eq:compactness-property-morphisms}
   \parbox{31em}{for each element $p \in \cP_{\bX}$, and real number $E$, the set of elements $p' \in \cP_{\bX}$ for which there is a label $\lambda \in \Gamma$ with $\cA(\lambda) \leq E$ and $\bX^\lambda(p,p')$ is non-empty and has virtual dimension $0$ is finite.}
 \end{equation}
\end{enumerate}
\end{defin}
\begin{rem}
  The reader may want to assume that $\cP_{\bX}$ is finite, in which case Condition \eqref{eq:compactness-property-morphisms} automatically holds.
\end{rem}
To clarify the first condition in the definition, let us point out that the datum of a lift of $\cP_{\bX}^{\Gamma_+}$ to $\cD \Kur$ entails a choice of category $\cQ^{\lambda}_{\bX}(p,p')$ for each pair of objects $(p,p')$ and for each element $\lambda \in \Gamma_+$, equipped with a map to $\cP^{\lambda}_{\bX}(p,p')$, and with a functor to $\cD \Kur$. The object $(p, \lambda,  p', \mu, p'' ) $ of $\cP^{\lambda + \mu}_{\bX}(p,p'') $ thus determines the subcategory $\partial^{ (p, \lambda,  p', \mu, p'' )} \cQ^{\lambda}_{\bX}(p,p')$. 
The following technical result is what allows us to use the theory of virtual counts developed in Part \ref{part:constr-theory-virt} in this setting:
\begin{lem} \label{lem:lift_flow_cat-to-factorised}
  Every Kuranishi flow category $\bX$ lifts to a category $\widetilde{\bX}$ enriched in $\cD \widetilde{\Kur}$.
\end{lem}
\begin{proof}
  We specify the lift of each morphism $\bX^{\lambda}(p,p')$ as follows: 
  each object $P$ of the stratifying category $\cQ_{\bX}^{\lambda}(p,p')$ can be represented as the image of unique composition
  \begin{equation}
     \cQ_{\bX}^{\lambda_1}(p,p_1) \times  \cQ_{\bX}^{\lambda_2}(p_1,p_2) \times \cdots \times   \cQ_{\bX}^{\lambda_k}(p_k,p') \to \cQ_{\bX}^{\lambda}(p,p')
   \end{equation}
   of maximal length, i.e. so that the inverse image of $P$ is a product of codimension $0$ objects $(P_1, \ldots, P_k)$. We thus obtain a lift
  \begin{equation}
  \partial^{P_1}   \bX^{\lambda_1}(p,p_1) \times  \partial^{P_2} \bX^{\lambda_2}(p_1,p_2) \times \cdots \times   \partial^{P_k} \bX^{\lambda_k}(p_k,p') \cong \partial^P \bX^{\lambda}(p,p')
   \end{equation}
   of the corresponding stratum.
 \end{proof}
 
\subsection{The chain complex associated to a flow category}
\label{sec:chain-compl-assoc}

Given a theory of virtual counts over a ring $\Bbbk$, we shall explain in this section a construction of a chain complex associated to each flow category. Let $\Lambda_0$ denote the Novikov ring with coefficients in $\Bbbk$ and exponents in $\Gamma$, i.e. the completion of the group ring of $\Gamma_+$ with respect to the topology induced by the action filtration.

The starting point is to assign to each flow category $\bX$ the graded $\Lambda_0$ module
\begin{equation}
  CF^*(\bX) \equiv \widehat{\bigoplus_{p \in \cP_{\bX}}} \ro_{p} \otimes_{\Bbbk} \Lambda_0,
\end{equation}
where $\widehat{\bigoplus}$ denotes the completion with respect to the $T$-adic topology on $\Lambda_0$.
\begin{rem}
  If $\cP_{\bX}$ is finite, then the map from the direct sum to its completion is an isomorphism. More generally, whenever we can ensure that, for each $p \in \cP_{\bX}$, there are only finitely many $p' \in \cP_{\bX}$ for which $\bX(p,p')$ is non-empty, we need not pass to the completion.
\end{rem}
We define an operator
\begin{equation}
  \bfm \co   CF^*(\bX)  \to   CF^{*+1}(\bX)
\end{equation}
as a sum of matrix coefficients labelled by triples $(p,\lambda,p')$, with $p$ and $p'$ in $\cP_{\bX}$, and $\lambda \in \Gamma_+$:
\begin{equation}
   \ro_{p} \otimes_{\Bbbk} \Lambda_0 \to \ro_{p'} \otimes_{\Bbbk} \Lambda_0.
\end{equation}
This term is given as the tensor product of $T^{\lambda}$ with a degree $1$ map
\begin{equation} \label{eq:component_of_differential}
   \bfm^{\lambda}_{p,p'} \co  \ro_{p} \to \ro_{p'}
 \end{equation}
 which we presently define.

 Start by recalling that the chain complex of strata $ C^*_{st}(\bX^\lambda(p,p'))$ is generated by the direct sum of $\Bbbk$-orientation lines $\ro_P$ for each element $P$ of the partially ordered set $\cP^{\lambda}_{\bX}(p,p')$ from Definition \ref{def:poset-flow-category}.  In the case of the element $\{p,\lambda,p'\}$, the orientation assumption in Definition \ref{def:flow_category} fixes an isomorphism of the associated line with $\ro_{p} \otimes \ro^{-1}_{p'} \otimes \ro^{-1}_{\bR^{p'}}$. On the other hand, combining Lemmas \ref{lem:virtual_counts_diagrams} and \ref{lem:lift_flow_cat-to-factorised}, we see that a theory of virtual counts for relatively oriented presentations determines a cochain map
 \begin{equation}
    C^*_{st}(\bX^{\lambda}(p,p'))  \to \Bbbk.
  \end{equation}
We thus obtain a map  \begin{equation}
    \ro_{p} \otimes \ro^{-1}_{p'} \otimes \ro^{-1}_{\bR} \to \Bbbk,
  \end{equation}
  which we can rearrange to the desired map in Equation \eqref{eq:component_of_differential}.

  \begin{lem} \label{lem:Floer-differential}
    The operator
    \begin{equation}
      \bfm \equiv \sum_{\lambda \in [0,\infty)}  \bigoplus_{p,p' \in \cP_{\bX}} T^{\lambda} \bfm^{\lambda}_{p,p'} 
    \end{equation}
    defines a differential on $ CF^*(\bX)$. 
  \end{lem}
  \begin{proof}
    The convergence of $\bfm$ in the $T$-adic topology follows from Condition \eqref{eq:compactness-property-morphisms}. To see that $\bfm$ squares to $0$, we observe that the coefficient of the component of $\bfm^2$ mapping the factor associated to $p$ to the one associated to $p''$ with valuation $\rho$ is the sum of all compositions
    \begin{equation}
       \bfm^{\lambda}_{p,p'} \circ \bfm^{\mu}_{p',p''},
     \end{equation}
     with $\lambda + \mu = \rho$. These terms exactly enumerate the codimension $1$ boundary strata of $\bX^{\rho}(p,p'')$, which in terms of the chain complex of strata corresponds to the fact that we have a map
     \begin{equation}
       C^*_{st}(\bX^\lambda(p,p')) \otimes C^*_{st}(\bX^\mu(p',p''))  \to  C^*_{st}(\bX^{\rho}(p,p'')),
     \end{equation}
     and that
     \begin{equation} \label{eq:boundary_of_fundamental_chain}
       \partial [\bX^{\rho}(p,p'')] = \sum_{\substack{p' \in \cP_{\bX} \\ \lambda + \mu = \rho }} [\bX^{\lambda}(p,p'')] \times [\bX^{\mu}(p,p'')].
     \end{equation}
     At this stage, we recall that the datum of a theory of virtual counts gives a cochain map
 \begin{equation}
    C^*_{st}(\bX(p,p'')) \to \Bbbk,
  \end{equation}     
     which implies that the right hand-side in Equation \eqref{eq:boundary_of_fundamental_chain}  maps trivially to $\Bbbk$. Unwinding the definitions, this is precisely the vanishing of the square of $\bfm$.
  \end{proof}

\section{The category of flow categories and the Floer functor}
\label{sec:flow-categ-chain}

The goal of this section is to lift the assignment of a cochain complex to each flow category to a functor from a category of flow categories to the category of chain complexes. The key task is to construct a cubically enriched category whose objects are flow categories, in which the morphisms are bimodules; this will itself be done in two steps, the first of which results in a category in which compositions are only associative up to homotopy, which is then rectified, in analogy with the construction of Section \ref{sec:categ-kuran-pres}.

\subsection{The partially ordered set controlling morphisms}
\label{sec:morph-flow-categ}

We now consider a pair $\bX$ and $\bY$ of flow categories, indexed by partially ordered sets $\cP_{\bX}$ and $\cP_{\bY}$.  In the simplest situations, the partially ordered set stratifying a bimodule is given as follows: we first define
\begin{equation}
  \bar{\cP}_{\bX}^{\mu} (p,p') = \begin{cases} \{p\} & \textrm{ if } p=p' \textrm{ and } \mu = 0 \\
    \cP_{\bX}^{\mu} (p,p') & \textrm{otherwise},
    \end{cases}
\end{equation}
and similarly for $\cP_{\bY}$. Then, for each pair $(p,q)$ of elements of $\cP_{\bX}$ and $\cP_{\bY}$, and real number $\mu$, we define the set
\begin{equation} \label{eq:poset-continuation}
  \cP_{\bX; \bY}^{\mu}(p,q) \equiv  \coprod_{\substack{ p' \in \cP_{\bX}, \, q'  \in \cP_{\bY} \\ \mu_- \in \Gamma_+, \,  \mu_+ \in \Gamma_+} } \bar{\cP}_{\bX}^{\mu_-}(p,p') \times \{\mu - \mu_- - \mu_+ \} \times \bar{\cP}_{\bY}^{\mu_+}(q',q),
\end{equation}
where we recall that $\Gamma_+$ is the set of elements of $\Gamma$ with non-negative action. This set is equipped with a partial ordering which is induced by the partial order on $\cP_{\bX} $ and  $\cP_{\bY} $. To write it out explicitly, we start by setting $\mu_0$ to be the difference $\mu - \mu_- - \mu_+$. We can then write the above as the set of (non-empty) sequences
\begin{equation}
  (p = p_{-\ell-1}, \mu_{-\ell}, p_{-\ell}, \ldots,  \mu_{-1}, p_{-1}, \mu_0, q_{1}, \mu_{1}, q_2 , \ldots,q_{r}, \mu_{r} , q = q_{r+1})
\end{equation}
with $\mu_i$ a collection of elements of $\Gamma$ which are assumed to have non-negative action, except for $\mu_0$ which is allowed to be arbitrary, and with $\{p_i\}_{i = -\ell}^{-1}$ and $\{q_i\}_{i=1}^{r}$ a collection of elements of $\cP_{\bX}$ and  $ \cP_{\bY}$ (see Figure \ref{fig:arc-morphism-flow-category}). We impose as before the condition that,
\begin{equation}
  \parbox{31em}{if $\cA(\mu_i)=0$, then $p_i < p_{i+1}$ if $i < 0$, and $q_i < q_{i+1}$ if $0 < i $.}
\end{equation}

\begin{figure}[h]
  \centering
  \begin{tikzpicture}
    \draw[thick] (-1,0) -- (7,0);
    \node[label=left:{$p$}] at (-1,0) {};
    \node[label=right:{$q$}] at (7,0) {};
    \foreach \i in {0,...,6}{\pgfmathtruncatemacro{\j}{\i-2};
     \ifthenelse{\i=2}{\filldraw (\i,0) circle (3pt)}{\filldraw (\i,0) circle (1pt)};
     \node[label=below:{$\mu_{\j}$}] at (\i,0) {};
     \ifthenelse{\i<2}{\node[label=above:{$p_{\j}$}] at (\i+.5,0) {}}{};
     \ifthenelse{2<\i}{\node[label=above:{$q_{\j}$}] at (\i-.5,0) {}}{};
    };
  \end{tikzpicture}
  \caption{Graphical representation of an object of $\cP^{\mu}_{\bX ; \bY}(p,q) $}
  \label{fig:arc-morphism-flow-category}
\end{figure}
Given a pair $(p,p')$ of elements of $\cP_{\bX}$, and an element $q$ of $\cP_{\bY}$, we have a natural concatenation map
  \begin{equation}
   \cP_{\bX}^{\lambda}(p,p') \times  \cP_{\bX ; \bY}^{\mu}(p',q) \to \cP_{\bX ; \bY}^{\mu + \lambda}(p,q),
 \end{equation}
 and for each element $p$ of $\cP_{\bX}$ and pair $(q',q)$ of elements of $\cP_{\bY}$, we have a map
  \begin{equation}
     \cP_{\bX ; \bY}^{\mu}(p,q') \times \cP_{\bY}^{\rho}(q',q) \to \cP_{\bX ; \bY}^{\mu + \rho}(p,q).
   \end{equation}
   These operations are the structure maps of a $\Cat^{\Gamma} $-enriched bimodule $ \cP_{\bX ; \bY}^{\Gamma}$ over the categories $\cP_{\bX}^{\Gamma_+}$ and $\cP_{\bY}^{\Gamma_+} $. We also introduce the sub-bimodule
   \begin{equation}
      \cP_{\bX ; \bY}^{\Gamma_+} \subset  \cP_{\bX ; \bY}^{\Gamma}
   \end{equation}
   which assigns to a pair $(p,q)$ and a label $\mu$ the subset of $\cP_{\bX ; \bY}^{\mu}(p,q) $ consisting of sequences for which all elements of $\Gamma$ have non-negative action.

\subsection{A symmetric semi-cubical set of morphisms}
\label{sec:cubic-set-morph}

Our next task is to construct a symmetric semi-cubical set of morphisms between flow modules. The naive idea is to consider bimodules enriched in $\cD \Kur^{\Gamma}$, lying over the product of the bimodule $ \cP_{\bX ; \bY}^{\Gamma}$ with $\Face \square^n$ (the set of faces of the cube, ordered by reverse inclusion).

\begin{rem}
We can extend the construction of this section to produce degeneracies as well, but this will cause some technical difficulties later, essentially because the product of cubical sets involves taking the quotient by redundant degeneracies.
\end{rem}

It is not too difficult to construct a semi-cubical set in terms of the above product, but the action of symmetries fails to model what is provided by geometry (the natural symmetric action is by post-composition with automorphisms of $\Face \square^n$, so there are no fixed points).

Our roundabout solution to this problem is to include in the datum of a $n$-cube additional data that specifies its image under all symmetries of cubes.  Given a pair $\bX$ and $\bY$ of flow categories and  a natural number $n$, we thus introduce a set $\flow_n(\bX, \bY)$ whose elements are given by the following:
   \begin{enumerate}
   \item For each permutation $\phi$ of $\{1, \ldots, n\}$, an $\bX$-$\bY$-bimodule $\bM_{\phi}$ enriched in $\cD \Kur^{\Gamma}$, whose image under the forgetful functor to  $\Cat^{\Gamma}$-enriched bimodules with respect to $\cP_{\bX}^{\Gamma_+}$ and $\cP_{\bY}^{\Gamma_+} $ is $\cP_{\bX ; \bY}^{\Gamma}  \times \Face  \square^{n}$.
   \item For each composition of permutations  $\phi_1 \circ \phi_2 $ an isomorphism
     \begin{equation} \label{eq:choice_iso_cube_Kur}
        \bM_{\phi_1  \circ \phi_2} \cong \phi_2^* \bM_{\phi_1}      
     \end{equation}
     where the right hand side is given at the level of stratifying categories by the pullback in the diagram
     \begin{equation}
       \begin{tikzcd}
     \phi_2^* \cQ_{\bM,\phi_1} \ar[r] \ar[d] &   \cQ_{\bM,\phi_1} \ar[d] \\
       \cP_{\bX ; \bY}^{\Gamma}  \times \Face  \square^{n} \ar[r, "\id \times \phi_2"]  & \cP_{\bX ; \bY}^{\Gamma}  \times \Face  \square^{n}
       \end{tikzcd}
     \end{equation}
     and at the level of diagrams of presentations by the composite functor
     \begin{equation}
            \phi_2^* \cQ_{\bM,\phi_1} \to \cQ_{\bM,\phi_1}   \to \Kur.
     \end{equation}
   \item For each triple composition $\phi_1 \circ \phi_2 \circ \phi_3$, we require that the isomorphisms of pullbacks fit in a commutative diagram
     \begin{equation}
       \begin{tikzcd}
         \bM_{\phi_1  \circ \phi_2 \circ \phi_3} \ar[r] \ar[d] & \phi_3^{*}  \bM_{\phi_1 \circ \phi_2} \ar[d] \\
     \left(\phi_2 \circ    \phi_3 \right)^* \bM_{\phi_1}  \ar[r]   & \phi_3^{*} \left( \phi_2^* \bM_{\phi_1} \right).
       \end{tikzcd}
     \end{equation}
\end{enumerate}
    
Up to isomorphism, the data assigned to any permutation is determined by the flow bimodule associated to the identity map of the $n$-cube, which we denote $\bM$. We can thus formulate all properties which are invariant under isomorphisms in terms of this bimodule, as we shall presently do for the orientation data on the elements of $\flow_n(\bX, \bY) $ after we state the following basic result:
   \begin{lem} \label{lem:codim_1-elt-bimodules}
  The image in $ \Face\square^n \times \cP_{\bX ; \bY}^{\mu}(p,q)$ of each codimension $1$ element of $\cQ_{\bM}^{\mu}(p,q)$ satisfies one of the following mutually exclusive alternatives:
   \begin{enumerate}
   \item it is the product of a facet of $ \square^n$ with the minimal element of $\cP_{\bX ; \bY}^{\mu}(p,q)$. In this case, the normal direction is identified with the normal direction of this facet of the cube.
   \item it is the product of top stratum of $\square^n$ with a codimension one element of $\cP_{\bX ; \bY}^{\mu}(p,q)$ of the form $(p, \lambda, p', \mu - \lambda, q)$. We identity the normal orientation line associated to this inclusion with $\ro_{\bR^{p'}} $, with the inner pointing covector corresponding to the standard orientation on $\bR$.
   \item it is the product of top stratum of $[0,1]^n$ with a codimension one element of $\cP_{\bX ; \bY}^{\mu}(p,q)$ of the form $(p, \mu - \rho, q', \rho, q)$.  We identify the normal orientation line associated to this inclusion by $\ro_{\bR^{q'}}$, with the outward pointing covector corresponding to the standard orientation.
   \end{enumerate} \qed
   \end{lem}

   We now specify the orientation data which we require each element $\bM$ of $\flow_n(\bX, \bY) $ to be equipped with as follows:
\begin{enumerate}
\item an orientation of the Kuranishi diagram $\bM^{\mu}(p, q)$ relative $\ro_{p} \otimes \ro^{-1}_{q} \otimes \ro_{[0,1]^n} $.
\item at the level of orientation lines, the map $\bX(p, p') \times \bM(p', q) \to \bM(p, q) $ is induced by the product of the identity on $\ro_{[0,1]^n} $ with the isomorphism
  \begin{equation}
    \ro_{p} \otimes \ro^{-1}_{p'} \otimes \ro^{-1}_{\bR^{p'}} \otimes  \ro_{p'} \otimes \ro^{-1}_{q} \cong    \ro_{p}  \otimes \ro^{-1}_{\bR^{p'}} \otimes \ro^{-1}_{q}.
  \end{equation}

\item at the level of orientation lines, the map $\bM(p, q') \times \bY(q', q) \to \bM(p, q) $ is induced by the product of the identity on $\ro_{[0,1]^n}$ with the isomorphism
  \begin{equation}
      \ro_{p} \otimes  \ro^{-1}_{q'} \otimes \ro_{q'} \otimes \ro^{-1}_{\bR^{q'}} \otimes  \ro^{-1}_{q} \cong   \ro_{p} \otimes \ro^{-1}_{\bR^{q'}} \otimes \ro^{-1}_{q}.
  \end{equation}
  \end{enumerate}

  \begin{lem} \label{lem:flow_bimodules_cubical_set}
    The collection $\{ \flow_n(\bX, \bY)\}_{n=0}^{\infty}$ naturally admits the structure of a symmetric semi-cubical set. 
  \end{lem}
  \begin{proof}
    We assign to each permutation $\phi$ the map
    \begin{equation}
            \flow_n(\bX, \bY) \to \flow_n(\bX, \bY) 
    \end{equation}
    which assigns to a permutation $\psi$ the data $\bX_{\phi \circ \psi}$.

 We assign to each injection of symmetric cubes $\square^k \stackrel{\beta}{\to} \square^n $ the map
    \begin{equation}
            \flow_n(\bX, \bY) \to \flow_k(\bX, \bY) 
    \end{equation}
    which is obtained by restricting to the permutations $\phi$ which are the identity on all factors of $\square^n$ in which this image is constant, and is given for these permutations (which are naturally identified with the permutations of $\{1, \ldots, k\}$) by restricting $\cQ$ to the corresponding boundary stratum.
       The choice of orientation on faces is induced by the identification of normal directions as in Lemma \ref{lem:codim_1-elt-bimodules}.
  \end{proof}
\subsection{Composition of morphisms}
\label{sec:comp-morph}

Our next goal is to associate to each triple of flow categories a composition map of flow bimodules. We begin at the level of partially ordered sets: given a triple of  partially ordered sets $\cP_{\bX}$, $\cP_{\bY}$, and  $\cP_{\bZ}$, we consider the composition bimodule 
\begin{equation}
    \cP_{\bX ; \bY}^{\Gamma} \times_{  \cP_{\bY}^{\Gamma_+}}  \cP_{\bY ; \bZ}^{\Gamma},
\end{equation}
over  $\cP_{\bX}^{\Gamma_+}$ and $\cP_{\bY}^{\Gamma_+} $, which can be explicitly described by specifying that it assigns to $p \in \cP_{\bX}$, $r \in \cP_{\bZ}$, and $\tau \in \Gamma$ the partially ordered set
\begin{equation}
\left(    \coprod_{\substack{q \in \cP_{\bY} \\ \lambda + \rho = \tau} }  \cP_{\bX ; \bY}^{\lambda}(p,q)  \times  \cP_{\bY ; \bZ}^{\rho}(q,r) \right)/ \sim,
\end{equation}
where $\sim$ is the equivalence relation that identifies the two images of each element of the sets
\begin{equation}
    \cP_{\bX ; \bY}^{\lambda}(p,q)  \times  \cP_{ \bY}^{\mu}(q,q')  \times  \cP_{\bY ; \bZ}^{\rho}(q',r)
\end{equation}
under the left and right actions of $\cP_{\bY}^{\Gamma_+} $ on these two modules.  Figure \ref{fig:arc-composition-flow-module} shows a representative element of this bimodule.
\begin{figure}[h]
  \centering
  \begin{tikzpicture}
    \draw[thick] (-1,0) -- (7,0);
    \node[label=left:{$p$}] at (-1,0) {};
    \node[label=right:{$r$}] at (7,0) {};
    \foreach \i in {0,...,6}{\pgfmathtruncatemacro{\j}{\i-2};\pgfmathtruncatemacro{\k}{\i-5};
     \ifthenelse{\i=2 \OR \i=5}{\filldraw (\i,0) circle (3pt)}{\filldraw (\i,0) circle (1pt)};
     \ifthenelse{0<\i}{ \ifthenelse{\i<3}{\node[label=above:{$p_{\j}$}] at (\i-.5,0) {}}{ \ifthenelse{5<\i}{\node[label=above:{$r_{\k}$}] at (\i-.5,0) {}}{\node[label=above:{$q_{\j}$}] at (\i-.5,0) {}}}}{};
     \ifthenelse{\i<3}{\node[label=below:{$\lambda_{\j}$}] at (\i,0) {}}{\ifthenelse{\i<5}{\node[label=below:{$\mu_{\j}$}] at (\i,0) {}}{\node[label=below:{$\rho_{\k}$}] at (\i,0) {}}};
    };
  \end{tikzpicture}
  \caption{Graphical representation of an object of $  \cP_{\bX ; \bY}^{\Gamma} \times_{  \cP_{\bY}^{\Gamma_+}}  \cP_{\bY ; \bZ}^{\Gamma} (p,r) $}
  \label{fig:arc-composition-flow-module}
\end{figure}

Note that there is a natural map of bimodules
  \begin{equation}
    \cP_{\bX ; \bY}^{\Gamma} \times_{  \cP_{\bY}^{\Gamma_+}}  \cP_{\bY ; \bZ}^{\Gamma} \to  \cP_{\bX ; \bZ}^{\Gamma},  \end{equation}
  given in the graphical description of Figure \ref{fig:arc-composition-flow-module} by collapsing the sub-arc consisting of edges labelled by element of $\cP_{\bY}$ (and adding the values of the elements of $\Gamma$ associated to the vertices). The following result will allows us to formulate associativity: 
\begin{lem}
For any quadruple of partially ordered sets $\cP_{\bW}$, $\cP_{\bX}$, $\cP_{\bY}$, and $\cP_{\bZ}$, the following diagram commutes:
  \begin{equation}
    \begin{tikzcd}
      \cP_{\bW ; \bX}^{\Gamma} \times_{  \cP_{\bX}^{\Gamma_+}}   \cP_{\bX ; \bY}^{\Gamma} \times_{  \cP_{\bY}^{\Gamma_+}}  \cP_{\bY ; \bZ}^{\Gamma} \ar[r] \ar[d] &   \cP_{\bW ; \bX}^{\Gamma} \times_{  \cP_{\bX}^{\Gamma_+}}   \cP_{\bX ; \bZ}^{\Gamma} \ar[d] \\
      \cP_{\bW ; \bY}^{\Gamma} \times_{  \cP_{\bY}^{\Gamma_+}}  \cP_{\bY ; \bZ}^{\Gamma} \ar[r] &   \cP_{\bW ; \bZ}^{\Gamma}.
    \end{tikzcd}
  \end{equation} \qed
\end{lem}

We now consider a triple $\bX$,  $\bY$, and $\bZ$ of flow categories, and a pair $\bM$ and $ \bN$ of elements of  $\flow_m(\bX,\bY)$  and $\flow_n(\bY, \bZ)$.      We note that a permutation of $\{1, \ldots, m+n\}$ induces permutations $\phi_m$ and $\phi_n$ of $\{1, \ldots, m\}$ and $\{1, \ldots, n\}$, obtained from the ordering of the images of $\{1, \ldots, m\}$ and $\{m+1, \ldots, m+n\}$ under $\phi$.

The composition $\Cat^{\Gamma}$ enriched bimodule $ \cQ_{\bM \circ \bN, \phi}$ over  $\cQ^{\Gamma_+}_{\bX}$ and  $\cQ^{\Gamma_+}_{\bZ}$  is given again as the coequaliser of the two functor 
    \begin{multline}
   \coprod_{\substack{q,q' \in \cP_{\bY} \\ \lambda + \mu + \rho = \tau}} \cQ^{\lambda}_{\bM, \phi_m}(p,q) \times \cQ^{\mu}_{\bY}(q,q') \times \cQ_{\bN, \phi_n}^{\rho}(q',r)  \rightrightarrows  \\     \coprod_{\substack{ q \in \cP_{\bY} \\ \lambda + \rho = \tau}} \cQ_{\bM, \phi_m}^{\lambda}(p,q) \times \cQ_{\bN, \phi_n}^{\rho}(q,r)  
      \end{multline}
      given by the action of  $\cQ^{\Gamma_+}_{\bY}$ on $\cQ_{\bM, \phi_m}$ and $\cQ_{\bN, \phi_n}$. In other words, the set of objects is the quotient of the disjoint union of the sets of objects of $ \cQ_{\bM, \phi_m}^{\lambda}(p,q) \times \cQ_{\bN, \phi_n}^{\rho}(q,r)  $ by the above equivalence relation, and the morphisms are the images of morphisms in this quotient, where we identify two morphisms which are the image of the same morphism in one of the categories $\cQ^{\lambda}_{\bM}(p,q) \times \cQ^{\mu}_{\bY}(q,q') \times \cQ_{\bN}^{\rho}(q',r)  $. Note that we have a natural map of bimodules
          \begin{equation}
  \cQ_{\bM \circ \bN, \phi} \to  \cP_{\bX ; \bZ}^{\Gamma} \times \Face(\square^{n+m}).
      \end{equation}

      \begin{defin} \label{def:composition-flow-bimodules}
      The composition 
  \begin{equation}
    \bM \circ  \bN \in \flow_{m+n}(\bX, \bZ)
  \end{equation}
  is the $m+n$ cube, stratified by $ \cQ_{\bM \circ \bN}  $ which assigns to a permutation $\phi$ of $\{1, \ldots, m+n\}$ and to each pair of objects
  \begin{equation}
   (P_{\bM},P_{\bN}) \in \cQ_{\bM, \phi_m}^{\lambda}(p,q) \times \cQ_{\bN, \phi_n}^{\rho}(q,r)
  \end{equation}
  the product Kuranishi presentation
  \begin{equation}
  \partial^{P_{\bM}}  \bM^{\lambda}_{\phi_m}(p,q) \times  \partial^{P_{\bN}} \bN_{\phi_n}^{\rho}(q,r).
  \end{equation}
  
\end{defin}

  The fact that this construction is functorial after passing to $  \cQ_{\bM \circ \bN} $ follows from the fact that the maps
  \begin{multline}
 \partial^{P_{\bM}} \bM^{\lambda}_{\phi_m}(p,q) \times \partial^{P_{\bY} \times P_{\bN}} \bN^{\mu + \rho}_{\phi_n}(q,r)   \leftarrow \partial^{P_{\bM}} \bM^{\lambda}_{\phi_m}(p,q) \times \partial^{P_{\bY}}\bY^{\mu}(q,q') \\ \times  \partial^{P_{\bN}}\bN^{\rho}_{\phi_n}(q',r) \to  \partial^{P_{\bM} \times P_{\bY}} \bM^{\lambda+ \mu}_{\phi_m}(p,q') \times \partial^{P_{\bN}} \bN^{\rho }_{\phi_n}(q',r).
  \end{multline}
are both isomorphisms of Kuranishi presentations.  Note as well that the fact that $ \bM^{\lambda}_{\phi_m}(p,q) $ and $ \bN^{\rho}_{\phi_n}(q,r)$ have only finitely many strata and are empty except for a set of values of $\lambda$ and $\rho$ on which the action map is proper and bounded below implies that the same finiteness properties holds for the fibre product. 

We omit the proof of the next result:
\begin{lem}
  The assignment $(\bM, \bN) \mapsto  \bM \circ  \bN $ extends to a map of cubical sets
  \begin{equation}
     \flow(\bX, \bY) \times \flow(\bY, \bZ) \to \flow(\bX, \bZ).
   \end{equation} \qed
\end{lem}

Unfortunately, the above map is not associative, in the sense that the diagram
   \begin{equation}
     \begin{tikzcd}
       \flow(\bX, \bY) \times \flow(\bY, \bZ) \times \flow(\bZ, \bW) \ar[r] \ar[d]  & \flow(\bX, \bZ) \times \flow(\bZ, \bW)  \ar[d] \\
    \flow(\bX, \bY) \times \flow(\bY, \bW) \ar[r] &  \flow(\bX, \bW)
     \end{tikzcd}
   \end{equation}
   does not commute. The problem amounts to the basic fact that the product of (partially ordered) sets is not strictly associative, i.e. $X \times (Y \times Z)$ is naturally isomorphic to $(X \times Y) \times Z $, but not equal to it.

      The standard way to encode this natural isomorphism would be to equip $\flow$ with a bicategorical structure: the $0$-cells are flow categories, and the $1$-cells are symmetric cubical categories. Explicitly, we consider $\flow_n(\bX,\bY)$ as a category with objects given by flow modules as above, where morphisms from $\bM$ to $\bN$ are given by maps of $\cD \Kur^{\Gamma}$ enriched bimodules
   \begin{equation} \label{eq:map_bimodule_2-category}
\bM_{\phi} \to \bN_{\phi},     
   \end{equation}
   so that the following properties hold:
   \begin{enumerate}
   \item The underlying map of $\Cat^{\Gamma}$ enriched bimodules fit in a commutative diagram
     \begin{equation}
       \begin{tikzcd}
         \cQ_{\bM_{\phi}} \ar[rr] \ar[dr] & & \cQ_{\bN_{\phi}} \ar[dl] \\ 
         & \cP_{\bX; \bY}^{\Gamma}  \times \Face  \square^{n}.& 
       \end{tikzcd}
     \end{equation}
   \item The isomorphisms in Equation \eqref{eq:choice_iso_cube_Kur} and the maps in Equation \eqref{eq:map_bimodule_2-category} are compatible in the sense that the following diagram commutes:
     \begin{equation}
              \begin{tikzcd}
          \bM_{\phi_1  \circ \phi_2} \ar[r] \ar[d] & \phi_2^* \bM_{\phi_1}     \ar[d] \\
\bN_{\phi_1  \circ \phi_2} \ar[r]  & \phi_2^* \bN_{\phi_1}.
       \end{tikzcd}
     \end{equation}
   \end{enumerate}
   The cubical structure maps are strictly compatible with this notion of morphism, so that each arrow $\alpha \co \square^k \to \square^n $ defines a functor
   \begin{equation}
\alpha^* \co     \flow_n(\bX, \bY) \to   \flow_k(\bX, \bY),
   \end{equation}
   with the property that we have a strict equality
   \begin{equation}
     \alpha_1^* \circ  \alpha_2^* = \left(  \alpha_1 \circ \alpha_2 \right)^*.
   \end{equation}
   \begin{lem}
     The composition map $\bM \circ \bN$, and the natural isomorphism
     \begin{equation}
       \left( \bM \circ \bN\right) \circ \bO \cong \bM \circ \left( \bN \circ \bO \right)       
     \end{equation}
     equips $\flow$ with the structure of a bicategory. \qed
   \end{lem}

\subsection{Strictifying the category of flow categories}
\label{sec:strict-categ-flow}

Our goal now is to produce an ordinary, cubically enriched, category of flow modules, which we shall achieve via a construction in the spirit of the Isbell construction. Let $\Flow_{n}(\bX, \bZ)$ denote the set whose elements consist of the following data:
   \begin{enumerate}
   \item For each face $\sigma$ of the $n$-cube of dimension $n_{\sigma}$, (i) a collection $(\bY^\sigma_0, \ldots, \bY^\sigma_{k_\sigma+1})$ of flow categories, with $\bY^\sigma_0 = \bX$ and $\bY^\sigma_{k_\sigma+1} = \bZ$, (ii) a collection $\{d^{\sigma}_i\}_{i=1}^{k_\sigma}$ of non-negative integers, and disjoint subsets $D_{i}^{\sigma}$ of $\{1, \ldots, n_{\sigma}\}$ of cardinality $d_i^{\sigma}$ (we allow the case $d^{\sigma}_i = 0$, in which case $  D^\sigma_i  = \emptyset$), and (iii) an element $\bM^{\sigma}_i $ of  $\flow_{d^{\sigma}_i}(\bY^\sigma_{i}, \bY^\sigma_{i+1}) $ for each $1 \leq i \leq k_\sigma$.
      \item For each pair $\sigma < \tau$ in $\Face\square^n$, a surjective monotone map
     \begin{equation}
\pi_{\sigma}^{\tau} \co     \{0, \ldots, k_{\tau} + 1\} \to   \{0, \ldots, k_{\sigma} + 1\},
\end{equation}
with the property that $D_{j}^{\sigma}$ maps, under the inclusion of $\{1, \ldots, n_{\tau}\} $ in $ \{1, \ldots, n_{\sigma}\}$ to $D_{\pi_{\sigma}^{\tau}(j)}^{\sigma} $. For each number $i$ between $0$ and $k_{\sigma} + 1 $, letting $k_i$ denote,  the smallest number in the source of $\pi_{\sigma}^{\tau}$ mapping to $i$, we have as well an equivalence of flow modules
     \begin{equation} \label{eq:map_of_decomposed_strata}
       \bM^\tau_{k_i+1} \times_{\bY^\tau_{k_{i+1}}} \bM^\tau_{k_i+2} \times_{\bY^\tau_{k_i+2}} \cdots \times_{\bY^\tau_{k_{i+1}-1}} \bM^\tau_{k_{i+1}} \times [0,1]^{e_{\sigma}^{\tau}(i)} \to \partial^{\tau} \bM^{\sigma}_i,
     \end{equation}
     where the right hand side is the boundary stratum associated to $\tau$, while $e_{\sigma}^{\tau}(i) $ is the difference between $\sum_{j=k_i+1}^{k_{i+1}} d^\tau_j$ and the dimension of the boundary stratum of $ \bM^{\sigma}_i$ associated to $\tau$ (which is the cardinality of the inverse image of $D_{i}^{\sigma}$ in $\{1, \ldots, n_{\tau}\}$).
\end{enumerate}
Finally we require these data to be compatible in the sense that, for each triple $\sigma < \tau < \rho$, the diagrams associated to the maps in Equation \eqref{eq:map_of_decomposed_strata} commute after applying the natural isomorphisms associated to reparenthesisation.
        \begin{rem}
          It may be useful to note that the construction at this stage could be simplified in two ways: we could assume that the subsets $D_i^{\sigma}$ form a partition, which would make the factor $[0,1]^{e_{\sigma}^{\tau}(i)} $ in Equation \eqref{eq:map_of_decomposed_strata} extraneous, and we could assume that this partition is standard (in the sense that it consists of successive subintervals). Since we need the construct $\Flow$ as a category enriched over cubical sets (rather than semi-cubical sets), the first condition would require us to have defined degeneracies in $\flow$, and then to address the fact that products of cubical sets entails taking the quotient by redundant degeneracies; this would give us less control in Appendix \ref{sec:lifting-flow-modules} when we prove a lifting theorem. On the other hand, the use of standard ordering would require restricting to non-symmetric cubical sets, which would not raise any difficulties for the construction of $\Flow$ as a category, but would preclude us from defining it as a multicategory in Section \ref{sec:mult-flow-categ} below.
                 \end{rem}
        It is straightforward to define face maps
        \begin{equation}
             \Flow_{n}(\bX, \bY) \to \Flow_{n-1}(\bX, \bY) ,     
           \end{equation}
           as they are simply given by considering the data associated to the given codimension $1$ stratum. We define the action of a symmetry $\phi$ by its action on the set of faces, the induced action on partitions, and the induced permutations 
           \begin{equation}
              \flow_{d^{\sigma}_i}(\bY^\sigma_{i}, \bY^\sigma_{i+1})  \to      \flow_{d^{\sigma}_i}(\bY^\sigma_{i}, \bY^\sigma_{i+1})        
           \end{equation}
           which intertwines the two orderings of $D_{i}^{\sigma} $ arising from the ordering of the coordinates of $\sigma$ and $\phi \sigma$.       Finally, the degeneracy map associated to an inclusion $s \co \{1,\ldots, n\} \to \{1, \ldots, m\}$ assigns to each face of the $m$-cube the collection of flow categories and flow modules corresponding to its image, and to each face $\sigma$ of the $m$-cube the subset of $ \{1, \ldots, m\} $ which is the image of $D_{i}^{s(\sigma)}$ under $s$.

           The compatibility of these maps is given by:
        \begin{lem}
          The collection $\{\Flow_{n}(\bX, \bY) \}_{n=0}^{\infty}$ forms a symmetric cubical set.
          \qed          
        \end{lem}

        There is a natural map
        \begin{equation} \label{eq:composition_lifted_flow_categories}
          \Flow_{n}(\bX, \bY) \times \Flow_{m}(\bY, \bZ) \to \Flow_{n+m}(\bY, \bZ)
        \end{equation}
        given, for each stratum of $\square^{n+m}$  by the formal product (concatenation) of sequences of bimodules. The key point is that this operation is strictly associative, is equivariant with respect to the action of the product of symmetric groups $\Sigma_n \times \Sigma_m $, and with both face and degeneracy maps, so we introduce the following definition:
        \begin{defin} \label{def:Flow_modules}
          The \emph{category of flow modules with factorised strata} is the cubically enriched category $\Flow$, with objects given by flow categories, morphisms by the cubical sets $\Flow(\bX, \bY) $, and compositions by Equation \eqref{eq:composition_lifted_flow_categories}.
        \end{defin}

        We distinguish the subcategory $ \Flow^{+}$ of \emph{monotone bimodules,} consisting of those bimodules which are supported at elements of $ \Gamma$ with non-negative real action.

\subsection{The functor from flow categories to chain complexes}
\label{sec:from-flow-categories}

We now return to the construction of Section \ref{sec:chain-compl-assoc}: let $\Lambda$ denote the Novikov field, with $\Bbbk$ coefficients and exponents in $\Gamma$, i.e. the completion along the  topology induced by the action homomorphism of the group algebra of $\Gamma$, over the ring $\Bbbk$. We consider the category $\Ch_{\Lambda}$ of chain complexes over $\Lambda$ as a category enriched over symmetric cubical sets:  explicitly, we define an $n$-cube of maps from $C^*_0$ to $C^*_1$ to be a map
\begin{equation}
  C_*( [0,1])^{\otimes n} \otimes C^*_0  \to C^*_1
\end{equation}
where $C_*([0,1])$ is the cochain complex with one generator in degree $-1$, corresponding to the top stratum of $[0,1]$, and two generators in degree $0$ corresponding to the endpoints $\{0,1\}$. The symmetric group acts via the Koszul sign conventions, the face maps are given by the natural inclusions, and the degeneracies act trivially. We also write $\Ch_{\Lambda_0}$ for the category of chain complexes over the Novikov ring $\Lambda_0$.

\begin{prop} \label{prop:floer-functor}
  A theory of virtual counts determines a cubically enriched functor
  \begin{equation}
  CF^* \co  \Flow \to \Ch_{\Lambda}
  \end{equation}
  and a lift to $\Ch_{\Lambda_0}$ of the restriction to $\Flow^+$.
\end{prop}
\begin{proof}
At the level of objects, the functor assigns to each flow category the chain complex $CF^*(\bX)$ in $ \Ch_{\Lambda_0}$, and its tensor product with $\Lambda$ in $\Ch_{\Lambda} $. 

At the level of morphisms, we must assign to each element $\bM$ of $ \Flow_n(\bX, \bY) $ a map 
  \begin{equation}  \label{eq:structure_map_bimodule}
\bfm^{\bM} \co  C_*( [0,1])^{\otimes n} \otimes CF^*(\bX)  \to CF^*(\bY).
  \end{equation}
  This map is given by a collection of $\Bbbk$-linear maps
  \begin{equation} \label{eq:structure_map_coefficient_bimodule}
    \bfm^{\bM^{\mu}}_{p,q} \co  C_*( [0,1])^{\otimes n} \otimes \ro_{p} \to \ro_q
  \end{equation} indexed by $p \in \cP_{\bX}$, $y \in \cP_{\bY}$, and $\mu \in \Gamma$. This triple determines a diagram
\begin{equation}
  \cQ_{\bM}^{\mu}(p,q) \to \Kur,
\end{equation}
which is equipped with a product decomposition. The main result of Appendix \ref{sec:lifting-flow-modules} is that the construction of Lemma \ref{lem:lift_flow_cat-to-factorised} extends to a cubically enriched functor $\Flow \to \widetilde{\Flow}$, whose target consists of flow modules equipped with lifts to $\widetilde{\Kur}$. Applying this result, we define the value of Equation \eqref{eq:structure_map_coefficient_bimodule}  on the top degree generator of $ C_*( [0,1])^{\otimes n}$ to be the sum  over all the codimension $0$ objects of $ \cQ_{\bM}^{\mu}(p,q)$ of the maps induced by the chosen theory of virtual counts:  the associated Kuranishi presentations are oriented relative
\begin{equation}
  \ro_{p} \otimes \ro^{-1}_{q} \otimes \ro_{[0,1]^n},
\end{equation}
hence induce a map
\begin{equation}
\ro_{[0,1]^n} \otimes  \ro_{p} \to \ro_{q}.
\end{equation}
This defines the map on the top degree generator of the domain in Equation \eqref{eq:structure_map_coefficient_bimodule}. The maps on the lower degree generators are determined by the maps associated to lower dimensional cubes.

To see that Equation \eqref{eq:structure_map_bimodule} is a cochain map, we use the description of the codimension $1$ elements of $\cQ_{\bM}^{\mu}(p,q)$; the three alternatives respectively correspond to the terms associated to the differential on $ C_*( [0,1])^{\otimes n}$, $CF^*(\bX)$, and $CF^*(\bY)$, with the discussion of orientations preceding Lemma \ref{lem:flow_bimodules_cubical_set} implying that these terms corresponds to the equation for a co-chain map.

The statement that we obtain a map of cubical sets from $ \Flow(\bX, \bY)$ to the space of morphisms from $\bX$ to $\bY$ now follows from the fact that the map associated to a boundary stratum of the cube agrees with the restriction to the corresponding subcomplex of $ C_*( [0,1])^{\otimes n} $, and degeneracies act trivially because they correspond at the level of Kuranishi presentations to taking the product with cubes, and we are working with a multiplicative theory of virtual counts.

Finally, the compatibility of this functor with compositions follows from the axiom that a theory of virtual counts is a monoidal natural transformation of functors. Namely, we need to know that the diagram
\begin{equation}
  \begin{tikzcd}
    \begin{gathered}
      \Flow_{n}(\bX, \bY) \times \\
      \Flow_{m}(\bY, \bZ)   
    \end{gathered}
   \ar[r] \ar[d] &
   \begin{gathered}
     \Hom(C_*( [0,1])^{\otimes n} \otimes CF^*(\bX), CF^*(\bY)) \otimes \\
     \Hom(C_*( [0,1])^{\otimes m} \otimes CF^*(\bY), CF^*(\bZ))   
   \end{gathered}
  \ar[d] \\
 \Flow_{n+m}(\bX, \bZ)   \ar[r] & \Hom(C_*( [0,1])^{\otimes m+n} \otimes CF^*(\bX), CF^*(\bZ))
  \end{tikzcd}
\end{equation}
commutes. Since the composition maps respect the cubical structure, it suffices to check that, for each triple $(p,q,r)$ of objects of $\bX$, $\bY$ and $\bZ$, for each pair $\bM \in  \Flow_{n}(\bX, \bY)$ and $\bN \in \Flow_{m}(\bY, \bZ)$, and for each pair of labels $\mu$ and $\rho$, the diagram
\begin{equation}
   \begin{tikzcd}
     C_{m}( [0,1])^{\otimes m} \otimes  C_{n}( [0,1])^{\otimes n} \otimes \ro_{p} \ar[r] \ar[d] &    C_{m+n}( [0,1])^{\otimes m +n } \otimes  \ro_{p} \ar[d] \\
      C_{m}( [0,1])^{\otimes m} \otimes  \ro_q \ar[r] & \ro_r
  \end{tikzcd}
\end{equation}
commutes, where the left vertical map is associated to $\bM^{\mu}(p,q)$, the bottom horizontal map to $\bN^{\rho}(q,r) $, and the right vertical map to their product. This commutativity is immediate for the assumption that virtual counts are given by a monoidal natural transformation. \end{proof}

\section{The multicategory of flow categories and the Floer multifunctor}
\label{sec:mult-flow-categ}

The purpose of this section is to extend the constructions of the previous one to a multicategory whose objects are flow categories; this encodes the relevant multiplicative structures that appear in Floer theory. We shall abuse notation and use the same notation $\Flow$ for this multicategory.

\subsection{The multicategory of poset bimodules}
\label{sec:mult-poset-bimod}

We begin by extending the construction of Section \ref{sec:comp-morph}, using the multicategory $\Multimod^{\Gamma}$, whose objects are those of $\Cat^{\Gamma}$, and whose multi-morphisms are $\Cat^{\Gamma}$-enriched multimodules.

To introduce this multicategory explicitly, note that by specialising the discussion of Appendix \ref{sec:multimodules}, we find that a multimodule $\cM$ over a collection
\begin{equation}
\vec{\cQ} = (\cQ_1, \ldots, \cQ_k; \cQ)  
\end{equation}
of objects of $ \Cat^{\Gamma}$  consists of an object $\cM(\vec{Q})$ of $\Cat^{\Gamma}$ for each  collection
\begin{equation}
  \vec{q} = (q_1, \ldots, q_k; q) 
\end{equation}
of objects of $\cQ_i$, and commuting strict actions of the morphisms in $ \cQ_i$ on these categories, i.e. functors
\begin{align}
  \cM(\vec{q}) \times \cQ(q;q') & \to   \cM(\vec{q} \circ_1 (q;q')) \\
  \cQ_i(q'_i;q_i) \times   \cM(\vec{q}) & \to \cM((q'_i;q_i) \circ_i \vec{q}), \textrm{ for } 1 \leq i \leq k
\end{align}
where the sequences in the outputs of the two compositions are
\begin{align}
  \vec{q} \circ_1 (q;q') & \equiv  (q_1, \ldots,q_k; q') \\
  (q'_i;q_i) \circ_i \vec{q} & \equiv (q_1, \ldots, q_{i-1}, q'_i, q_{i+1}, \ldots, q_k; q).  
\end{align}
The multicategory $\Multimod^{\Gamma}$ has multimorphisms
\begin{equation}
 \Multimod^{\Gamma} \vec{\cQ}
\end{equation}
 described above. In order to describe compositions in this multicategory, consider a pair of sequences
\begin{align}
  \vec{\cQ}_{-} & = (\cQ_{-,1}, \ldots, \cQ_{-,k}; \cQ_{-}) \\
  \vec{\cQ}_{+} & = ( \cQ_{+,1}, \ldots, \cQ_{+,\ell}; \cQ_{+}).
\end{align}
Given multimodules $\cM_\pm$ over these sequences, and an integer $1 \leq i \leq k$  such that $\cQ_{-,i} = \cQ_{+}$, we have to describe a multimodule $ \cM_+ \circ_i \cM_-$ over the sequence 
    \begin{equation}
     \vec{\cQ}_{+} \circ_{i}  \vec{\cQ}_{-}  \equiv (\cQ_{-,1}, \ldots, \cQ_{-,i-1}, \cQ_{+,1}, \ldots, \cQ_{+,\ell}, \cQ_{-,i+1}, \ldots, \cQ_{-,k}; \cQ_{-,0}):
    \end{equation}

  \begin{defin} \label{def:composition_poset_multimodules}
    The \emph{composition $\cM_+ \circ_i \cM_{-}$}  assigns to a sequence
    \begin{equation}
      \vec{q} \in   \vec{\cQ}_{+} \circ_{i} \vec{\cQ}_{-}
    \end{equation}
and to an element $\tau$ of $\Gamma$ the coequaliser $ \left( \cM_+ \circ_i \cM_- \right)^{\tau}(\vec{q}) $ of the two maps of partially ordered sets
    \begin{equation}
      \coprod \cM^{\lambda}_{+}(\vec{q}_{+}) \times \cQ^{\mu}_{+}(q_+,q_+') \times \cM^{\rho}_{-}(\vec{q}_{-})  \rightrightarrows       \coprod \cM^{\lambda}_{+}(\vec{q}_{+}) \times  \cM^{\rho}_{-}(\vec{q}_{-})
      \end{equation}
      given by the action of  $\cQ$ on $\cM_{-}$ and $\cM_{+}$, where the disjoint union in the source is taken over all triples $\vec{q}_{+}$, $ (q_+,q_+')$, and $\vec{q}_{-}$ such that
      \begin{equation}
        \vec{q}_{+} \circ_1  (q_+,q_+') \circ_i \vec{q}_{-}  = \vec{q}
      \end{equation}
      and all triples $(\lambda, \mu, \rho)$, whose sum is $\tau$, and such that $\mu$ lies in $\Gamma_+$. The disjoint in the target is taken over all pairs  $\vec{q}_{+} $ and $\vec{q}_{-}$ such that
      \begin{equation}
        \vec{q}_{+} \circ_i  \vec{q}_{-} = \vec{q}
      \end{equation}
      as well as pairs $(\lambda, \rho)$ in $\Gamma$ whose sum is $\tau$.
       \end{defin}

   Noting that it is straightforward to construct an isomorphism
   \begin{equation}
  \Multimod^{\Gamma}\left( \cQ_1, \ldots, \cQ_k; \cQ \right) \cong  \Multimod^{\Gamma}  \left( \cQ_{\sigma(1)}, \ldots, \cQ_{\sigma(k)}; \cQ \right)
   \end{equation}
   by the relabelling associated to each permutation $\sigma$ of the set $\{1, \ldots, k\}$, which is compatible with compositions, we have:
   \begin{lem}
     The construction of Definition \ref{def:composition_poset_multimodules} is the multicomposition of a multicategory whose objects are those of $\Cat^{\Gamma}$, and whose multimorphisms are the sets $\Multimod^{\Gamma} \vec{\cQ}$.  \qed
   \end{lem}

\subsection{Multimodules of trees}
\label{sec:multimodules-trees}

There is a specific multimodule which will play a key r\^ole in our construction: let $\vec{\cP} = (\cP_1, \ldots, \cP_k; \cP)$ be a sequence of partially ordered sets as in Section \ref{sec:chain-compl-assoc-1}. 
Recall that we have constructed categories $ \cP_i^{\Gamma}$ with objects the elements of $\cP_i$. We shall construct a multimodule $ \cP_{\vec{\cP}} $ over these categories as follows: for each collection $\vec{p}  = (p_1, \ldots, p_k; p)$ of objects of these categories and each real number $\mu$, we introduce the category \begin{equation}
  \cP_{\vec{\cP}}^{\mu} (\vec{p}) \equiv   
    \coprod_{\substack{ p'_i \in \cP_i  \\ \mu^i \in \Gamma_+} } \bar{\cP}_1^{\mu_1}(p_1,p'_1) \times  \cdots \times \bar{\cP}_k^{\mu_k}(p_k,p'_k) \times \{   \mu - \sum_{i=0}^{k} \mu_i\} \times \bar{\cP}^{\mu_0}(p',p),
\end{equation}
where we recall that $ \bar{\cP}_i^{\mu_i}(p_i,p'_i)$ is the singleton $\{p_i\}$ whenever $\mu_i = 0$ and $p_i = p'_i$, and is otherwise given by $ \cP_i^{\mu_i}(p_i,p'_i)$. Since the element $  \mu - \sum_{i=0}^{k} \mu_i$ of $\Gamma$ is determined by the other data, we often omit it from our notation for elements of this set. We represent an element of $  \cP_{\vec{\cP}}^{\mu}(\vec{p}) $ graphically in Figure \ref{fig:tree-multiflow-category} as a tree, with edges labelled by elements of $\cP_i$ and vertices by elements of $\Gamma$, and in which all vertices are bivalent except for a single vertex of valence $k+1$.

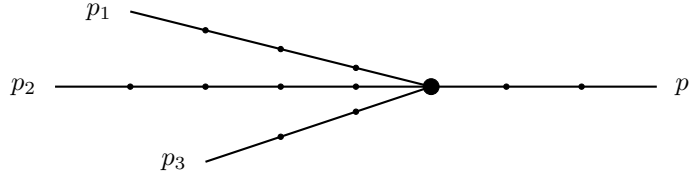
\begin{figure}[h]
  \centering
  \begin{tikzpicture}
    \draw[thick] (-1,0) -- (7,0);
    \node[label=left:{$p_{2}$}] at (-1,0) {};
    \node[label=right:{$p$}] at (7,0) {};
     \foreach \i in {0,...,6}{
     \ifthenelse{\i=4}{\filldraw (\i,0) circle (3pt)}{\filldraw (\i,0) circle (1pt)};
   };
   \node[label=left:{$p_{1}$}] at (0,) {};
   \draw[thick] (0,1) -- (4,0);
   \foreach \i in {1,...,3}{\pgfmathparse{1-\i/4}
      \edef\j{\pgfmathresult}
     \filldraw (\i,\j) circle (1pt);
   };
      \node[label=left:{$p_{3}$}] at (1,-1) {};
      \draw[thick] (1,-1) -- (4,0);
       \foreach \i in {2,...,3}{\pgfmathparse{1-(\i-1)/3}
      \edef\j{\pgfmathresult}
     \filldraw (\i,-\j) circle (1pt);
   };
  \end{tikzpicture}
  \caption{Graphical representation of an object of $  \cP_{\vec{\cP}}^{\mu} (\vec{p})$. The edges along the arc from the leaf labelled $p_i$ to the $4$-valent vertex are labelled by elements of $\cP_i$, and the outgoing leaf by elements of $\cP$.}
  \label{fig:tree-multiflow-category}
\end{figure}

This set inherits a partial order which corresponds in the graphical representation to collapsing a subtree, and defining the label in $\Gamma$ of the quotient vertex to be the sum of the labels of the vertices in the inverse image. Our conventions is that such a collapse map corresponds to an arrow from the tree corresponding to its target to the source. This is compatible with the codimension function being given by the number of internal edges. Since the geometric realisation of the partially ordered set of elements of $\cP_{\vec{\cP}}^{\mu} (\vec p) $ between $\vec{P} < \vec{P}' $ is a product of cubes, we conclude:
\begin{lem}
  The partially ordered set $\cP_{\vec{\cP}}^{\mu} (\vec p) $ is a model for manifolds with generalised corners. \qed
\end{lem}

We denote by $  \cP_{\vec{\cP}}^{\Gamma}( \vec{p})$ the object of $\Cat^{\Gamma}$ associated to these collections of partially ordered sets.
\begin{lem}
  The assignment
  \begin{equation}
    \vec{p} \to \cP_{\vec{\cP}}^{\Gamma}( \vec{p}) 
  \end{equation}
  defines a multi-module over the categories $\vec{\cP}$. \qed
\end{lem}

\begin{figure}[h]
  \centering
  \begin{tikzpicture}
    \draw[thick] (0,0) -- (4,0);
    \filldraw (2,0) circle (3pt);
    \node[label=right:{$\cP_+$}] at (-1,0) {};
    \node[label=right:{$\cP_{-}$}] at (4,0) {};
    \node[label=left:{$\cP_{-,1}$}] at (0,1) {};
   \draw[thick] (0,1) -- (2,0);
     \node[label=left:{$\cP_{-,3}$}] at (1,-1) {};
     \draw[thick] (1,-1) -- (2,0);
     \begin{scope}[shift={(-4,0)}]
         \filldraw (2,0) circle (3pt);
       \draw[thick] (0,1) -- (2,0);
       \draw[thick] (1,-1) -- (2,0);
         \draw[thick] (2,0) -- (3,0);
       \node[label=left:{$\cP_{+,1}$}] at (0,1) {};
     \node[label=left:{$\cP_{+,2}$}] at (1,-1) {};
   \end{scope}
     \begin{scope}[shift={(-2,-4)}]
         \draw[thick] (2,0) -- (4,0);
    \filldraw (2,0) circle (3pt);
    \node[label=right:{$\cP_{-}$}] at (4,0) {};
    \node[label=left:{$\cP_{-,1}$}] at (0,1) {};
    \draw[thick] (0,1) -- (2,0);
     \node[label=left:{$\cP_{-,3}$}] at (1,-1) {};
     \draw[thick] (1,-1) -- (2,0);
      \node[label=left:{$\cP_{+,2}$}] at (0,-.33) {};
      \draw[thick] (0,-.33) -- (2,0);
         \node[label=left:{$\cP_{+,1}$}] at (-1,.33) {};
     \draw[thick] (-1,.33) -- (2,0);
   \end{scope}
    \end{tikzpicture}
  \caption{Graphical representation of objects of $\cP_{\vec{\cP}_{+}}^{\Gamma} \circ_{2}   \cP_{\vec{\cP}_{-}}^{\Gamma}$, and those of  $\cP_{\vec{\cP}_{+} \circ_{2}  \vec{\cP}_{-}}^{\Gamma}$.}
  \label{fig:composition_tree-multiflow-category}
\end{figure}
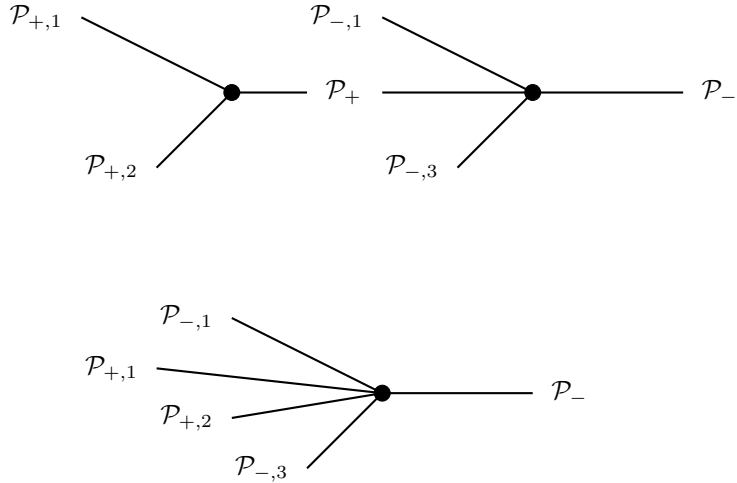

Next, we observe that, given a pair of sequences
\begin{align}
  \vec{\cP}_{-} & = (\cP_{-,1}, \ldots, \cP_{-,k}; \cP_{-}) \\
  \vec{\cP}_{+} & = ( \cP_{+,1}, \ldots, \cP_{+,\ell}; \cP_{+}).
\end{align}
and an integer $1 \leq i \leq k$ each $i$ such that $\cP_{-,i} = \cP_{+}$, we have a natural map of multimodules
\begin{equation} \label{eq:multi-composition-category-of-trees}
   \cP_{\vec{\cP}_{+}}^{\Gamma} \circ_i   \cP_{\vec{\cP}_{-}}^{\Gamma} \to \cP_{\vec{\cP}_{+} \circ_{i}  \vec{\cP}_{-}}^{\Gamma}
\end{equation}
which corresponds to collapsing the arc consisting of edges labelled by elements of $\cP_{+}$ (see Figure \ref{fig:composition_tree-multiflow-category}). This map is associative in the sense that it defines a multicategory (enriched in $\Cat^\Gamma$) with objects partially ordered sets, multimorphisms given by the multimodules $ \cP_{ \vec{\cP}}^{\Gamma}$, and multi-compositions given by Equation \eqref{eq:multi-composition-category-of-trees}. These data satisfy the axioms given in Definition \ref{def:multicategory_enriched_in_symmetric_monoidal}, the most important of which is that the functors associated to grafting trees are uniquely defined.

\begin{rem}
  The reader who is looking ahead to the construction of $A_\infty$ structures in Lagrangian Floer theory may be surprised at our construction of $ \cP_{ \vec{\cP}}^{\Gamma} $, and may have expected to find a (larger) category, whose objects are given by labelled trees with vertices of arbitrary valency greater than $1$. If our goal was specifically to construct $A_\infty$ structures, that would likely have been the procedure we followed, with the idea of constructing charts for the moduli spaces of pseudo-holomorphic discs of arbitrary modulus. Instead, the approach that we intend to take produces an $A_\infty$ structure by considering stable moduli spaces of maps parametrised by a cube mapping to the moduli space of abstract discs. This essentially exhibits Lagrangian Floer complexes as modules over the operad of chains on the Stasheff operad. The main advantage of this approach is that it is better adapted to producing more general operadic structures, as in the joint work \cite{AbouzaidGromanVarolgunes2021}.
\end{rem}

\subsection{Multimodules of flow categories}
\label{sec:mult-flow-categ-1}

Given a collection $\vec{\bX} = (\bX_1, \ldots, \bX_k; \bX) $ of flow categories, we shall explain the construction of a cubical set of multimorphisms with domain $(\bX_1, \ldots, \bX_k) $ and target $\bX$. For the next definition, we introduce the notation
\begin{equation}
  \ro_{\vec{p}} \equiv \ro_{p_1} \otimes \cdots \otimes \ro_{p_k} \otimes \ro^{-1}_{p}
\end{equation}
for the orientation line associated to a sequence consisting of objects of these categories, and we write
\begin{equation}
    \vec{\cP}^{\Gamma_+}_{\bX} \equiv \left(\cP_{\bX_1}^{\Gamma_+}, \cdots, \cP^{\Gamma_+}_{\bX_k}; \cP^{\Gamma_+}_{\bX}\right),
  \end{equation}
  for the objects of $\Cat^{\Gamma}$ associated to the sequence $\vec{\bX}$ of flow categories.
\begin{defin} \label{def:Kuranishi-n-multimodule}
  A \emph{Kuranishi $n$ multimodule} over $\vec{\bX}$ consists of the following data
   \begin{enumerate}
   \item For each permutation $\phi$ of $\{1, \ldots, n\}$, a multimodule $\bM_{\phi}$ enriched in $\cD \Kur^{\Gamma}$  over $ \vec{\bX} $,  with underlying morphism of $\Cat^{\Gamma}$-enriched multimodules with respect to the sequence  $\cP_{\bX_i}^{\Gamma_+}$
     \begin{equation}
\pi_\phi \co              \cQ_{\bM,\phi} \to \cP_{\vec{\bX} }^{\Gamma}  \times \Face  \square^{n}.
     \end{equation}
   \item For each composition $\phi_1 \circ \phi_2$ of permutations, an isomorphism
     \begin{equation}\label{eq:choice_isomorphism_cube-Kur}
        {\bM}_{\phi_1  \circ \phi_2} \cong \phi_2^* {\bM}_{\phi_1}      
     \end{equation}
     where the right hand side is given at the level of stratifying categories by the pullback in the diagram
     \begin{equation}
       \begin{tikzcd}
     \phi_2^* \cQ_{{\bM},\phi_1} \ar[r] \ar[d] &   \cQ_{{\bM},\phi_1} \ar[d] \\
       \cP_{\vec{\bX}}^{\Gamma}  \times \Face  \square^{n} \ar[r, "\id \times \phi_2"]  & \cP_{\vec{\bX}}^{\Gamma}  \times \Face  \square^{n}
       \end{tikzcd}
     \end{equation}
     and at the level of diagrams of presentations by the composite functor.
   \item For each triple composition $\phi_3 \circ \phi_2 \circ \phi_1$, we require that the isomorphisms of pullbacks fit in a commutative diagram
     \begin{equation}
       \begin{tikzcd}
         {\bM}_{\phi_1  \circ \phi_2 \circ \phi_3} \ar[r] \ar[d] & \phi_3^{*}  {\bM}_{\phi_1 \circ \phi_2} \ar[d] \\
     \left(\phi_2 \circ    \phi_3 \right)^* {\bM}_{\phi_1}  \ar[r]   & \phi_3^{*} \left( \phi_2^* {\bM}_{\phi_1} \right).
       \end{tikzcd}
     \end{equation}
\end{enumerate}
We assume that the Kuranishi diagram $\bM^{\mu}(\vec{p})$ is oriented relative $\ro_{\vec{p}} \otimes \ro_{[0,1]^n}$, and the maps of orientations lines associated to the three cases in Lemma \ref{lem:codim_1-elt-bimodules}:
\begin{enumerate}
\item are compatible with the identification of the normal directions of the cube in the case of boundary strata of $\square^n$,
\item are induced by the product of the identity on $\ro_{[0,1]^n} $ and on $\ro_{p_j}$ for $j\neq i$ with the isomorphism
  \begin{equation}
    \ro_{p_i} \otimes \ro^{-1}_{p'_i} \otimes \ro^{-1}_{\bR} \otimes  \ro_{p'_i}  \cong    \ro_{p_i}  \otimes \ro^{-1}_{\bR^{p'_i }}
  \end{equation}
  where we use the standard orientation on $\bR^{p'_i}$, in the case of the image of the map
  \begin{equation}
  \bX_{i}(p_i, p'_i) \times \bM( \vec{p}) \to \bM((p_i, p'_i) \circ_i \vec{p}).
  \end{equation}
\item  are  induced by the product of the identity on $\ro_{[0,1]^n} $ and on $\ro_{p_i}$ with the isomorphism
  \begin{equation}
    \ro^{-1}_p \otimes \ro_{p'} \otimes \ro^{-1}_{\bR} \otimes  \ro^{-1}_{p'} \cong  \ro^{-1}_{\bR^{p'}} \otimes \ro^{-1}_{p'}
  \end{equation}
  where we use the opposite of the standard orientation on $\bR^{p'}$, in the case of the image of the map
  \begin{equation}
  \bM(\vec{p}) \times \bX(p, p') \to \bM(\vec{p} \circ_{i} (p, p')) .
  \end{equation}
 
  \end{enumerate}
\end{defin}

We write $\flow_{n}(\vec{\bX})$ for the set of Kuranishi flow $n$-modules associated to $\vec{\bX}$. In the same way as in Section \ref{sec:comp-morph}, these form the objects of a category which we denote by the same symbol, with morphisms given by maps of $\cD \Kur^{\Gamma}$-enriched multimodules for each permutation of $\{1, \ldots, n\}$, whose underlying maps of $\Cat^{\Gamma}$ enriched bimodules are given by the identity on $  \cP_{\vec{\bX} }^{\Gamma}  \times \Face  \square^{n}$, and which commute with the isomorphisms associated to compositions of permutations.

These categories are themselves multimorphisms in a multicategory enriched in symmetric semi-cubical categories: first, we associate to each face map $\square^k \to \square^n$ a functor
\begin{equation}
  \flow_{n}(\vec{\bX}) \to  \flow_{k}(\vec{\bX})
\end{equation}
defined at the level of parametrising categories by taking the inverse image in $\cQ_{\bM}$ of the corresponding subcategory of $\Face \square^{n}$, and at the level of Kuranishi diagrams by restriction $\bM$ to this subcategory. The fact that we take the inverse image (rather than the fibre product), ensures that the diagrams associated to compositions are strictly commutative (rather than up to a natural isomorphism).

Next, we note that each permutation $f$ of the inputs induces a natural isomorphism
\begin{equation}
    \flow_{n}(\vec{\bX}) \cong \flow_{n}( f_* \vec{\bX})
\end{equation}
given by defining
\begin{equation}
 \cQ^{\lambda}_{f \bM_\phi}  \equiv  \cQ^{\lambda}_{\bM_\phi },
\end{equation}
and setting $f \bM_{\phi} $  to be the composite of $\bM_{\phi}$ with this isomorphism. These morphisms are readily seen to be compatible with the symmetric semi-cubical structure maps. 

Finally, given sequences $\vec{\bX}_+$ and $\vec{\bX}_-$, with corresponding $n_{\pm}$-multimodules $\bM_{\pm}$, we define a composition $n_+ + n_-$ multimodule $\bM_{+} \circ_i \bM_-$, whenever $\bX_{+} = \bX_{-,i} $. At the level of strata, we associate to a permutation $\phi$ of $\{1, \ldots, n_+ + n_-\}$ the $\Cat^{\Gamma}$-enriched multimodule
  \begin{equation}
      \cQ_{\left(\bM_{+} \circ_i \bM_{-}\right)_\phi} \equiv  \cQ_{\bM_{+},\phi_+} \circ_i  \cQ_{\bM_{-},\phi_-},
  \end{equation}
  where $\phi_{\pm}$ are the permutations of $\{1, \ldots, n_\pm\}$ associated to $\phi$, and the left hand side is the multimodule over $ \vec{\cP}^{\Gamma_+}_{\bX} $  which results from Definition \ref{def:composition_poset_multimodules}.  We then define \begin{equation}
     \left(\bM_{+} \circ_i \bM_-\right)_{\phi} \co    \cQ_{\left(\bM_{+} \circ_i \bM_{-}\right)_\phi} \to \Kur  
  \end{equation}
  to be the composite of this isomorphism with the functor the assigns to a pair $(\vec{p}_+, \vec{p}_-)$ of composable sequences, and a pair of objects
  \begin{equation}
   (P_{+},P_{-}) \in \cQ_{\bM_{+}, \phi_+}^{\lambda}(\vec{p}_+) \times \cQ_{\bM_-}^{\rho}(\vec{p}_-)
  \end{equation}
  the product Kuranishi presentation
  \begin{equation}
  \partial^{P_{+}}  \bM_{+}^{\lambda}(\vec{p}_+) \times  \partial^{P_{-}} \bM_{-}^{\rho}(\vec{p}_-).
\end{equation}
As in Section \ref{sec:comp-morph}, this multi-composition is associative up to a natural isomorphism, and the compositions of the natural isomorphisms associated to different factorisation strictly agree, in the sense that we have a symmetric semi-cubically enriched functor
\begin{equation}
 \underset{T}{\bigcirc}  \co \prod_{v \in V(T) }  \widetilde{\flow}(\vec{\bX}_v) \to     \widetilde{\flow}(\vec{\bX})
\end{equation}
associated to every directed tree $T$ with edges labelled by flow categories, with $\vec{\bX}_v $ the sequence of flow categories given by the edges adjacent to each vertex, and $\vec{\bX} $ the sequence associated to the outgoing leaves. The functor $ \underset{T}{\bigcirc}  $ is defined up to canonical isomorphism, and a representative of its isomorphism class  is determined by a total ordering of the vertices which refines the partial order associated to the tree and a compatible parenthesisation of the resulting sequence.
   \begin{lem} \label{lem:multi-category-flow}
     The multicomposition map $\bM \circ_i \bN$, and the natural isomorphisms
     \begin{align}
       \left( \bM \circ_i \bN\right) \circ_j \bO  & \cong \bM \circ_{i+j-1} \left( \bN \circ_j \bO \right)   \\
        \bM_{1} \circ_{i_1 + m_2-1} \left( \bM_2 \circ_{i_2} \bN \right)  & \cong \bM_{2} \circ_{i_2} \left( \bM_1 \circ_{i_1} \bN \right) 
     \end{align}
     equip $\flow$ with the structure of a multicategory enriched in symmetric semi-cubical categories. \qed
   \end{lem}

\subsection{Flow multimodules with factorisations, and the Floer multifunctor}
\label{sec:flow-mult-with}

 Continuing in analogy with the constructions of Section \ref{sec:flow-categ-chain}, we now extract from the multicategory $\flow$, which is enriched in categories, a multicategory enriched in sets:
  \begin{defin}\label{def:Flow_multicategory}
    The set $\Flow_{n}(\vec{\bX})$ associated to a sequence $\vec{\bX}$ of flow categories consists of the following data:
     \begin{enumerate}
     \item For each face $\sigma$ of the $n$-cube, a directed tree $T_{\sigma}$ with edges labelled by flow categories and vertices by flow multimodules as above, an identification of the labels of the leaves of $T$ with the sequence $\vec{\bX}$, and a disjoint collections of subsets $D_v$ of $\{1, \ldots, n\}$, indexed by the vertices of $T_{\sigma}$, so that the number of elements of $D_v$ agrees with the dimension of $\bM_v$.
     \item For each pair $\sigma < \tau$ in $\Face\square^n$, a map $\pi_{\sigma}^{\tau}$ of trees $T_{\tau} \to T_{\sigma}$, which preserves the labels of all edges which are not collapsed,  and so that $D_u$ maps to $D_{\pi_{\sigma}^{\tau}(u)}$ for each vertex $u$ of $T_{\tau}$.
       \item For each vertex $v$ of $T_{\sigma}$ a map
       \begin{equation} \label{eq:map_decomposed_strata_multimodule}
[0,1]^{e_{\sigma}^{\tau}(v)} \times \underset{u \to T_{\tau}^v}{\bigcirc} \bM_{u} \to \bM_v,            
\end{equation}
where $T_{\tau}^{v} $ denotes the inverse image of this tree in $T_{\tau}$, which is an isomorphism onto the boundary strata of $\bM_v$ associated to $\tau$ (the integer $e_{\sigma}^{\tau}(v)$ is determined by the difference in dimensions).
   \end{enumerate}
        We require these data to be compatible in the sense that, for each triple $\sigma < \tau < \rho$, the diagrams associated to the maps in Equation~\eqref{eq:map_decomposed_strata_multimodule} commute.
      \end{defin}

      The cubical structure maps are defined in the same way as for the case in which there is only one input (c.f. Section \ref{sec:strict-categ-flow}): face maps are given by restriction to the associated face of the $n$-cube, permutations act on the choices of subsets $D_v$ of $\{1, \ldots, n\}$ and on the multimodules $\bM_v$, and degeneracies are defined formally by associating to a cube the data labelled by its image. We also have a natural isomorphism
      \begin{equation}
            \Flow_{n}(\vec{\bX}) \cong  \Flow_{n}(f^*\vec{\bX}),  
      \end{equation}
      whenever $f$ is a permutation of the input flow categories, given by composing $f$ with the labelling of the leaves.

      Whenever $\vec{\bX}_+$ and $\vec{\bX}_-$ are sequences of flow categories with the property that $\bX_{+} = \bX_{-,i} $, we now define a composition map
      \begin{equation} \label{eq:multi-composition-Flow}
             \Flow_{n}(\vec{\bX}_+) \times     \Flow_{m}(\vec{\bX}_-) \to \Flow_{n+m}( \vec{\bX}_+ \circ_i \vec{\bX}_-)
      \end{equation}
      by concatenating labelled trees: every stratum of $\square^{n+m}$ may be uniquely written as the product $\sigma_+ \times \sigma_-$ of strata of $\square^n$ and $\square^m$, and the tree $T_{\sigma}$ is defined to be the union of the trees $T_{\sigma_+}$ and $T_{\sigma_-}$ along the root of $T_{\sigma_+}$, and the leaf of $T_{\sigma_-}$ labelled by $ \bX_{-,i}$. The edges and vertices of $T_{\sigma}$ then carry the labels and operations induced from those of the edges and vertices of $T_{\sigma_{\pm}}$.

      \begin{lem}
        \label{lem:Flow_multicat-well-defined}
        The multicomposition in Equation \eqref{eq:multi-composition-Flow} equips $\Flow$ with the structure of a multicategory enriched in cubical sets. \qed
      \end{lem}

We complete this section by showing that Proposition \ref{prop:floer-functor} is compatible with multiplicative structures:
\begin{prop}
   A theory of virtual counts determines a cubically enriched multifunctor
  \begin{equation}
  CF^* \co  \Flow \to \Ch_{\Lambda}
  \end{equation}
  and a lift to $\Ch_{\Lambda_0}$ of the restriction to $\Flow^+$.
\end{prop}
\begin{proof}[Sketch of proof]
  The key remaining ingredient are the results of Appendix \ref{sec:lifting-flow-modules}, in which we construct a multicategory $\widetilde{\Flow}$, in analogy with $\Flow$, but consisting of diagrams valued in $\widetilde{\Kur}$, and where we establish the existence of a cubically enriched multifunctor $\Flow \to \widetilde{\Flow}$ (c.f. Proposition \ref{prop:lift_flow_to_tflow}). Using this lift, we assign a map of degree $-n$
    \begin{equation}  \label{eq:structure_map_multimodule}
 CF^*(\bX_1) \otimes  CF^*(\bX_2) \otimes  \cdots \otimes  CF^*(\bX_k)  \to CF^*(\bX).
  \end{equation}
  for each $\bM \in \Flow_n(\vec{\bX})$, as follows: for a sequence $\vec{p} = (p_1, \ldots, p_k; p)$ of objects of these categories, and an element $\lambda$ of $\Gamma$, the sum of the maps from the orientation lines of the top strata of $\bM^{\lambda}(\vec{p})$ defines a map
  \begin{equation}
\ro_{[0,1]^n} \otimes \ro_{p_1} \otimes \cdots \otimes \ro_{p_k} \otimes \ro_p,
  \end{equation}
because we have assumed $\bM$ to be oriented relative the corresponding graded line. Equation \eqref{eq:structure_map_multimodule} is then obtained by trivialising the $\ro_{[0,1]^n} $ factor, and taking the sum of these contributions, weighted by $T^{\lambda}$. 
\end{proof}

\part{Application to Hamiltonian Floer theory}
\label{part:appl-hamilt-floer}

Our goal in this part is to apply the formalism developed in the previous parts to Hamiltonian Floer theory. The treatment we provide follows the approach taken in \cite{AbouzaidGromanVarolgunes2021}, and is designed to be directly applicable for the study of local Floer homology pursued therein.
\begin{rem}
We expect that the recent work of Bai-Xu \cite{BaiXu2022} and Rezchikov \cite{Rezchikov2022} will soon be adapted to reprove the results of this part, using geometrically constructed global charts.
\end{rem}
\section{The multicategory of Hamiltonian Floer data}
\label{sec:mult-hamilt}

We briefly recall, from \cite{AbouzaidGromanVarolgunes2021}, the construction of the Hamiltonian Floer multicategory $\cF$ associated to a closed symplectic manifold $M$. Its objects are pairs $(J,H)$, consisting of a tame almost complex structure and a non-degenerate Hamiltonian
\begin{equation}
  H \co S^1 \times M \to \bR.  
\end{equation}
The multimorphisms are symmetric cubical sets of Hamiltonian Floer data, i.e. for each sequence of such pairs $\overrightarrow{(J,H)}$, and each natural number $n$, we have a set $\cF_n(\overrightarrow{(J,H)})$, whose elements are given by the following data:
\begin{itemize}
\item (Labelled trees) A functor $T$ from $\Face \square^n$ to the category of directed trees whose edges are labelled by objects of $\cF$, with an identification of the labels of the leaves with $\overrightarrow{(J,H)}$ (the morphism are given by collapsing subtrees).
\item (Stable Riemann surfaces) A smooth family $\{\Sigma_v\}_{v \in T_\sigma}$ of pre-stable Riemann surfaces, parametrised by the interior of each stratum $\sigma$ of the cube, of topological type $T_{\sigma}$ (i.e. for each vertex $v$, a smooth family $\Sigma_v$ of genus $0$ Riemann surfaces with punctured labelled by the edges adjacent to $v$).
\item (Floer data)  An equivalence class of smooth choices, in the interior of each stratum $\sigma$ and for each $v \in T$, of a triple $(J_v,H_v, \alpha_v)$ consist of a family $J_v$ of almost complex structures parametrised by $\Sigma_v$, of a closed $1$-form $\alpha_v$ on $\Sigma_v$, and of a function
  \begin{equation}
        H_v \co \Sigma_V \times M \to \bR.
      \end{equation}
      The equivalence class is the quotient by the $\bR^*$ action
\begin{equation}
\lambda \cdot (J_v,H_v, \alpha_v)  =  (J_v,\lambda H_v, \lambda^{-1} \alpha_v).    
\end{equation}
\end{itemize}
We require to the choice of surfaces $\Sigma_v$ and Floer data $(H_v,\alpha_v)$ to be obtained by gluing near each stratum of the cube, and include the datum of an atlas, presenting this data in a neighbourhood of each stratum of the cube as the outcome of gluing. More precisely, we associate to each stratum $\sigma$ of the cube $[0,1]^n$ the neighbourhood of its interior $\nu \Int{\sigma}$ obtained by taking the product with intervals of length $1/3$ in all normal directions (any constant smaller than $1/2$ would be adequate).
\begin{itemize}
\item (Gluing atlas) For each stratum $\sigma$, a smooth extension of the families $\{\Sigma_v\}_{v \in T_\sigma}$ of Riemann surfaces and Floer data $(H_v,\alpha_v)$ to $\nu \Int{\sigma}$, together with a choice of gluing parameters yielding the families associated to the adjacent strata.
\end{itemize}
We further require that the gluing data associated to adjacent edges are compatible in the sense that the composition of gluing maps associated to a pair of nested strata agree.

\begin{rem}
  There are two minor differences between our notation and that of \cite{AbouzaidGromanVarolgunes2021}: (i) we shall not exclude bivalent vertices which are labelled by continuation maps for which the Hamiltonian does not change, and in particular we allow the constant continuation map, and (ii) we omit the additional perturbation datum used in \cite{AbouzaidGromanVarolgunes2021} to achieve transversality in the symplectically monotone setting.
\end{rem}
\begin{rem}
One can generalise the above construction to symplectic manifolds which are geometrically bounded in the sense of \cite{Groman2015}. In that case, we restrict attention to those Hamiltonian data satisfying the notion of dissipativity described in \cite[Appendix C]{AbouzaidGromanVarolgunes2021}. These conditions are required to ensure compactness of moduli spaces of holomorphic curves of bounded energy, whose proof is completely independent of the considerations of our construction.
\end{rem}
It is straightforward to equip the collections of sets $\{\cF_n(\overrightarrow{(J,H)})\}_{n=0}^{\infty} $ with the structure of a symmetric cubical set, to construct reordering isomorphisms of symmetric cubical sets
\begin{equation}
  \cF(\overrightarrow{(J,H)}) \cong \cF(f \overrightarrow{(J,H)})  
\end{equation}
for each permutation $f$ of the inputs, and to define (multi)-composition maps (given by concatenating trees), which are appropriately associative.

The main goal of the following section is to prove the following result:
\begin{thm}
  \label{thm:Kuranishi-Hamiltonians-Floer}
 There is a multifunctor
  \begin{equation}
    \bX \co \cF \to \Flow
  \end{equation}
  which assigns to each pair $(J,H)$ a Kuranishi flow category with objects the set of time-$1$ Hamiltonian orbits of $H$, whose restriction to monotone homotopies factors through $\Flow^+$, and so that the following normalisation condition holds:
  \begin{equation}
    \label{eq:normalisation_Ham-Floer}
    \parbox{31em}{the flow bimodule associated to the constant continuation map is given, for trivial action, by the trivial Kuranishi presentation on the point for each orbit.}
  \end{equation}
\end{thm}
\begin{rem}
It will be clear from the construction that the multifunctor $\bX$ lifts the multifunctor $\cM$, constructed in \cite{AbouzaidGromanVarolgunes2021}, which assigns to each Hamiltonian a topological flow category, and to each cube of Floer data a multimodule over the cube.
\end{rem}
Composing $\bX$ with the multifunctor $\Flow \to \Ch$ associated in Section \ref{sec:flow-mult-with} to a theory of virtual counts over a ring $\Bbbk$, we obtain a functor
\begin{equation}
  CF^* \co \cF \to \Ch_{\Bbbk},
\end{equation}
which we call the \emph{Floer functor}. We now state a useful consequence of Theorem \ref{thm:Kuranishi-Hamiltonians-Floer}:
\begin{cor} \label{cor:Hamiltonian-CF-well-defined}
  If $M$ is compact, then the chain complex $CF^*(J,H)$ associated by the Floer functor to each non-degenerate Hamiltonian $H$ and almost complex structure $J$ is independent, up to contractible choice, of this choice in the sense that the functor $CF^*$ is naturally homotopy equivalent  to a constant functor.
\end{cor}
\begin{proof}
  Since the space of morphisms between any two objects in $\cF$ is contractible, it suffices to show that every morphism induces a homotopy equivalence, and that the morphism induced by endomorphisms is homotopic to the identity (this second check is required because of the lack of unitality). Starting with endomorphisms, the statement  of Theorem \ref{thm:Kuranishi-Hamiltonians-Floer} implies that the map
  \begin{equation}
  f \co    CF^*(J,H) \to CF^*(J,H)
  \end{equation}
  associated to the constant continuation induces the identity modulo the maximal ideal of $\Lambda_0$.  Proceeding by induction on energy, we conclude the existence of a homotopy inverse. On the other hand, since the result of gluing the constant continuation datum to itself is again constant, we obtain from the Floer functor a homotopy from $f^2$ to $f$; composing with the homotopy inverse of $f$ yields a homotopy from $f$ to the identity.
  
  We now claim that each continuation map from $(J_0,H_0)$ to $(J_1,H_1)$ induces a chain homotopy equivalence
  \begin{equation}
    CF^*(J_0,H_0) \to CF^*(J_1,H_1).
  \end{equation}
  To see this, we consider the continuation map in the opposite direction, and the canonical homotopy from the composite to the constant homotopy on $(J_0,H_0)$ provides a left inverse up to homotopy. Considering the homotopy from the composite in the other direction to the constant homotopy on $(J_1,H_1)$ provides a right inverse. The fact that the space of morphisms in $\cF$ is contractible completes the proof.
\end{proof}

\section{The flow category associated to a Hamiltonian}
\label{sec:flow-categ-assoc}

In order to prove Theorem \ref{thm:Kuranishi-Hamiltonians-Floer}, we begin with the construction at the level of objects, i.e. we assign to each pair $(J,H)$ consisting of an almost complex structure and a non-degenerate Hamiltonian $H$ a flow category $\bX(J,H)$, with object set $\cP$ given by the set of time-$1$ Hamiltonian chords of $H$. The discussion that we present is a version of the one given in \cite{AbouzaidBlumberg2021}, with simplifications due to the fact that, when studying ordinary rather than generalised Floer homology, we need not record information about the tangential structure of the moduli spaces.

\subsection{Categories of thickening data}
\label{sec:categ-thick-data}

Our proximate goal is to define the indexing category $A^{\lambda}(p,p')$ for the Kuranishi presentation of the moduli space of Floer trajectories of action $\lambda$ connecting the time-$1$ orbits $p$ and $p'$ of $H$. To this end, we recall that we introduced a partially ordered set  $ \cP^{\lambda}(p,p')$ in Definition \ref{def:poset-flow-category}, consisting of sequences of elements of $\cP$ starting at $p$ and ending at $p'$ together with a choice of elements of $\bR_+$ for each pair of successive elements. We write $T_{P}$ for the labelled arc associated to each element $P$ of $ \cP^{\lambda}(p,p')$.

We shall need to refer to an appropriate moduli space of abstract curves in the construction of each object of $A^{\lambda}(p,p')$: such a moduli space will be determined by (i) an arrow $P \to P'$ in $ \cP^{\lambda}(p,p')$, (ii) a natural number $S_T$ for each sub-arc $T$ of  $T_{P'}$, and (iii) a sequence $\vec{r}^{T}$ of strictly positive integers of length $S_T$ for each such sub-arc. Given this data, we let $\cM_{\{S_T\},\{\vec{r}_T\}, P \to P'}$ denote the moduli space of stable genus $0$ Riemann surfaces with marked points labelled by the union of the set $\{p,p'\}$ with the union over all $T$ of the sets $\{\vec{r}_T\}$, equipped with the following data:
\begin{enumerate} 
\item a labelling by an element of $\cP^{\lambda}(p,p') $ lying between $P$ and $P'$ in the sense that the corresponding arc is identified with the arc of components connecting the marked points labelled by $p$ and $p'$.
    \item a framing of the marked points $p$ and $p'$ and of the nodes along the path between them.
\end{enumerate}
We require that the framings be compatible in the following sense: each component along the path from $p$ to $p'$ acquires two identifications with $(\bP^{1},0,\infty)$ from the framings of the nodes, which is given up to positive dilation, and we require that these identifications agree. We require as well that the following condition hold:
\begin{equation} \label{eq:all-bubble-stabilised}
  \parbox{31em}{the marked points indexed by the sequence $\vec{r}_{T}$ lie in the components associated to $T$, and the map obtained by forgetting all but one of the elements of this sequence does not collapse any component.}
\end{equation}
More precisely, if the given Riemann surface is labelled by an element $Q$ of $\cP^{\lambda}(p,p') $, then the arrow $Q \to P'$ corresponds to a collapse map from  $T_{P'}$ to $T_{Q}$, and the arc $T$ then determines a collection of components, in which we require the given marked points to lie. This moduli space admits a natural action by the product of the symmetric groups on  $r_{T,i}$, given by relabelling marked points.

\begin{defin} \label{def:thickening-data}
  A  choice $\alpha$  of \emph{thickening data}  for the moduli space $\cM^{\lambda}(p,p')$ consists of
  \begin{enumerate}
  \item (Choice of stratum) A morphism $P_\alpha \to P'_\alpha $ in $ \cP^{\lambda}(p,p')$. We write $\partial^{\alpha} \cP^{\lambda}(p,p') $ for the partially ordered set of objects between $P_\alpha$ and $P'_\alpha$, and $T_{\alpha}$ for the labelled arc $T_{P'_\alpha}$.
  \item (Additional marked points) For each sub-arc $T$ of $T_{\alpha}$, a natural number $S^{\alpha}_T$, and a sequence $\vec{r}_{T}^{\alpha}$ of length $S_T^\alpha$, consisting of strictly positive integers. We write $\vec{r}^{\alpha}$ for the union of the sequence $\vec{r}_T^{\alpha}$ over all arcs $T$,  $G_\alpha$ for the product of the symmetric groups on $r^\alpha_{i}$ elements, $\cM_{\alpha}$ for the moduli space associated to this data, and $\cC_{\alpha}$ for the universal curve over $\cM_{\alpha}$.
    \item (Stabilising divisors) A sequence $\vec{D}^{\alpha}_{T}=(D^\alpha_{T,1}, \ldots, D^{\alpha}_{T,S^\alpha_T})$ of smooth codimension $2$ submanifolds of $M$ for each arc $T \subset T_{\alpha}$. We write $\vec{D}^{\alpha}$ for the union of these sequences.
    \item (Inhomogenous term) A $G_\alpha$-invariant finite dimensional subspace
      \begin{equation}
      V_\alpha \subset \bR^{\infty} \otimes \bR[G_{\alpha}]           
      \end{equation}
       and an equivariant map
    \begin{equation}
Y_\alpha \co      V_\alpha \to \Omega^{0,1}_{vert}(\cC_{\alpha}, C^{\infty}(M,TM))
    \end{equation}
    into the space of $(0,1)$-forms on the fibres of the universal curve, with value in the space of vector fields on $M$, with support away from all nodes, and which is obtained by gluing near broken curves.
  \end{enumerate}
\end{defin}
We define a map $\alpha \to \beta$ of thickening data to consist first of all of a factorisation
\begin{equation} \label{eq:inclusion-strata}
P_{\beta} \to P_\alpha \to P'_\alpha \to P'_{\beta}.
\end{equation}
The map $P'_{\alpha} \to P'_\beta$ corresponds to a map $T_{\beta} \to T_{\alpha}$, so that each sub-arc of $T_{\beta}$ projects to a sub-arc of $T_{\alpha}$.  We then require the data of an embedding
  \begin{equation}
    \coprod_{\pi(T') = T }  \{1, \ldots, S^\beta_{T'}\} \subset  \{1, \ldots, S^\alpha_T\},
  \end{equation}
which induces embeddings of sequences of natural numbers and divisors
\begin{align} \label{eq:inclusion_marked_points}
\vec{r}^{\beta} & \subset \vec{r}^{\alpha} \\ \label{eq:inclusion-divisors}
  \vec{D}^{\beta} & \subset \vec{D}^{\alpha}.
\end{align}
To formulate the remaining data, note that these maps induce a surjection $G_\alpha \to G_\beta$ of groups, and a $G_\alpha$-equivariant map
\begin{equation}
    \cM_\alpha \to \cM_\beta
\end{equation}
of moduli spaces, defined by forgetting the marked points which are not in the image of the chosen map in Equation \eqref{eq:inclusion_marked_points}, and that Equation \eqref{eq:inclusion-strata} implies that this forgetful map is compatible with our constraints on the labelling of the moduli spaces. Condition \eqref{eq:all-bubble-stabilised} moreover implies that the induced map of universal curves does not collapse any component, and is a submersion.

The final piece of data in the definition of a map of  thickening data is a  $G_\alpha$-equivariant inclusion
\begin{equation}
    V_\alpha \hookrightarrow V_\beta,
\end{equation}
fitting in a commutative diagram
    \begin{equation} \label{eq:commutative_diagram_map_deformations}
      \begin{tikzcd}
V_\beta \ar[r]  &         \Omega^{0,1}_{vert}(\cC_{\beta}, C^{\infty}(M,TM)) \ar[d]  \\
  V_\alpha \ar[r] \ar[u] &    \Omega^{0,1}_{vert}(\cC_{\alpha}, C^{\infty}(M,TM))
      \end{tikzcd}
    \end{equation}
    in which the right vertical map is given by pullback.

    It is straightforward to define composition of maps of thickening data, yielding a category $A^{\lambda}(p,p')$.

\subsection{Morphisms the flow category associated to Hamiltonians}
\label{sec:morph-hamilt-flow}

We now assign to each object of $ A^{\lambda}(p,p')$ a Kuranishi chart as follows: we begin by defining $X_{\alpha}(p,p')$ to be the set of pairs $(u,v)$, with $v$ an element of $V_\alpha$, and  $u$ a map to $M$ with domain $\Sigma$ a fibre of the universal curve $\cC_\alpha$ satisfying the deformed Floer equation
\begin{equation} \label{eq:deformed_Floer}
\left( du - X_H \otimes dt \right)^{0,1}  = Y_\alpha(v),
\end{equation}
with asymptotic conditions along each node of $\Sigma$ along the path from $p$ to $p'$ given by the corresponding element of $\cP$, and so that the topological energy of each component along this path (including all attached sphere bubbles) is given by the corresponding positive real number. In addition, we require that the marked points labelled by $r_{T,i}^\alpha$ map to $D^\alpha_{T,i}$.
\begin{rem}
  To clarify the meaning of Equation \eqref{eq:deformed_Floer}, we note that there is distinguished identification of each component of $\Sigma$ along the path between the marked points $(p,p')$ with the cylinder $\bR \times S^1$, which is canonical up to translation in the $\bR$ direction. This implies that the form $dt$ is well-defined along these components, and we extend it trivially to components not lying along this path.
\end{rem}
\begin{defin}
  The Kuranishi chart $\bX_\alpha$ consists of
  \begin{enumerate}
  \item  the finite group $G_\alpha$,
  \item the partially ordered set $\partial^{\alpha} \cP^{\lambda}(p,p') $ equipped with its natural map to $ \cP^{\lambda}(p,p') $,
  \item  the manifold 
      \begin{equation}
X^{reg}_{\alpha}(p,p') \subset X_{\alpha}(p,p'),
\end{equation}
consisting of curves whose evaluation maps at the points labelled by $r_i^\alpha$ are transverse to $D^\alpha_i$, and which are regular, 
\item the trivial $G_\alpha$ vector bundle over $ X^{reg}_{\alpha}(p,p') $  with fibre $V_\alpha $, and
  \item the section of this vector bundle given by the projection map 
    \begin{equation}
\fs \co      X^{reg}_{\alpha}(p,p') \to V_\alpha .
    \end{equation}
  \end{enumerate}
\end{defin}
We recall that the regular locus consists of all points so that $V_\alpha$ surjects onto the cokernel of the linearisation of the Floer equation in Equation \eqref{eq:deformed_Floer}. The statement that this is a topological manifold with boundary is a consequence of the gluing theory of pseudoholomorphic curves, which in this context was extensively worked out by Pardon in \cite[Appendix C]{Pardon2016}. Moreover, the usual analysis of orientations in Floer theory implies that there are graded lines $\ro_p$ associated to each orbit so that this presentation is oriented relative $\ro_{p,p'}$ as in Equation \eqref{eq:orientation-line-flow-category}.

Each arrow $f\co \alpha \to \beta$ in $A^{\lambda}(p,p')$ induces a map
\begin{equation}
  X_{\alpha}(p,p') \to X_{\beta}(p,p'),
\end{equation}
given by the inclusion $V_\alpha \to V_\beta$ of spaces of inhomogenous perturbations, and the forgetful map on universal curves. This map is equivariant with respect to the $G_\alpha$ action given on the right hand side by the projection map of products of symmetric groups $G_\alpha \to G_\beta $, and respects stratifications with respect to the inclusion
\begin{equation}
  \partial^{\alpha} \cP^{\lambda}(p,p') \subset \partial^{\beta} \cP^{\lambda}(p,p').
\end{equation}
Condition \eqref{eq:all-bubble-stabilised} implies that the kernel acts freely on $ X_{\alpha}(p,p')$, and the commutativity of Diagram \eqref{eq:commutative_diagram_map_deformations} implies that this map preserves regular loci, and moreover that that the diagram 
\begin{equation}
  \begin{tikzcd}
  V_\alpha \ar[r] & V_\beta \\
  X_{\alpha}(p,p') \ar[u] \ar[r] & \partial^f X_{\beta}(p,p') \ar[u]   
  \end{tikzcd}
 \end{equation}
is transverse.

The above discussion implies that each arrow in $A^{\lambda}(p,p')$ induces a map of Kuranishi charts
\begin{equation} \label{eq:map_Kuranishi-charts-morphisms}
  \bX_\alpha (p,p')\to \bX_\beta(p,p'),
\end{equation}
which is functorial in $A^{\lambda}(p,p') $.

For the next statement, we recall the footprint functor from Equation \eqref{eq:footprint_functor}, which assigns to each Kuranishi chart the quotient of the zero-locus of $\fs$ by the automorphism group. Since the zero locus of $\fs$ consists of solutions to Floer's equation, this quotient is naturally identified with an open subject of the corresponding stratum of the compactified moduli space $\cM(p,p')$ of Floer trajectories. 
\begin{prop}
The natural map
  \begin{equation}
    \colim_{\alpha \in  A^{\lambda}(p,p')}  \cM(\bX_\alpha(p,p')) \to \cM(p,p')
  \end{equation}
  is a homeomorphism, and the nerve of the category of charts covering each point in $ \cM(p,p') $ is contractible.
\end{prop}
\begin{proof}
  The first claim is proved by first noting that we can extend each chart defined on a stratum of non-zero codimension to include the top stratum by using gluing to extend the choice of inhomogenous terms from the boundary to the interior of the moduli space of curves with marked points. This implies that the colimit of footprints is homeomorphic to the subcolimit on the category of charts with $P_\alpha$ the minimal element, for which the maps to $\cM(p,p') $ are open inclusions. The homeomorphism then follows from the connectivity of the space of such charts covering each point, which can be proved by adapting the argument below to a simpler setting.

  We reduce the proof of the second claim to the subcategory of charts for which $P_\alpha = P'_\alpha$ labels the stratum containing a chosen point $[u]$. Given any finite diagram $F$ of such charts, we choose a divisor $D_+ \subset M$ which is transverse to $u$, stabilises all its component, and is disjoint from all the divisors appearing in the diagram $F$. We let $r_+$ denote the cardinality of $u^{-1}(D_+)$, and choose a representation $V_+$ of the symmetric group on $r_+$ letters, equipped with an equivariant surjection to the cokernel of the linearised Floer operator at $u$. We extend this surjection in an arbitrary way to a choice of inhomogenous term for the moduli space with $r_+$ marked points. Setting the sequences $\vec{r}^T_+$ for subtrees of $T_\alpha$ to be empty except for $T = T_\alpha$, in which case it is the singleton $r_+$, we obtain in this way a chart, and we write $F_+$ for the constant functor on the domain of $F$ with this value.

  We construct three additional diagram the same domain as $F$, equipped with natural transformations
  \begin{equation}
    F \Leftarrow F' \Rightarrow F'' \Leftarrow F''' \Rightarrow F_+
  \end{equation}
  as follows: in $F'$, we add $r_+$ to the set of marked points in $F$ and $D_+$ to the set of divisors, keeping the same inhomogeneous terms, in $F''$, we enlarge the inhomogeneous terms to include $V_+$, while in $F'''$, we only take $V_+$ as our inhomegenous term. The map from $F'''$ to $F_+$ is then obtained by forgetting the marked points and divisors associated to $F$. This zig-zag provides a homotopy between the map that $F$ induces on nerves and the constant map, which shows that the nerver of the category of charts covering $[u]$ has trivial homotopy groups, hence is contractible.
\end{proof}
\begin{cor}
  The collection of functors
  \begin{equation}
  \bX^{\lambda}(p,p') \co A^{\lambda}(p,p') \to \Chart,
\end{equation}
indexed by $\lambda \in \bR_+$ defines an object $\bX(p,p')$ of $\cD \Kur^{\bR_+}$, lying over $\cP^{\bR_+}(p,p')$. \qed
\end{cor}
\subsection{Composition in the flow category associated to a Hamiltonian}
\label{sec:comp-flow-categ}

Given a triple $(p,p',p'')$ of time-$1$ orbits of $H$, we first consider the map
\begin{equation} \label{eq:map-domain-Kuranishi-presentation}
  A^{\lambda_1}(p,p') \times A^{\lambda_2}(p',p'') \to A^{\lambda_1 + \lambda_2}(p,p'')
\end{equation}
which at the level of stratifying sets is induced by the map
\begin{equation}
  \cP^{\lambda_1}(p,p') \times \cP^{\lambda_2}(p',p'') \to \cP^{\lambda_1 + \lambda_2}(p,p'')
\end{equation}
from Section \ref{sec:flow-categories}, and at the level of marked points and stabilising divisors by concatenating sequences, with the one associated to $(p,p')$ appearing first. These choices specify a moduli space $ \cM_{\alpha_1 \times \alpha_2} $, and determine a homeomorphism
\begin{equation} \label{eq:inclusion_abstract_moduli_spaces_charts}
  \cM_{\alpha_1} \times \cM_{\alpha_2} \to \cM_{\alpha_1 \times \alpha_2}.
\end{equation}
At the level of inhomogenous terms, we complete the definition of the map in Equation \eqref{eq:map-domain-Kuranishi-presentation} by taking direct sums of vector spaces, and extending each map $Y_{\alpha_i}$ by zero to the components of the curve coming from the other factor. 

At the level of moduli space of curves, Equation \eqref{eq:inclusion_abstract_moduli_spaces_charts} induces a $G_{\alpha_1} \times G_{\alpha_2}$-equivariant homeomorphism
\begin{equation}
     X_{\alpha_1}(p,p') \times X_{\alpha_2}(p',p'') \to X_{\alpha_1 \times \alpha_2}(p,p''),
\end{equation}
which is compatible with the projection to $V_{\alpha_1} \times V_{\alpha_2}$, and preserves regular loci as well as orientations. This map is natural with respect to morphisms in $ A^{\lambda_1}(p,p') \times A^{\lambda_2}(p',p'')$, so we conclude that we have constructed a natural isomorphism
  \begin{equation}
    \begin{tikzcd}
      A^{\lambda_1}(p,p') \times A^{\lambda_2}(p',p'') \ar[d] \ar[r] & \Chart_{ \cP^{\lambda_1}(p,p') } \times \Chart_{ \cP^{\lambda_2}(p,p') } \ar[d,""{name=fromhere}] \\
      A^{\lambda_1 + \lambda_2}(p,p'') \ar[r,""{name=tohere}] & \Chart_{ \cP^{\lambda_1 + \lambda_2}(p,p'')}.
         \arrow[Rightarrow,from=fromhere, to=tohere, start anchor={east},
    end anchor={south}, shorten >=3pt, shorten <=6pt,bend right=30]
    \end{tikzcd}
  \end{equation}
  and hence a map
  \begin{equation} \label{eq:composition_Kuranishi-flow-category}
    \bX(p,p') \times \bX(p',p'') \to \bX(p,p'')
  \end{equation}
  in $\cD \Kur^{\bR^+}$. These maps lift to the category of oriented Kuranishi diagrams, with isomorphisms of the underlying orientation lines satisfying the properties listed in Definition \ref{def:flow_category}.

Having arranged in Definition \ref{def:thickening-data} for the vector spaces of choices of inhomogeneous terms to be subspaces of $\bR^{\infty} \otimes  \bR[G_{\alpha}]  $, this operation is strictly associative.  We conclude:
  \begin{prop}
    The Kuranishi presentations $ \bX(p,p') $, indexed by $\bR_+$, are the morphisms of a Kuranishi flow category $\bX(H)$ with objects the time-$1$ Hamiltonian orbits of $H$, with compositions given by Equation \eqref{eq:composition_Kuranishi-flow-category}.  \qed
  \end{prop}

  \section{Flow multimorphisms in Hamiltonian Floer theory}
\label{sec:flow-mult-hamilt}
Let $\overrightarrow{(J,H)}$ be a sequence of Hamiltonians and $\kappa$ an element of $\cF_n(\overrightarrow{(J,H)})$. We shall assign to this cube an element of $\Flow_n(\bX(\overrightarrow{(J,H)}))$, but this requires two preliminary steps: first, we perform all constructions at the level of multimodules over flow categories, then we lift to $\flow_n(\bX(\overrightarrow{(J,H)}))$.

\subsection{Kuranishi diagrams associated to cubes}
\label{sec:kuran-diagr-assoc}

Proceeding as in Section \ref{sec:categ-thick-data}, we begin by associating to the cube $\kappa$, to a real number $\mu$, and to a sequence $\vec{p}$ of elements of $\cP_{\overrightarrow{(J,H)}}$ a partially ordered set $\cQ^{\mu}_{\kappa,\id}(\vec{p})$ whose elements $Q$ consist of:
  \begin{enumerate}
  \item A choice $\sigma_Q$ of a stratum of the $n$-cube,
  \item A tree $T_Q$ obtained  from the tree associated to $T_{\sigma}$ by adding bivalent vertices to some of its edges, whose vertices are labelled by real numbers, with the property that the vertices not coming from $T_\sigma$ are labelled by positive real numbers. In addition, we have the datum of a label of all edges by Hamiltonian orbits of $H_{e}$, so that the leaves are labelled by $\vec{p}$. We shall say that these added vertices are of \emph{Floer type}.
  \end{enumerate}
  To clarify the statement about labelling of edges, we recall that the datum of the cube $\kappa$ includes a labelling of the edges of the tree $T_\sigma$ by Hamiltonian functions; the edges of $T_Q$ inherit this labeling, which we refine with the data of Hamiltonian orbits. The ordering is given by inclusion of strata of the $n$-cube, and by collapse of trees. Note that we have a forgetful map
  \begin{equation} \label{eq:projection-category-multi-module-cube}
    \cQ^{\mu}_{\kappa,\id}(\vec{p}) \to \cP_{\bX(\overrightarrow{(J,H)})}^{\mu}(\vec{p}) \times \Face \square^n,
  \end{equation}
whose first factor is given by collapsing all interior vertices, and whose second factor is given by the chosen face of the cube.

To each $Q \in \cQ^{\mu}_{\kappa,\id}(\vec{p})$, we shall assign a category $A_{Q}$ indexing a Kuranishi presentation. To specify the data determining an object $\alpha$ in $A_Q$ we proceed in two steps.  First, we choose  a pair
\begin{equation}
  Q_\alpha < Q'_\alpha
\end{equation}
of elements of $\cQ^{\mu}_{\kappa,\id}(\vec{p})$ lying under $Q$. We write $\partial^{\alpha} \cQ^{\mu}_{\kappa,\id}(\vec{p})  $ for this set, and  $T_\alpha$ for the tree $T_{Q'_\alpha}$. 
In addition, for each subtree $T$ of $T_\alpha$, we pick an integer $S^\alpha_T$ and a sequence $ \vec{r}^{\alpha}_{T}$ of natural numbers consisting of $S^\alpha_T$ elements. 

To describe the remaining data that are required to specify the object $\alpha$, we introduce the moduli space
  \begin{equation}
    \cM_{\alpha} \to \square^{\sigma_{Q_\alpha}}
  \end{equation}
  consisting of a point in $\square^{\sigma_{Q_\alpha}}$, and a stable curve with points marked by the disjoint union of $\vec{p}$ with the elements of a collection of finite sets  $\{r^{\alpha}_{T,i}\}$, equipped with an identification of the tree of components connecting the marked points $\vec{p}$, with the tree corresponding to an element $Q'$ of  $\partial^{\alpha} \cQ^{\mu}_{\kappa,\id}(\vec{p})$, a framing along each node along the paths from some input to the output, and an identification of the components between the inputs and the output with the prestable curve associated by $\kappa$ to this point which is compatible with framings. More precisely, if $v$ is a vertex of $T_{Q'}$ obtained from a vertex $v$ of the tree $T_{\sigma_{Q'}}$, then we include in our data a choice of biholomorphism between $\Sigma_v$ and the Riemann surface obtained by forgetting all marked points, and collapsing all unstable components not lying on a path from input to output. As before, we impose the additional condition that the marked points labelled by $r^{\alpha}_{T,i}$ stabilise the unions of components corresponding to the image of $T$ in $T_{Q'}$. In particular, forgetting all but of the one of the sets $r^{\alpha}_{T,i}$ does not collapse any component lying over the tree $T$. Finally, we require that the framings of the components associated to vertices of $T_{Q'}$ which are not vertices of $T_{\sigma_{Q'}} $ are compatible in the sense that they induce the same identification with $\bP^1$, up to positive dilation.
  \begin{lem}  The moduli space $ \cM_{\alpha} $ is a manifold with boundary, stratified by the elements of $\cQ^{\mu}_{\kappa,\id}(\vec{p}) $ lying between $Q_\alpha$ and $Q'_\alpha $.                                   
  \end{lem}
  \begin{proof}
  Since the topological type of the curve is constant along each stratum, the key point is to prove that a neighbourhood of each boundary point in $\cM_{\alpha}$ is modelled after the product of a neighbourhood in its stratum with a cube corresponding to the normal directions. This is a consequence of gluing: we may choose gluing parameters so that the corresponding translated cylindrical ends are disjoint from any choice of marked points on a Riemann surface representing a point in $\cM_{\alpha}$, and the glued surfaces yield the desired slice which is modelled after a corner. Performing this construction parametrically with respect to the choices of marked points and modulus on the underlying Riemann surfaces yields an open neighbourhood of the chosen point.
  \end{proof}
  \begin{rem}
    We point out at this stage that it is crucial that the unstable components associated to the vertices of Floer type are considered modulo real translation only (i.e. that we keep track of the framing data). Otherwise, the corresponding corner strata will be collapsed along the induced $S^1$ action.
  \end{rem}

  We can now proceed to complete the definition the domain of the Kuranishi presentation associated to $Q \in \cQ^{\mu}_{\kappa,\id}(\vec{p})$. As before, we write $G_\alpha$ for the product of permutations groups associated to the sequence of natural numbers $ \vec{r}^{\alpha}_{T}$ for all subtrees $T$ of $T_{\alpha}$:
  \begin{defin}
    An object $\alpha$ of $A_Q$ consists of  a pair $ Q_\alpha < Q'_\alpha $ of elements of $\cQ^{\mu}_{\kappa,\id}(\vec{p})$ under $Q$, together with the following data for each subtree $T$ of $T_\alpha$:
    \begin{enumerate}
    \item  a natural number $S^\alpha_T$,
    \item a sequence $ \vec{r}^{\alpha}_T$ of natural numbers of length $S^\alpha_T$, and
      \item a sequence $\vec{D}^{\alpha}_T $ of codimension $2$ submanifolds of $X$, consisting of $S^\alpha_T$ elements.
    \end{enumerate}
  We require as well a finite dimensional $G_\alpha$ representation $V_\alpha$ embedded in $\bR^{\infty} \otimes \bR[G_{\alpha}]   $, and a $G_\alpha$ equivariant map
    \begin{equation}
Y_\alpha \co      V_\alpha \to \Omega^{0,1}_{vert}(\cC_{\alpha}, C^{\infty}(M,TM))
    \end{equation}
    into the space of $(0,1)$-forms on the fibres of the universal curve over the moduli space $\cM_\alpha$, with support away from all the nodes.  
  \end{defin}

  As before, we then introduce the moduli space of maps
  \begin{equation}
         X_{\alpha}(\vec{p}) \to V_{\alpha},
  \end{equation}
  consisting of pairs $(u,v)$, with $v$ an element of $V_\alpha$ and $u$ a map from the fibre of the universal curve over $\cM_{\alpha}$ to $M$ satisfying the equation
  \begin{equation} \label{eq:deformed_Floer_many_inputs}
    \left( du - X_{H} \otimes \alpha \right)^{0,1} = Y(v_\alpha),    
  \end{equation}
mapping the points marked by $r_{T,i}^{\alpha}$ to $D_{T,i}^\alpha$, and  with asymptotic conditions along all nodes lying on the path from inputs to output given by the corresponding label of the tree underlying the chosen element of $\cM_{\alpha} $. In Equation \eqref{eq:deformed_Floer_many_inputs}, the function $H$ and the $1$-form $\alpha$ are the Floer data included in the choice of cube $\kappa$.

  The space $ X_{\alpha}(\vec{p})$ inherits a natural action of $G_\alpha$ from the action of this group on $\cM_\alpha$ (by relabelling marked points) and on $V_\alpha$. We write
  \begin{equation}
      X^{reg}_{\alpha}(\vec{p}) \subset  X_{\alpha}(\vec{p}) 
  \end{equation}
  for the open submanifold of curves which are transverse to the divisors at the marked points, and which are regular. Writing $
  \fs_\alpha$ for the projection map to $V_\alpha$, we have constructed a Kuranishi chart
  \begin{equation}
\bX_{\alpha} \equiv   (G_\alpha, \partial^{\alpha} \cQ^{\mu}_{\kappa,\id}(\vec{p}) ,    X^{reg}_{\alpha}(\vec{p}), V_\alpha, \fs_\alpha),
  \end{equation}
  where we abuse notation and write $ V_\alpha$ for the trivial vector bundle on $X^{reg}_{\alpha}(\vec{p})$ with this fibre.

  We define a morphism $\alpha \to \beta$ in $A_Q$ to consist of a factorisation in  $\cQ^{\mu}_{\kappa,\id}(\vec{p})$ as in Equation \eqref{eq:inclusion-strata}, inclusions of sequences of natural numbers and divisors as in Equations \eqref{eq:inclusion_marked_points} and \eqref{eq:inclusion-divisors}, and a map of representations inducing a commutative diagram as in Equation \eqref{eq:commutative_diagram_map_deformations}. This induces a map of the associated Kuranishi charts as in Equation \eqref{eq:map_Kuranishi-charts-morphisms}, which defines the desired functor
  \begin{equation} \label{eq:Kuranishi-presentation-stratum-multimodule}
    \bX^Q \co A_Q \to \Chart.    
  \end{equation}

  The Kuranishi presentations from Equation \eqref{eq:Kuranishi-presentation-stratum-multimodule} assemble into a Kuranishi diagram
  \begin{equation}
       \bX^{\mu}_{\kappa,\id}(\vec{p}) \co \cQ^{\mu}_{\kappa,\id}(\vec{p}) \to \Kur.
  \end{equation}
Indeed, each pair $Q \to Q'$ determines an inclusion of categories
\begin{equation} \label{eq:inclusion_categories_strata}
  A_{Q'} \to A_Q,    
\end{equation}
with the property that the $\bX^{Q'}$ agrees with the restriction of $\bX^Q$. We write
\begin{equation}
    \bX_{\kappa,\id}(\vec{p}) \in \cD \Kur^{\Gamma}.
\end{equation}
for the resulting $\bR$-graded Kuranishi diagram. At this stage, we note that a permutation $f$ of the inputs induces an isomorphism of $\bR$-graded Kuranishi diagrams
\begin{equation}
   \bX_{\kappa,\id}(\vec{p})  \cong \bX_{\kappa,\id}(f^*\vec{p}),
\end{equation}
induced by relabelling the incoming leaves of the trees that enter in the definition of the strata and of the Kuranishi presentations.

\subsection{Multimodule structure maps associated to a cube}
\label{sec:mult-struct-maps}

We now explain the multi-module structure maps on $\bX_{\kappa,\id}$. Recall that we have constructed module maps
\begin{align}
  \cP_{\bX_i}^{\lambda}(p'_i, p_i) \times   \cP_{\vec{\bX}}^{\mu}(\vec{p}) & \to \cP_{\vec{\bX}}^{\lambda+\mu}((p'_i, p_i) \circ_i \vec{p}) \\
  \cP_{\vec{\bX}}^{\mu}(\vec{p}) \times   \cP_{\bX}^{\lambda}(p, p') & \to \cP_{\vec{\bX}}^{\mu+\lambda}( \vec{p} \circ_1 (p,p')) 
\end{align}
in Section \ref{sec:multimodules-trees}. Since these are defined by concatenating trees, they naturally lift to maps
  \begin{align}
    \cP_{\bX_i}^{\lambda}(p'_i, p_i) \times    \cQ^{\mu}_{\kappa,\id}(\vec{p}) & \to \cQ^{\lambda+\mu}_{\kappa,\id}((p'_i, p_i) \circ_i \vec{p}) \\
    \cQ^{\mu}_{\kappa,\id}(\vec{p})  \times   \cP_{\bX}^{\lambda}(p, p') & \to  \cQ^{\mu+\lambda}_{\kappa,\id}( \vec{p} \circ_1 (p,p')),
  \end{align}
which preserve the projection to the cube in Diagram \eqref{eq:projection-category-multi-module-cube}.

Since we constructed $\cF$ as a category enriched in Kuranishi presentations (rather than in Kuranishi diagrams), we now apply the functor from $\Kur$ to $\cD \Kur$ that assigns to each presentation the diagram indexed by the strata: whenever $P$ is an element of $\cP^{\lambda}(p, p')$, we denote by
\begin{equation}
  A_{P} \subset A^{\lambda}(p,p')
\end{equation}
the subcategory consisting of thickening data as in Definition \ref{def:thickening-data} whose underlying strata lie under $P$. Writing $Q \circ_1 P$ for the image of the pair $(Q,P)$ in  $  \cQ_{\kappa,\id}^{\mu+\lambda}( \vec{p} \circ_1 (p,p'))$, there is a natural functor
\begin{align}
  A_Q \times A_P & \to A_{Q \circ_1 P} \\
  (\beta, \alpha) & \mapsto \beta \circ_1 \alpha
\end{align}
defined as follows: we set
\begin{equation}
  Q_{\beta \circ_1 \alpha}  < Q'_{\beta \circ_1 \alpha}
\end{equation}
so be the images of $(Q_{\beta}, Q_{\alpha})$ and $(Q'_\beta, Q'_{\alpha})$ under the product map. In this way, we find that $T_{\beta \circ_1 \alpha } $ is the concatenation of $T_\beta$ with $T_\alpha$.  We then define $S_{\beta \circ_1 \alpha}^T$ to agree with $S_{\beta}^T $ if $T \subset T_\beta$, with $S^T_\alpha$ if $T \subset T_\alpha$, and to otherwise vanish. We then take the sequences of marked points and of divisors for the subtrees of the composition to be given by those specified in the data for $\alpha$ and $\beta$.  This construction induces a homeomorphism of moduli spaces
\begin{equation} \label{eq:embedding_abstract_moduli-spaces-multimodule}
  \cM_{\beta} \times \cM_\alpha \to \cM_{\beta \circ_1 \alpha}.
\end{equation}
Taking the direct sum of vector spaces
\begin{equation}
  V_{\beta \circ_1 \alpha }  \equiv V_\beta \oplus V_{\alpha},
\end{equation}
we have a natural map
\begin{equation}
  Y_{\beta \circ_1 \alpha  }  \co V_{\beta \circ_1 \alpha } \to \Omega^{0,1}_{vert}(\cC_{ \beta \circ_1 \alpha}, C^{\infty}(M,TM)),
\end{equation}
which is the direct sum of $ Y_{\beta}$ and $Y_{\alpha}$ along the image of Equation \eqref{eq:embedding_abstract_moduli-spaces-multimodule}, extended by $0$ to the other components.

As in Section \ref{sec:comp-flow-categ}, Equation \eqref{eq:embedding_abstract_moduli-spaces-multimodule} induces a homeomorphism of moduli spaces
\begin{equation}
   X_{\beta}(\vec{p}) \times X_{\alpha}(p,p') \to X_{ \beta \circ_1 \alpha}( \vec{p} \circ_1 (p,p')),
\end{equation}
which is $\Gamma_{ \beta \circ_1 \alpha }$-equivariant, preserves regular loci, and intertwines the projections to $  V_{\beta \circ_1 \alpha } $. We thus obtain a map of Kuranishi charts
\begin{equation}
 \bX_{\beta} \times \bX_{\alpha} \to \bX_{ \beta \circ_1 \alpha},
\end{equation}
which is natural with respect to the choice of objects $ (\beta, \alpha) \in  A_Q \times A_P $, hence induces a map of Kuranishi presentations
\begin{equation}
  \bX^Q  \times     \bX^P \to     \bX^{Q \circ_1 P}.
\end{equation}

This entire construction is functorial with respect to the choice of pair
\begin{equation}
(Q,P) \in   \cQ^{\mu}_{\kappa,\id}(\vec{p}) \times \cP^{\lambda}(p, p'),  
\end{equation}
hence defines a map of Kuranishi diagrams
\begin{equation}
   \bX^{\mu}_{\kappa,\id}(\vec{p}) \times  \bX^{\lambda}(p, p' ) \to   \bX^{\mu+\lambda}_{\kappa,\id}(\vec{p} \circ_1 (p,p')).
\end{equation}
Assembling these maps for all choices of a real number $\mu$ and a non-negative real number $\lambda$, we obtain a map
\begin{equation}
  \bX_{\kappa,\id}(\vec{p}) \times  \bX(p, p' ) \to   \bX_{\kappa,\id}(\vec{p} \circ_1 (p,p'))
\end{equation}
of $\bR$-graded Kuranishi diagrams.

Applying the same construction to the operations $\circ_i$ for the inputs yields maps
\begin{align}
  \bX(p'_i, p_i) \times \bX_{\kappa,\id}(\vec{p})  \to   \bX_{\kappa,\id}((p'_i, p_i) \circ_i \vec{p}).
\end{align}

The associativity properties satisfied by these maps, and their compatibility with permutations of the inputs, are summarised by the following statement:
\begin{lem}
  The collection of $\bR$-graded Kuranishi diagrams $  \bX_{\kappa,\id}(\vec{p}) $, indexed by sequences $\vec{p}$ of time-$1$ orbits of the Hamiltonians $\overrightarrow{(J,H)}$, define a multimodule over the flow categories $\bX(\overrightarrow{(J,H)})$. \qed
\end{lem}

We now assign to each cube $\kappa$ in $\cF_n(\overrightarrow{(J,H)})$ an element $\bX_{\kappa} $ of $\flow_n(\bX(\overrightarrow{(J,H)}))$: recall from Definition \ref{def:Kuranishi-n-multimodule} that this consists of the assignment of a multimodule over the sequence $\bX(\overrightarrow{(J,H)})$, enriched in the category of $\bR$-graded Kuranishi diagrams, for each permutation of the coordinates of $\square^n$, together with coherent isomorphisms for compositions of permutations.

The module assigned to a permutation $\phi$ is
\begin{equation}
  \bX_{\kappa,\phi} \equiv   \bX_{\phi^* \kappa, \id}.
\end{equation}
For each composition $\phi_1 \circ \phi_2$ of permutations, we consider the isomorphism
\begin{equation} \label{eq:isomorphism_Kuranishi-multimodule_pullback}
  \bX_{(\phi_1 \circ \phi_2)^* \kappa, \id} \to  \phi_2^* \bX_{\phi_1^* \kappa, \id }
\end{equation}
defined as follows: the action of $\phi_2$ on the faces of the cube determines an isomorphism of partially ordered sets
\begin{equation}
\phi_2 \co  \cQ^{\mu}_{\phi_1^* \kappa, \id}(\vec{p}) \cong  \cQ^{\mu}_{(\phi_1 \circ \phi_2)^* \kappa, \id}(\vec{p}) 
\end{equation}
lying over the action of $\phi_2$ of faces of the cube. The categories $A_Q$ and $A_{\phi_2 Q}$ associated to an element of the left hand side and its image are naturally isomorphic, since they are given by the same data, and the functors to $\Kur$ are intertwined by these isomorphisms. The isomorphism of Kuranishi multimodules in Equation \eqref{eq:isomorphism_Kuranishi-multimodule_pullback} then follows from the compatibility of the isomorphism with maps of the stratifying sets of the resulting Kuranishi diagrams, relabeling of inputs, and multimodule actions.
\begin{lem}
  The collection multimodules $ \{ \bX_{\phi^* \kappa,\id} \}$, together with the isomorphism of Equation \eqref{eq:isomorphism_Kuranishi-multimodule_pullback}, define an object of $\flow_n(\bX(\overrightarrow{(J,H)})) $ associated to each cube $\kappa$ in $\cF_n(\overrightarrow{(J,H)}) $. This assignment is equivariant with respect to the action of the symmetric group on these two sets. \qed
\end{lem}
\subsection{Composition of multimodules of Kuranishi diagrams}
\label{sec:comp-mult-kuran}

We now consider a pair $\overrightarrow{(J,H)}_+$ and $\overrightarrow{(J,H)}_-$ of sequences of Hamiltonians, with $H_+ = H_{-,i}$, and cubes
\begin{equation}
  \kappa_{\pm} \in \cF_{n_\pm}(\overrightarrow{(J,H)}_\pm).
\end{equation}
For each pair $\vec{p}_\pm$ of sequences of time-$1$ orbits of $ \overrightarrow{(J,H)}_+$ with the property that $p_+ = p_{-,i}$, and for each pair $\mu_\pm$ of real numbers,  we have a map of partially ordered sets
\begin{equation} \label{eq:map_circ_i-posets}
\circ_i \co   \cQ^{\mu_+}_{\kappa_+,\id}(\vec{p}_+) \times  \cQ^{\mu_-}_{\kappa_-,\id}(\vec{p}_-) \to  \cQ^{\mu_+ + \mu_-}_{\kappa_+ \circ_i \kappa_-,\id}( \vec{p}_+ \circ_i \vec{p}_-)
\end{equation}
defined by taking the products of strata of cubes, and by concatenating trees. Given a pair $(Q_+, Q_-)$ lying in the source of this map, we have a functor
\begin{align}
  A^{Q_+} \times A^{Q_-} & \to A^{Q_+ \circ_i Q_-} \\
  (\alpha_+, \alpha_-) & \mapsto \alpha_+ \circ_i \alpha_-
\end{align}
obtained as follows: we assign to pairs $Q_{\alpha_+} \to Q'_{\alpha_+}$ and  $Q_{\alpha_-} \to Q'_{\alpha_-}$ the image pair
\begin{equation}
    Q_{\alpha_+ \circ_i \alpha_-} \to Q'_{\alpha_+ \circ_i \alpha_-}
  \end{equation}
  under Equation \eqref{eq:map_circ_i-posets}. The tree $T_{ \alpha_+ \circ_i \alpha_-}$ is then naturally isomorphic to the concatenation of $T_{\alpha_+}$ with $T_{\alpha_-}$. We set the collections of marked points and divisors associated to any subtrees of $ T_{ \alpha_+ \circ_i \alpha_-}$ to agree with the marked points and divisors associated to the corresponding subtree of $T_{\alpha_+}$ or $T_{\alpha_-}$ if it is contained in it,  and to be empty otherwise. This induces a homeomorphism of moduli spaces
\begin{equation}
  \cM_{\alpha_+} \times \cM_{\alpha_-} \to \cM_{\alpha_+ \circ_i \alpha_- }
\end{equation}
lying over the products of the cubes associated to $Q_+$ and $Q_-$. Defining $V_{\alpha_+ \circ_i \alpha_- } $ to be the direct sum of $V_{\alpha_+}$ and $V_{\alpha_-}$, we have a natural map
\begin{equation}
  Y_{\alpha_+ \circ_i \alpha_-  }  \co V_{\alpha_+ \circ_i \alpha_- } \to \Omega^{0,1}_{vert}(\cC_{\alpha_+ \circ_i \alpha_- }, C^{\infty}(M,TM)),
\end{equation}
given by the direct sum of $ Y_{\alpha_\pm}$ along the image of the embedding of moduli spaces, extended by $0$ to the other components of $  \cM_{\alpha_+ \circ_i \alpha_- }$.

Passing to regular loci of the associated moduli spaces of maps, we obtain from this construction a map of  Kuranishi charts
\begin{equation}
 \bX_{\alpha_+} \times \bX_{\alpha_-} \to \bX_{\alpha_+ \circ_i \alpha_- } ,
\end{equation}
which is natural with respect to the choice of objects $ \alpha_\pm \in  A_{Q_\pm}$, hence induces a map of Kuranishi presentations
\begin{equation}
  \bX^{Q_+}  \times     \bX^{Q_-} \to     \bX^{Q_+ \circ_i Q_-}.
\end{equation}
Since the maps in Equation \eqref{eq:inclusion_categories_strata} are inclusions of categories which are compatible with the Kuranishi presentations we have defined, this construction is again functorial with respect to the choice of strata $Q_\pm$, hence defines a map
\begin{equation} \label{eq:product_multimodules-map}
  \bX_{\kappa_+,\id}(\vec{p}_+) \times   \bX_{\kappa_-,\id}(\vec{p}_-) \to   \bX_{\kappa_+ \circ_i \kappa_-,\id}( \vec{p}_+ \circ_i \vec{p}_-)
\end{equation}
of $\bR$-graded Kuranishi diagrams.

At this stage, we consider the square of $\bR$-graded Kuranishi diagrams
\begin{equation} \label{eq:diagram_composition_multimodules-middle}
  \begin{tikzcd}
    \bX_{\kappa_+,\id}(\vec{p}_+) \times \bX(p,p_{-,i}) \times  \bX_{\kappa_-,\id}(\vec{p}_-) \ar[d] \ar[r] & \bX_{\kappa_+,\id}(\vec{p}_+ \circ_1 (p,p_{-,i})) \times  \bX_{\kappa_-,\id}(\vec{p}_-) \ar[d] \\
  \bX_{\kappa_+,\id}(\vec{p}_+) \times \bX(p,p_{-,i}) \times  \bX_{\kappa_-,\id}(\vec{p}_-) \ar[r] &  \bX_{\kappa_+ \circ_i \kappa_-,\id}( \vec{p}_+ \circ_1 (p,p_{-,i}) \circ_i \vec{p}_-).
  \end{tikzcd}
\end{equation}
We claim that this diagram (strictly) commutes: first, it strictly commutes at the level of the partially order set of strata because the corresponding maps are defined by concatenation of trees. Next, we see that the indexing diagrams for the categories of Kuranishi presentation strictly commute because we have defined our choices of sequences of marked points and divisors to be associated to the subtrees of the trees underlying the strata. Finally, the diagram at the level of Kuranishi presentations commutes because of the uniqueness of the associator isomorphisms of sets.

An easier analysis shows that all the maps in Diagram \eqref{eq:diagram_composition_multimodules-middle} commute with the multimodule structure maps associated to the sequence $ \overrightarrow{(J,H)}_+  \circ_i \overrightarrow{(J,H)}_- $. Its compatibility with the symmetric group actions implies:
\begin{lem}
  The collection of maps given by applying Equation \eqref{eq:diagram_composition_multimodules-middle} to composable sequence of time-$1$ orbits of $\overrightarrow{(J,H)}_+$ and $\overrightarrow{(J,H)}_-$ defines a morphism
  \begin{equation}
     \bX_{\kappa_+} \circ_i   \bX_{\kappa_-} \to   \bX_{\kappa_+ \circ_i \kappa_-}
   \end{equation}
in $\flow_{n_+ + n_-}(\bX(\overrightarrow{(J,H)}_+ \circ_i \overrightarrow{(J,H)}_-)) $. \qed
\end{lem}
Having arranged for the marked points and divisors to be indexed by the vertices of trees, the associativity of concatenation of trees implies:
\begin{lem}
  Consider a triple $\overrightarrow{(J,H)}_+$, $\overrightarrow{(J,H)}$, and  $\overrightarrow{(J,H)}_-$ of Hamiltonians, and either (i) a pair $i < j$ of labels for the inputs of  $\overrightarrow{(J,H)}_-$ with $H_{-,i}=H$ and $H_{-,j}=H_+$, or (ii) a pair $i$ and $j$ of labels for the inputs of  $\overrightarrow{(J,H)}_-$ and $\overrightarrow{(J,H)}$ with $H_{-,i}=H$ and $H_{j}=H_+$. In the two respective situations, the diagram associated to taking composition in the two possible orders commutes up to the natural isomorphism from Lemma \ref{lem:multi-category-flow}
 \qed
\end{lem}
\subsection{Construction of the multifunctor}
\label{sec:constr-mult}

It remains at this stage to assemble the constructions of the previous section into a multifunctor from the multicategory of Hamiltonians to the multicategory $\Flow$ of flow categories. The key missing ingredient is that an element of $\Flow_n(\bX(\overrightarrow{(J,H)}))$ contains the datum of compatible factorisations of the flow multimodules associated to each stratum of the $n$-cube as a multicomposition. The next result asserts the existence and uniqueness of such lifts:
\begin{lem} \label{lem:factorisation-cubes}
  For each cube $\kappa \in \cF_n(\overrightarrow{(J,H)})$, there is a tree $T_{\kappa}$, with edges labelled by Hamiltonians so that the incoming ones are labelled by $\overrightarrow{(J,H)}$, a collection $n_v$ of natural numbers indexed by the vertices of $T_\kappa$ whose sum is strictly smaller than $n$, a collection of non-degenerate cubes $\kappa_{v} \in \cF_{n_v}(\overrightarrow{(J,H)})$, a permutation $\phi$ of $\{1, \ldots, n\}$, and an isomorphism
  \begin{equation}
\Psi \co [0,1]^{n - \sum n_v} \times \underset{T_\kappa}{\bigcirc} \kappa_v \cong \phi^* \kappa.
\end{equation}
This choice is unique if we assume that the restriction of $\phi$ to the collection of coordinates associated to each cube $\kappa_v$, and to the degenerate coordinates, preserves order, and satisfies the following universal property: for each choice of data $(T'_\kappa, \{n_{v'}\}, \kappa_{v'}, \phi', \Psi')$ satisfying these properties, then there is a map of trees $\pi \co T_\kappa \to T'_\kappa$, permutations $\phi_{v'}$ of $\{1, \ldots, n_{v'}\} $, and isomorphisms
   \begin{equation} \label{eq:factorisation-less-refined}
   \Psi_{v'} \co  [0,1]^{n_{v'} - \sum n_v} \underset{\pi^{-1} T_{v'}}{\bigcirc} \kappa_v \cong \phi^*_{v'} \kappa,
   \end{equation}
   so that, assuming that the permutations $\phi_{v'} $ are chosen to be order preserving on the subintervals associated to each factor, the map $\Psi$ is obtained by taking the composition of $\Psi'$ with the product of the maps $\Psi'_v$.
\end{lem}
\begin{proof}
  We first consider the case of a non-degenerate element. Recall that there is a cube $T_{[\square^n]}$ associated by $\kappa$ to the top stratum of the cube. The key point is to construct a quotient $T_{\kappa}$ of this tree, and a corresponding partition of the coordinates by vertices of $T_{\kappa}$ (some elements of this partition may be empty). This quotient is constructed as follows: consider the set of pairs $(T,I)$, consisting of a subtree $T$ of $T_{[\square^n]}$, and a (possibly empty) subset $I$ of the coordinates $\{1, \ldots, n\}$ with the property that the Floer data associated to a vertex $v$ of $T_{[\square^n]}$ depend on a coordinate lying in $I$ if and only if $v$ is a vertex of $T$ (note that this condition is vacuous if $I$ is empty). This collection of pairs is partially ordered by inclusion of trees and of subsets of the coordinates, and this partial order has a unique set of minimal elements because it is closed under intersections. Since the pair $(T_{[\square^n]}, \{1, \ldots, n\})$ is an element of this set, the collection of minimal elements provides a partition of the tree, and we define $T_\kappa$ to be the quotient $T_{[\square^n]}$ by the resulting equivalence relation. Choosing an ordering of the vertices of $T_\kappa$, and hence a decomposition of $\{1, \ldots, n\}$ as the ordered union of the intervals $\{1, \ldots, n_v\}$, the permutation $\phi$ can be chosen to be the unique map whose restriction to $\{1, \ldots, n_v\}$ preserves order with image $I_v$. The cubes $\kappa_{v'}$ are then defined to be the non-degenerate cubes associated to the restriction of $\kappa$ to the inverse image of $v$ in $T_{[\square^n]}$, and the factorisation follows by construction. We extend this construction to arbitrary (degenerate) cubes by taking the product with the intervals labelled by the coordinates corresponding to the degeneracies.

  The universal property is proved by noting that every other factorisation provides an equivalence relation on the set of vertices of $T_{[\square^n] }$ which is refined by the equivalence relation associated to the minimal trees.
\end{proof}

We are now ready to construct an object $\tilde{\bX}_{\kappa}$ of $\Flow_n(\bX(\overrightarrow{(J,H)}))$ for each cube $\kappa$ in $ \cF_n(\overrightarrow{(J,H)}) $ (c.f. Definition \ref{def:Flow_multicategory}): if $\kappa$ is non-degenerate, and $\sigma \in \Face(\square^n)$ is a stratum of codimension $k$, we write $\kappa|\sigma$ for the cube of dimension $n-k$ obtained from the associated face map, and we define the tree
\begin{equation}
  T_\sigma \equiv T_{\kappa|\sigma},
\end{equation}
and the collection of flow modules $\{ \bX_{\kappa,v} \}$ indexed by the vertices of this tree to be those obtained by applying the above result to  $\kappa|\sigma$. To construct the maps associated to pairs $\sigma < \tau $ of faces, note that the factorisation of $\kappa|\sigma $ into a composition induces a factorisation of $\kappa|\tau$, so that Equation \eqref{eq:factorisation-less-refined} provides the desired maps in Equation \eqref{eq:map_decomposed_strata_multimodule}. The compatibility diagram for triples of faces is then a consequence of the compatibility of the maps to $\kappa$.

At this stage, one may readily check that this construction is compatible with (i) the structure maps of cubical sets, i.e. permutations, injections, and surjections of cubes, (ii) the permutation maps of inputs, and (iii) multicomposition operations, which implies:
\begin{prop}
  The assignment $\kappa \to \tilde{\bX}_{\kappa}$ defines a multifunctor
  \begin{equation}
    \cF \to \Flow.
  \end{equation} \qed
\end{prop}

\part{Appendices}
\label{part:appendices}

\appendix

\section{Symmetric monoidal categories, functors, and natural transformations}
\label{sec:symm-mono-categ-1}

As a preliminary notion, recall that a weak equivalence of functors
  \begin{equation} 
\begin{tikzcd}[column sep=huge]
\cC
  \arrow[bend left=50]{r}[name=U,label=above:$F$]{}
  \arrow[bend right=50]{r}[name=D,label=below:$G$]{} &
\Ch.
\arrow[from=U.south-|D,to=D,Rightarrow,shorten=5pt]
\end{tikzcd}   
\end{equation}
is a natural transformation which is given, for every object $\bX$ of $\cC$, by a quasi-isomorphism of chain complexes
\begin{equation}
  F(\bX) \to G(\bX).  
\end{equation}
Recall as well that, if $F$ and $G$ are (lax) monoidal functors (i.e. if we have an associative natural transformation $F(\bX) \otimes F(\bY)   \to F(\bX \times \bY)$, and similarly for $G$), then a monoidal natural transformation
\begin{equation}
  F \Rightarrow G
\end{equation}
is a natural transformation with the property that the diagram
\begin{equation}
  \begin{tikzcd}
    F(\bX_1) \otimes F(\bX_2) \ar[r] \ar[d] & F( \bX_1 \times \bX_2) \ar[d] \\
     G(\bX_1) \otimes G(\bX_2) \ar[r]  & G( \bX_1 \times \bX_2)
  \end{tikzcd}
  \end{equation}
  commutes. If $F$ and $G$ are symmetric monoidal (i.e. if the morphism $F(\bX) \otimes F(\bY)   \to F(\bX \times \bY)$ commutes with transposition), a symmetric monoidal natural transformation is a monoidal transformation in the above sense (no additional property is required). A key example is given by:
  
  \begin{prop} \label{prop:Eilenberg-Zilber}
  The Eilenberg-Zilber shuffle map defines a quasi-isomoprhism
  \begin{equation}
    C_* X \otimes C_* Y \to C_* \left(X \times Y\right)
  \end{equation}
  whenever $X$ and $Y$ are topological spaces, and $C_*$ denotes the singular chain functor. This quasi-isomorphism is symmetric in the sense that it lifts the singular chain functor from the category of topological spaces (or simplicial sets) to the category of cochain complexes to a lax symmetric monoidal functor. \qed
\end{prop}

The notions discussed above  lead to the \emph{homotopy category} of symmetric monoidal functor, whose objects are symmetric monoidal functors, and whose morphisms are representable by zig-zags
\begin{equation}
  F \Leftarrow \cdots \Rightarrow G  
\end{equation}
of monoidal natural transformations between symmetric monoidal functors, in which all the left pointing arrows are weak equivalences. We shall not need to quote from the literature any result about this category, as we prove by hand, in Proposition \ref{prop:cofibrant_property_of_lift_Kur}, the result that we need.

\section{Symmetric cubical sets}
\label{sec:symm-cubic-sets}

In Part \ref{part:from-virtual-counts}, we use cubical sets as a model for (the homotopy theory of) spaces. The main advantage of cubes over simplices is that the product of cubes is again a cube. Unfortunately, in the usual notion of cubical sets, the monoidal structure that corresponds to the product of cubes (called the Day convolution) is not symmetric, because permutations of cubes are not considered as part of the data. One is then led to consider the notion of a symmetric cubical set:
\begin{defin}[Section 2.2 \cite{Grandis2007}]
  A \emph{symmetric cubical set} $A_\bullet$ consists of the following data:
  \begin{itemize}
  \item ($n$-cubes) A collections of sets $A_n$ indexed by the natural numbers (starting with $n=0$).
  \item (Face maps) For each $k$-dimensional face of $[0,1]^n$, a face map $A_{n} \to A_{k} $.
  \item (Degeneracies) For each projection map $[0,1]^n \to [0,1]^k$ a degeneracy map $A_{k} \to A_n$.
    \item (Symmetries) For each permutation of the set $\{1, \cdots, n\}$, an isomorphism $A_n \cong A_n$.
    \end{itemize}
    These structures are required to satisfy all the relations dual to the maps of cubes which define them, in the following sense: if $\mathrm{Cubes}$ is the subcategory of the category of spaces whose objects are cubes $[0,1]^n$ (for $n$ a natural number), and whose maps are compositions of (i) inclusions of boundary faces, (ii) projection maps, and (iii) reordering of factors, then the above data defines a functor
    \begin{equation}
      A_{\bullet} \co \mathrm{Cubes} \to \Set.      
    \end{equation}
\end{defin}

\begin{rem}
  In some part of the literature, e.g. \cite{Isaacson2011}, one assumes as well the existence of structures called \emph{connections}, which arise from the map
  \begin{equation}
    \max \co [0,1]^2 \to [0,1].
  \end{equation}
  Cubical sets with connections have the advantage that they straightforwardly model the homotopy theory of spaces in the sense that all objects are fibrant (i.e. satisfy the analogue of the Kan condition for simplicial sets). However, they are inconvenient to use in our context because the function $\max$ is not smooth.
\end{rem}

Symmetric cubical sets form a category with morphisms $A_\bullet \to B_\bullet$ given by a collection of maps
\begin{equation}
  A_n \to B_n  
\end{equation}
which commute with all structure maps. The category of cubical sets admits products, which are given by taking the product of the underlying sets and the structure maps, but this does not have the desired properties for us. Instead, we are interested in the symmetric monoidal structure which is associated to the fact that the category $\cC$ of cubes is symmetric monoidal, via the general process of Day convolution. This may be explicitly described by defining  the tensor product $A \otimes B$ of symmetric cubical sets to be given by
  \begin{equation} \label{eq:monoidal_structure_cubical_set}
    \left(A \otimes B\right)_n \equiv \coprod_{i+j=n} \Iso(n, (i, j)) \times_{\Sigma_i \times \Sigma_j} A_i \times A_j.    
  \end{equation}
  where $\Iso(n,(i,j))$ is the set of bijections from $(\{1, \ldots, n\}$ to the disjoint union of $\{1, \ldots, i\}$ with $\{1, \ldots, j\}$, and $\Sigma_i \times \Sigma_j$ is the product of the symmetric group on $i$ and $j$ letters, which acts on $A_i \times A_j$ via the structure maps and on $\Iso(n, (i,j)) $ via reparametrising the target. All the structure maps of $ A \otimes B$ are induced in a straightforward way from the corresponding structure maps of $A$ and $B$, keeping in mind that one uses the choice of element of $\Iso(n,(i,j))$ to transfer any map of cubes to a pair of maps of cubes. There is a natural isomorphism
  \begin{equation}
     A \otimes B \cong B \otimes A   
  \end{equation}
  of cubical sets, arising from the isomorphism
  \begin{equation}
        \Iso(n, (i,j)) \cong \Iso(n,(j,i))
  \end{equation}
  induced by the isomorphism
  \begin{equation}
   \{1, \ldots, i\} \amalg \{1, \ldots, j\}   \cong  \{1, \ldots, j\}  \amalg \{1, \ldots, i\}.
  \end{equation}
\begin{prop}[Section 6.3 of \cite{Grandis2009}]
 The category of symmetric cubical sets is symmetric monoidal with respect to the monoidal structure in Equation \eqref{eq:monoidal_structure_cubical_set}. \qed
\end{prop}

Next, we assign to each symmetric cubical set $A$ the chain complex $ C_*(A; \Bbbk)$ which in degree $n$ is the quotient of the free $\Bbbk$-module generated by the $n$-cubes modulo (i) the submodule generated by degenerate cubes, and (ii) the submodule generated by the sum $\alpha + \tau \left(\alpha\right) $, for $\alpha$ an arbitrary $n$-cube and  $\tau$ a transposition of $\{1, \ldots, n\}$. The differential on this complex is inherited from the alternating sum of the face maps. 
\begin{lem}
  The assignment
  \begin{equation}
    A \to    C_*(A; \Bbbk)  
  \end{equation}
  defines a symmetric monoidal functor from the category symmetric cubical sets to the category of chain complexes over $\Bbbk$.
\end{lem}
\begin{proof}[Sketch of proof:]
  The functoriality is straightforward. The monoidal structure is induced by mapping the generators associated to $\alpha$ and $\beta$ to their product. Having taken the quotient by the relation which identifies cubes which differ by a permutation  (up to sign), the monoidal structure is symmetric. 
\end{proof}

The following consequence is implicit in \cite[Appendix C]{Abouzaid2011}:
\begin{cor}
  The symmetric cubical chains define a symmetric monoidal functor from the category of topological spaces to the category of chain complexes over $\Bbbk$, which computes ordinary homology. \qed
\end{cor}

\section{Multiplicativity of colimits via the prismatic subdivision}
\label{sec:mult-colim-via}

Given a category $\cC$, let $  \Delta  \cC$ denote the category of simplices,  i.e. objects are functors from a poset $\mathbf{n} = \{0 < \ldots < n\}$ to $\cC$, and morphisms are given by an order preserving map $\mathbf{n} \to \mathbf{m} $ so that the following  diagram commutes 
\begin{equation}
  \begin{tikzcd}
    \mathbf{n} \ar[r] \ar[d] & \cC . \\
    \mathbf{m} \ar[ur] & 
  \end{tikzcd}
\end{equation}
We note that there is a natural projection map
\begin{equation}
\pi \co \Delta \cC \to \cC  
\end{equation}
which assigns to each simplex its last element.

We shall be interested in functors
\begin{equation}
F \co  \Delta \cC \to \Ch.
\end{equation}
We denote the (ordinary) left Kan extension of this functor to $\cC$ by $LF$. Given a functor $A \to \cC$, we shall associate to $F$ the homotopy colimit of $L F$ over $A$, which we shall model via the bar construction:
\begin{equation}
  \hocolim_{A} L F \equiv B( \Bbbk, C_*(\Delta A), L F),  
\end{equation}
where we write $C_*(\Delta A)$ for the chains of the nerve, and $\Bbbk$ for the module over $C_*(\Delta A)$ given by the forgetful map from $\Delta A$ to a point.

Assuming that $\cC$ is a symmetric monoidal category, our main goal in this Appendix is to formulate the structures (at the level of functors) which ensure that such a homotopy colimit yield a symmetric monoidal functor on the category of diagrams in $\cC$.

The starting point is to consider the category $ \Delta^{\RS{lr}}\cC^2 $ whose objects are compositions $  \mathbf{k} \to \mathbf{n}_1 \times \mathbf{n}_2 \to \cC \times \cC$, and whose morphisms are commutative diagrams 
\begin{equation}
  \begin{tikzcd}
    \mathbf{k} \ar[r] \ar[d] & \mathbf{n}_1 \times \mathbf{n}_2 \ar[r] \ar[d] & \cC \times \cC \ar[d,"="] \\
    \mathbf{\ell} \ar[r]  & \mathbf{m}_1 \times \mathbf{m}_2 \ar[r]  & \cC \times \cC.
  \end{tikzcd}
\end{equation}
This category is equipped with natural functors to $\Delta \cC \times \Delta \cC$ as well as to $\Delta \left( \cC \times \cC\right)$. Given a functor $F \co \Delta \cC \to \Ch$, we thus obtain a diagram
  \begin{equation}
     \begin{tikzcd}
       \Delta^{\RS{lr} }\cC^2 \ar[r] \ar[d] &\Delta \left( \cC \times \cC\right) \ar[r] &  \Delta \cC  \ar[r]   & \Ch. \\
      \Delta  \cC \times \Delta \cC \ar[urrr,dashed] & & &
    \end{tikzcd}
  \end{equation}
  The dashed arrow refers to the left Kan extension of the horizontal composition, over the projection map. We shall denote this Kan extension by $F^{\RS{lr}} $. 
  \begin{defin} \label{def:mult_structure_functor}
    A \emph{multiplicative structure} on $F \co \Delta \cC \to \Ch$ is a natural transformation
    \begin{equation}
      F \otimes  F \Rightarrow   F^{\RS{lr}}
    \end{equation}
     of functors from $ \Delta \cC \times  \Delta \cC$ to $\Ch$.
  \end{defin}
  In order to use this natural transformation to define products, it is useful to consider the diagram
  \begin{equation} \label{eq:diagram_out_of_multiplicative}
     \begin{tikzcd}
       \Delta^{\RS{lr} }\cC^2 \ar[r] \ar[d] &\Delta \left( \cC \times \cC\right) \ar[r] \ar[d] &  \Delta \cC  \ar[r]   & \Ch. \\
       \Delta  \cC \times \Delta \cC \ar[r] & \cC \times \cC  \ar[urr,dashed] & & 
     \end{tikzcd}
  \end{equation}
  We write $ L F^{\RS{lr}} $ for the Kan extension of composite $ \Delta^{\RS{lr} }\cC^2 \to \Delta  \cC \times \Delta \cC \to \cC \times \cC  $, and $L F$ for the Kan extension of the projection $\Delta \left( \cC \times \cC\right) \to  \cC \times \cC $. We note that we have natural transformations
  \begin{equation} \label{eq:multiply_in_and_out}
       L F \otimes L F \Rightarrow   L F^{\RS{lr}} \Rightarrow L F,  
  \end{equation}
  where the first is induced by the multiplicative structure, and the second by Diagram \eqref{eq:diagram_out_of_multiplicative}.

For the next statement, we consider a triple $A_1$, $A_2$, and $A_{12}$ of categories over $\cC$ equipped with a commutative diagram
    \begin{equation} \label{eq:product_categories_over_C}
      \begin{tikzcd}
        A_1 \times A_2 \ar[r] \ar[d]  & \cC \times \cC \ar[d] \\
        A_{12} \ar[r] & \cC,
      \end{tikzcd}
    \end{equation}
where the right vertical map is the monoidal structure on $\cC$.
  \begin{lem}
    A multiplicative structure on $F$ induces a natural map
    \begin{equation}
      \hocolim_{A_1} L F \otimes \hocolim_{A_2} L F \to \hocolim_{A_{12}}L F.
    \end{equation}
  \end{lem}
  \begin{proof}
    The desired map is given by the following composition:
    \begin{align} \label{eq:product_induced_by_mult_structure-1}
      \hocolim_{A_1} L F \otimes \hocolim_{ A_2} L F & \to  \hocolim_{A_1 \times A_2} L F \otimes L F \\
      & \to \hocolim_{A_1 \times A_2} L F  \\ \label{eq:product_induced_by_mult_structure-last}
  & \to \hocolim_{A_{12}} L F .
\end{align}
We use the composition in Equation \eqref{eq:multiply_in_and_out} for the second step.
  \end{proof}

  Next, consider the problem of establishing properties of this product. We begin by considering commutativity, and noting that the transposition on $ \cC \times \cC$ induces a transposition $\tau$ on $ \Delta^{\RS{lr}}\cC^2$ which lifts the transpositions on $\Delta \cC \times \Delta \cC $. Using the assumption that the monoidal structure on $\cC$ is symmetric, we obtain a natural isomorphism
  \begin{equation}
F^{\RS{lr} } \Rightarrow \tau^* F^{ \RS{lr}}   
  \end{equation}
of functors on $\Delta \cC \times \Delta \cC $.

  \begin{defin}
    A multiplicative structure on a functor $F \co \Delta \cC \to \Ch$ is commutative if the following diagram of natural transformations commutes:
    \begin{equation}
      \begin{tikzcd}
   F \otimes F \ar[r,Rightarrow] \ar[d,Rightarrow]  &     \ar[d,Rightarrow] F^{\RS{lr} }    \\
   \tau^{*} \left( F \otimes F \right) \ar[r,Rightarrow] &     \tau^* F^{\RS{lr} }  .
      \end{tikzcd}
    \end{equation}
  \end{defin}

  We now consider a pair of categories $A_1$ and  $A_2$, 
  over $\cC$.
    The following result is readily proved by checking that each step in Equations \eqref{eq:product_induced_by_mult_structure-1}--\eqref{eq:product_induced_by_mult_structure-last} is intertwined by the transposition maps:
    \begin{lem} \label{lem:commutative_map_hocolim}
      If the multiplicative structure on $F$ is commutative, the following diagram, in which the right vertical map is induced by the isomorphism $A_{1} \times A_2 \cong A_{2} \times A_{1}$ 
      and    the natural transformation of the symmetry for the monoidal structure of $\cC$,
            commutes
        \begin{equation}
      \begin{tikzcd}
     \displaystyle{ \hocolim_{A_1} L F  \otimes \hocolim_{ A_2} L F } \ar[r]  \ar[d]  & \displaystyle{\hocolim_{A_{1} \times A_{2}} L F} \ar[d]   \\
      \displaystyle{ \hocolim_{A_2}L F \otimes  \hocolim_{ A_1} L F}   \ar[r] &     \displaystyle{   \hocolim_{A_{2} \times A_1} L F} .
      \end{tikzcd}
    \end{equation} \qed
  \end{lem}

  Finally, we consider associativity. We begin by considering the categories $\Delta^{\RS{lir}} \cC^{3}$, $\Delta^{\RS{l{Llr}}} \cC^{3} $, and $ \Delta^{\RS{{Llr}r}} \cC^{3} $ with the following objects
  \begin{align}
  \Ob  \Delta^{\RS{lir}} \cC^{3}     & \equiv \{ \bfk \to \bfm_1 \times \bfm_2 \times \bfm_3 \to \cC \times \cC \times \cC \} \\
  \Ob  \Delta^{\RS{l{Rlr}}} \cC^{3}       & \equiv \{ \bfk \to \bfn_1 \times \bfn_2 \to \bfn_{1} \times \bfm_1 \times \bfm_2 \to \cC \times \cC \times \cC\} \\
  \Ob   \Delta^{\RS{{Llr}r}} \cC^{3}    & \equiv \{ \bfk \to \bfn_1 \times \bfn_2 \to  \bfm_1 \times \bfm_2 \times \bfn_2 \to \cC \times \cC \times \cC\},
  \end{align}
  and morphisms given by commutative diagrams as before. We have a commutative diagram of functors 
  \begin{equation} \label{eq:diagram_comparison_two_products}
    \begin{tikzcd}
   & &  \Delta \cC & &  \\
      \Delta^{\RS{rl}} \cC^2 \ar[d]    & \Delta^{\RS{{Llr}r}} \cC^{3}  \ar[d] \ar[l] \ar[r] \ar[ur]  &  \Delta^{\RS{lir}} \cC^{3} \ar[d]  \ar[u, bend right]    \ar[u, bend left]  &  \Delta^{\RS{l{Rlr}}} \cC^{3} \ar[d] \ar[l] \ar[ul]  \ar[r]  & \Delta^{\RS{rl}} \cC^2 \ar[d] \\
      \left(\Delta \cC\right)^2   &   \ar[l]  \Delta^{\RS{rl}} \cC^2  \times \Delta \cC \ar[r]   &  \left( \Delta \cC \right)^3   &  \ar[l]  \Delta \cC  \times \Delta^{\RS{rl}} \cC^2 \ar[r] & \left( \Delta \cC \right)^2 
         \end{tikzcd}
  \end{equation}
  where the two parallel vertical arrows in the top row are naturally isomorphic, and we abusively write $ \left( \Delta \cC \right)^3   $ for the category obtained from either of the two parenthesisations of product $\Delta \cC \times \Delta \cC \times \Delta \cC$. Beside its commutativity, the main property of this diagram which we shall use is:
  \begin{lem}
    The leftmost and rightmost squares in Diagram \eqref{eq:diagram_comparison_two_products} are pullback squares. \qed
  \end{lem}
  \begin{cor} \label{cor:pullback_of_Kan_is_Kan}
    Given a functor $F \co \Delta^{\RS{rl}} \cC^2 \to \Ch  $, the pullback to $\Delta^{\RS{rl}} \cC^2  \times \Delta \cC $ of the left Kan extension of $F$ to $\left( \Delta \cC \right)^2  $ is naturally isomorphic to the left Kan extension of the composition of $F$ with the map  $\Delta^{\RS{{Llr}r}} \cC^{3} \to \Delta^{\RS{rl}} \cC^2 $. Similarly, the pullback to $ \Delta \cC  \times \Delta^{\RS{rl}}  \cC^2$ of this left Kan extension is naturally isomoprhic to the left Kan extension of the composite functor on $  \Delta^{\RS{l{Rlr}}} \cC^{3}$. \qed
  \end{cor}
  Returning to the choice of functor $F$, we 
  consider the left Kan extensions 
  \begin{equation}
    \begin{tikzcd}
      &  \Delta^{\RS{{Llr}r}} \cC^{3}  \ar[dr] \ar[d] & &  \Delta^{\RS{l{Rlr}}} \cC^{3} \ar[d] \ar[dl] & \\
      \Ch & \ar[l,"F",swap]   \Delta \cC  & \ar[l]        \Delta^{\RS{lir}} \cC^{3} \ar[r] \ar[d] &  \ar[r,"F"]  \Delta \cC  & \Ch \\
      & &   \left(\Delta \cC\right)^3 \ar[urr, dashed] \ar[ull, dashed] & & 
    \end{tikzcd}
  \end{equation}
  which we denote $F^{\RS{l{Rlr}}}$ and $F^{\RS{{Llr}r}} $, and which correspond to the two different ways of associating an ordered triple of objects of $\cC$. These two functors are naturally isomorphic.
  \begin{lem} \label{lem:multiply_twice}
    A multiplicative structure on $F$ induces natural transformations of functors on $\left(\Delta \cC\right)^3 $
    \begin{align}
      F \otimes ( F \otimes F) \Rightarrow &  F^{\RS{l{Rlr}}} \\
    (F \otimes F)  \otimes F \Rightarrow & F^{\RS{{Llr}r}}.
    \end{align}
  \end{lem}
  \begin{proof}
    Since the two cases are similar, we shall only discuss the first case, which amounts to considering the right half of Diagram \eqref{eq:diagram_comparison_two_products}. The map will be constructed in three steps:

{\bf Step 1:}  We start with the natural transformation
    \begin{equation} \label{eq:first_multiplication_in_assoc}
           F \otimes ( F \otimes F) \Rightarrow F \otimes F^{\RS{lr}}
           \end{equation}
           induced by the multiplicative structure.

  {\bf Step 2:} In this step, we shall construct a natural transformation from the right hand side of Equation \eqref{eq:first_multiplication_in_assoc} to the Kan extension of the map $  \Delta^{\RS{l{Rlr}}} \cC^{3} \to  \Ch$ along projection map to $  \left( \Delta \cC \right)^3$. The target of Equation \eqref{eq:first_multiplication_in_assoc} is the left Kan extension of the diagram
  \begin{equation}
    \begin{tikzcd}
    \Delta \cC  \times \Delta^{\RS{rl}} \cC^2 \ar[r] \ar[d]& \left( \Delta \cC \right)^2 \ar[r] & \Ch \times \Ch \ar[r] & \Ch.  \\
    \left( \Delta \cC \right)^3  \ar[urrr,dashed] & &  &
    \end{tikzcd}
  \end{equation}
  By abuse of notation, we shall denote the horizontal composition $F \otimes F$. This fits in a commutative diagram
  \begin{equation}
    \begin{tikzcd}
      \Delta^{\RS{l{Rlr}}} \cC^{3} \ar[d]  \ar[r]  & \Delta^{\RS{rl}} \cC^2 \ar[d]\ar[r] & \Ch \\
      \Delta \cC  \times \Delta^{\RS{rl}} \cC^2 \ar[r] & \left( \Delta \cC \right)^2 \ar[r] \arrow[ur, dashed, ""{name=D}]  
      & \Ch \times \Ch. \ar[u] \arrow[Rightarrow, to=D,shorten=1ex] 
     \end{tikzcd}
  \end{equation}
  The dashed arrow represents $F^{\RS{lr}}$, which we recall is the left Kan extension of the top right horizontal map along the middle vertical arrow. The multiplicative structure is the data of a natural transformation from $F \otimes F$ to $F^{\RS{lr}}$ (thought of as functors on $ \left( \Delta \cC \right)^2$), and hence induces the displayed natural transformation. Corollary \ref{cor:pullback_of_Kan_is_Kan} identifies the pullback of $F^{\RS{lr}}$ to a functor on $\Delta \cC  \times \Delta^{\RS{rl}} \cC^2 $ with the Kan extension of the map $  \Delta^{\RS{l{Rlr}}} \cC^{3} \to  \Ch$ along the right vertical map, so we obtain from the previous diagram a natural transformation
   \begin{equation}
    \begin{tikzcd}
      \Delta^{\RS{l{Rlr}}} \cC^{3} \ar[d]  \ar[r]   & \Ch \\
    \Delta \cC  \times \Delta^{\RS{rl}} \cC^2 \ar[r] \arrow[ur, dashed, ""{name=D}]  & \Ch \times \Ch \ar[u]\arrow[Rightarrow, to=D,shorten=1ex]  ,
     \end{tikzcd}
  \end{equation}
  where the diagonal arrow is again the left Kan extension of the top vertical map.  Composing further with the projection from $\Delta \cC  \times \Delta^{\RS{rl}} \cC^2  $ to $ \left( \Delta \cC \right)^3$, and using the functoriality of Kan extensions, we obtain a natural transformation from $  F \otimes F^{\RS{lr}} $ to the Kan extension of the map $  \Delta^{\RS{l{Rlr}}} \cC^{3} \to  \Ch$ along projection map to $  \left( \Delta \cC \right)^3$.

{\bf Step 3:}  The commutativity of the diagram
  \begin{equation}
       \begin{tikzcd}
         \Delta^{\RS{l{Rlr}}} \cC^{3} \ar[r]  \ar[dr] & \Delta^{\RS{lir}} \cC^{3} \ar[r] \ar[d]  & \Ch \\
         & \left( \Delta \cC \right)^3 \arrow[ur, dashed]  &
     \end{tikzcd}  
   \end{equation}
   gives a natural transformation from the Kan extension associated to the diagonal map to the one associated to the vertical map. The desired map is the composition of this natural transformation with Equation \eqref{eq:first_multiplication_in_assoc}.

\end{proof}

  The above result leads to the following definition:
  \begin{defin}
    A multiplicative structure on $F$ is \emph{associative} if the following triangle of natural transformations of functors on $\left( \Delta \cC \right)^3 $ commutes:
    \begin{equation} \label{eq:associativity_of_multiplication}
      \begin{tikzcd}
          F \otimes ( F \otimes F) \ar[Rightarrow, r] \ar[Rightarrow, d,"\cong"] & F^{\RS{l{Rlr}}} \ar[Rightarrow, d,"\cong"] \\
    (F \otimes F)  \otimes F\ar[Rightarrow, r]   &  F^{\RS{{Llr}r}}.
      \end{tikzcd}
\end{equation}
\end{defin}

We now state the consequence of associativity at the level of homotopy colimits, for a triple of  categories $A_{1}$, $A_{2}$, and $A_{3}$, over $\cC$. 
\begin{lem} \label{lem:associative_map_hocolim}
  If the multiplicative structure on $F$ is associative, the following diagram commutes:
  \begin{equation}
    \begin{tikzcd}[column sep = 10]
    \displaystyle{  \hocolim_{A_1} L F  \otimes  \hocolim_{A_1}LF \otimes \hocolim_{A_3} LF} \ar[r] \ar[d] & \displaystyle{\hocolim_{A_{1} \times A_2 } LF  \otimes \hocolim_{A_3} LF}  \ar[d] \\
   \displaystyle{  \hocolim_{A_1} LF  \otimes  \hocolim_{A_{2} \times A_3} LF }  \ar[r] & \displaystyle{ \hocolim_{ A_{1} \times A_2 \times A_3} LF}.
    \end{tikzcd}
  \end{equation}
\end{lem}
\begin{proof}
The argument of Lemma \ref{lem:multiply_twice} shows that the two composition factor through
  \begin{equation}
         \hocolim_{A_1 \times A_2 \times A_3} L F^{\RS{l{Rlr}}} \cong \hocolim_{A_1 \times A_2 \times A_3} L F^{\RS{{Llr}r}}.
       \end{equation}
       Equation \eqref{eq:associativity_of_multiplication} then implies the desired commutativity.
     \end{proof}

     We shall also need to understand the naturality of these constructions with respect to the choice of functor $F$. For our main definition we note that a natural transformation of functor $F \Rightarrow G$ on $\Delta \cC$ induces a natural transformation
     \begin{equation}
           F^{\RS{lr}} \to  G^{\RS{lr}}  
         \end{equation}
         of the induced functors on $\left(\Delta \cC \right)^2$: 
     \begin{defin}
       A \emph{multiplicative natural transformation} $F \Rightarrow G$ of multiplicative functors on $\Delta \cC$ is a natural transformation with the property that the following diagram commutes:
       \begin{equation}
         \begin{tikzcd}
   F \otimes F  \ar[r, Rightarrow] \ar[d, Rightarrow]       & F^{\RS{lr}} \ar[d, Rightarrow] \\
      G \otimes G  \ar[r, Rightarrow]        & G^{\RS{lr}}.
         \end{tikzcd}
       \end{equation}
     \end{defin}
     
     The main reason for stating the above definition is to be able to formulate the following result, which refers to Diagram \eqref{eq:product_categories_over_C}:
     \begin{lem} \label{lem:multiplicative-map-hocolim-functorial}
       If a natural transformation $F \Rightarrow G$ is  multiplicative, the induced map of homotopy colimits fits in a commutative diagram
       \begin{equation}
         \begin{tikzcd}
        \displaystyle{   \hocolim_{A_1} L F \otimes \hocolim_{A_2} L F } \ar[r] \ar[d] & \displaystyle{ \hocolim_{A_{1} \times A_2}L F }\ar[d] \\
       \displaystyle{   \hocolim_{A_1} L G \otimes \hocolim_{A_2} L G} \ar[r] & \displaystyle{ \hocolim_{A_{1} \times A_2}L G} .
         \end{tikzcd}
       \end{equation} \qed
     \end{lem}

\section{Algebraic constructions in symmetric monoidal categories}
\label{sec:algebr-constr-symm}

In this section, we review various standard categorical constructions, in the enriched context. We thus fix 
 a symmetric monoidal category $\cC$; the reader may want to have in mind the category of topological spaces, or the category of cochain complexes. The most important example for us will be the category of Kuranishi presentations, though in that case we shall impose some additional conditions on the formal data that we shall presently consider.

We take the notion of a $\cC$-enriched category as given: this is a category with a set of objects, whose morphisms are objects of $\cC$, with composition data are specified as morphisms in $\cC$.

\subsection{Multimodules}
\label{sec:multimodules}
Let $\vec{\cA} = (\cA_1, \ldots, \cA_k; \cA)$ be a collection of categories enriched in $\cC$. We shall define the notion of a multimodule with inputs $(\cA_1, \ldots, \cA_k)$ and output $\cA$. This requires some additional notation: given an integer $i$ betwen $1$ and $k$, a sequence $\vec{p} = (p_1, \ldots, p_k; p)$ of elements of $\vec{A}$, and an element $p'_i$ of $\cA_i$, we write
\begin{equation}
  (p'_i ,p_i) \circ_{i} \vec{p} = (p_1, \ldots , p_{i-1}, p'_i, p_{i+1}, \ldots, p_k; p).
\end{equation}

\begin{defin}
  A $\cC$-enriched multimodule $ \cM$ over $\vec{\cA}$  consists of the data of an object
\begin{equation}
  \cM(\vec{p}; p) \in \Ob \cC
\end{equation}
for each sequence of objects $\vec{p} = (p_1, \ldots, p_k, p)$, and commuting maps
\begin{align}
  \cA_i(p'_i ,p_i) \otimes \cM( \vec{p}; p )   & \to \cM\left( (p'_i ,p_i) \circ_{i} \vec{p} ; p\right) \\
   \cM( \vec{p} ; p)  \otimes \cA(p, p') & \to \cM\left(  \vec{p} ; p'  \right)
 \end{align}
 that are associative.
\end{defin}

\subsection{Multicategories in symmetric monoidal categories}
\label{sec:mult-symm-mono}

\begin{defin} \label{def:multicategory_enriched_in_symmetric_monoidal}
  A \emph{multicategory} $\cX$ in $\cC$ consists of the following data:
  \begin{enumerate}
  \item (Object) a set of object $\Ob(\cX)$
  \item (Multimorphisms) For each sequence $( p_1, \ldots, p_k; p)$ of objects of $\cX$, an object $\cX(p_1, \ldots, p_k; p)$ of $\cC$. 
  \item (Multicompositions) For each pair of sequences $(p^{-}_1, \ldots, p^{-}_k;p^-)$ and $(p^{+}_1, \ldots, p^{+}_\ell; p^+)$, and for each $i$ such that $p^{-}_i = p^{+} $, a map 
    \begin{equation}
      \cX(p^-_1, \ldots, p^-_k; p^-) \times \cX(p^{+}_1, \ldots, p^{+}_\ell, p^{+}) \stackrel{\circ_i}{\to} \cX(p_1, \ldots, p_{k+\ell-1}; p)
    \end{equation}
    where the output sequence is
    \begin{equation}
      (p_1, \ldots, p_{k+\ell-1}; p) = (p^-_1, \ldots, p^-_{i-1}, p^{+}_1, \ldots, p^{+}_\ell, p^-_{i+1}, \ldots, p_{k}; p^-).
    \end{equation}
  \item (Symmetry) For each permutation $\sigma$ of the sequence $(1, \ldots, k)$, an isomorphism
    \begin{equation}
      \cX(p_1, \ldots, p_k; p) \cong \cX(p_{\sigma(1)}, \ldots, p_{\sigma(k)}; p).
    \end{equation}
  \end{enumerate}
  These data are assumed to satisfy the following axioms:
  \begin{enumerate}
  \item (Associativity) Representing multicomposition by grafting of labelled trees with a unique internal vertex, we require that the operations associated to different choices of orderings on the decomposition of a labelled tree agree.
  \item (Symmetric group representation) The isomorphism associated to a composition $\sigma_1 \circ \sigma_2$ of permutation agrees with the composition of the isomorphisms associated to $\sigma_1$ and $\sigma_2$.
  \item (Symmetry of multicompositions) The multicomposition commute with the action of permutations of the inputs.
  \end{enumerate}
\end{defin}
The reader may find explicit formulae for the associativity property of multicomposition in \cite{LeinsterBook}. The main thing to keep in mind is that we do not allow the inputs of a multimorphisms to be empty, and we do not require the existence of units morphisms.

\section{Lifting flow multimodules}
\label{sec:lifting-flow-modules}

The purpose of this appendix is to prove a technical result concerning the multicategory $\Flow$ constructed in Sections \ref{sec:flow-categ-chain} and \ref{sec:mult-flow-categ}. Recall that this cubically enriched multicategory is built from diagrams valued in the category $\Kur$ of Kuranishi presentations. We shall introduce an analogous multicategory $\widetilde{\Flow}$ built from diagrams valued in $ \widetilde{\Kur}$, and then prove the following result, which allows us to then apply theories of virtual counts to pass from geometry to algebra:

\begin{prop}
  \label{prop:lift_flow_to_tflow}
  There is a cubically enriched multi-functor $\Flow \to \widetilde{\Flow}$, which at the level of objects is given by the canonical lift of Lemma \ref{lem:lift_flow_cat-to-factorised}, and whose composition with the forgetful map is naturally isomorphic to the identity of $\Flow$.
\end{prop}

Before proceeding with the construction, we stress the fact that, while the construction of Lemma \ref{lem:lift_flow_cat-to-factorised} is canonical, the multi-functor $\Flow \to \widetilde{\Flow}$ is not. The essential problem concerns the action of symmetries on the sets of $n$-cubes. In $\Flow$, the data of the action includes a choice of isomorphism between two Kuranishi diagrams. However, since there are very few isomorphisms in $\widetilde{\Kur}$ we will have to choose a representative of each equivalence class, and replace these isomorphisms by reorderings of factors. 

Proceeding with the construction, we consider a sequence of Kuranishi flow categories $\vec{\bX} = (\bX_1, \ldots, \bX_k; \bX)$. Using Lemma \ref{lem:lift_flow_cat-to-factorised}, we obtain a sequence
\begin{equation}
  \vec{\widetilde{\bX}} \equiv \left( \widetilde{\bX}_1, \ldots, \widetilde{\bX}_k, \widetilde{\bX} \right)  
\end{equation}
of categories enriched over $\cD \widetilde{\Kur}^{\Gamma}$. This lift allows us to introduce a set $ \widetilde{\flow}_{n}(\vec{\bX})$ whose elements consist of the following data:
   \begin{enumerate}
   \item For each permutation $\phi$ of $\{1, \ldots, n\}$, a multimodule $\widetilde{\bM}_{\phi}$ enriched in $\cD \widetilde{\Kur}^{\Gamma}$  over $ \vec{\widetilde{\bX}} $,  whose morphism of $\Cat^{\Gamma}$-enriched multimodules with respect to the sequence  $\cP_{\bX_i}^{\Gamma_+}$ is
     \begin{equation}
\pi_\phi \co              \cQ_{\bM,\phi} \to \cP_{\vec{\bX} }^{\Gamma}  \times \Face  \square^{n}.
     \end{equation}
   \item For each composition $\phi_1 \circ \phi_2$ of permutations, an isomorphism
     \begin{equation}\label{eq:choice_isomorphism_cube-tKur}
        \widetilde{\bM}_{\phi_1  \circ \phi_2} \cong \phi_2^* \widetilde{\bM}_{\phi_1}      
     \end{equation}
     where the right hand side is given at the level of stratifying categories by the pullback in the diagram
     \begin{equation}
       \begin{tikzcd}
     \phi_2^* \cQ_{\widetilde{\bM},\phi_1} \ar[r] \ar[d] &   \cQ_{\widetilde{\bM},\phi_1} \ar[d] \\
       \cP_{\vec{\bX}}^{\Gamma}  \times \Face  \square^{n} \ar[r, "\id \times \phi_2"]  & \cP_{\vec{\bX}}^{\Gamma}  \times \Face  \square^{n}
       \end{tikzcd}
     \end{equation}
     and at the level of diagrams of presentations by the composite functor.
   \item For each triple composition $\phi_3 \circ \phi_2 \circ \phi_1$, we require that the isomorphisms of pullbacks fit in a commutative diagram
     \begin{equation}
       \begin{tikzcd}
         \widetilde{\bM}_{\phi_1  \circ \phi_2 \circ \phi_3} \ar[r] \ar[d] & \phi_3^{*}  \widetilde{\bM}_{\phi_1 \circ \phi_2} \ar[d] \\
     \left(\phi_2 \circ    \phi_3 \right)^* \widetilde{\bM}_{\phi_1}  \ar[r]   & \phi_3^{*} \left( \phi_2^* \widetilde{\bM}_{\phi_1} \right).
       \end{tikzcd}
     \end{equation}
\end{enumerate}
\begin{rem}
  The above definition is entirely analogous to the one given in Section \ref{sec:cubic-set-morph}. The essential difference is that the isomorphism in Equation~\eqref{eq:choice_iso_cube_Kur} is much less constrained than the one in Equation~\eqref{eq:choice_isomorphism_cube-tKur}, because morphisms in $\widetilde{\Kur}$ are essentially prescribed by permutations of the factors associated to each stratum.
\end{rem}

We now state a preliminary result which we will not directly use, but which the reader can treat as a testing ground for proving Proposition \ref{prop:lift_flow_to_tflow}:
\begin{lem}
  \label{lem:life_flow_to_tflow-no-composition}
  Given a sequence $\vec{\bX}$ of flow categories, there is a map of symmetric semi-cubical sets
  \begin{equation}
   \flow(\vec{\bX}) \to \widetilde{\flow}(\vec{\bX})    
  \end{equation}
  whose composite with the forgetful map takes every $n$-cube to an equivalent one, and which is equivariant with respect to the symmetries of the sequence $\vec{\bX} $. \qed
\end{lem}

At this stage, we have multi-compositions $\circ_i$ on $\widetilde{\flow}$ defined as in Definition \ref{def:composition_poset_multimodules}, which equip it with the structure of a multicategory enriched in semi-cubical categories; in particular, we have an operation
\begin{equation} \label{eq:multi-composition-tflow}
 \underset{T}{\bigcirc}  \co \prod_{v \in V(T) }  \widetilde{\flow}(\vec{\bX}_v) \to     \widetilde{\flow}(\vec{\bX})
\end{equation}
associated to every directed tree $T$ with edges labelled by flow categories, with $\vec{\bX}_v $ the sequence of flow categories given by the edges adjacent to each vertex, and $\vec{\bX} $ the sequence associated to the outgoing leaves.

\begin{rem}
  As in Section \ref{sec:mult-flow-categ-1}, the domain of Equation \eqref{eq:multi-composition-tflow} is not canonically defined, because it requires a choice of ordering and of parentheses, but again, there are coherent isomorphisms for such choices, so that specifying them for one of them determines all others.
\end{rem}

We use these operations introduce the multicategory $\widetilde{\Flow}$ enriched in symmetric cubical sets, by analogy with the construction from Definition \ref{def:Flow_multicategory}: given a sequence $\vec{\bX}$, we define $\widetilde{\Flow}_{n}( \vec{\bX})$ to consist of the following data:
\begin{enumerate}
\item For each face $\sigma$ of the $n$-cube, a directed 
  tree $T_{\sigma}$ with each edge $e$ labelled by a flow category $\bX_e$ and each vertex labelled by an object $\widetilde{\bM}_{v}$ of $\widetilde{\flow}\left(\vec{\bX}_{T_v}\right)$. In addition, we have an identification of the labels of the leaves of $T$ with the sequence $\vec{\bX}$, and a disjoint collection of subsets $D_v$ of $\{1, \ldots, n\}$, indexed by the vertices of $T_{\sigma}$, so that the number of elements of $D_v$ agrees with the dimension of the cube underlying $\widetilde{\bM}_v$.
     \item For each pair $\sigma < \tau$ in $\Face\square^n$, a map  $\pi_{\sigma}^{\tau}$ of trees $T_{\tau} \to T_{\sigma}$, which preserves the labels of all edges which are not collapsed, and so that $D_u$ maps to $D_{\pi_{\sigma}^{\tau}(u)}$ for each vertex $u$ of $T_{\tau}$. We are given as well the following data for each vertex $v$ of $T_{\sigma}$: letting $T_{\tau}^{v} $ denote the inverse image of this tree in $T_{\tau}$, and $ \widetilde{[0,1]}$ denote the interval considered as an object of $\widetilde{\Kur}$, a map
       \begin{equation} \label{eq:map_decomposed_strata_multimodule-tFlow}
\widetilde{[0,1]}^{e_{\sigma}^{\tau}(v)} \times \underset{u \to T_{\tau}^v }{\bigcirc} \widetilde{\bM}_{u} \to \widetilde{\bM}_v            
\end{equation}
which is an isomorphism onto the boundary strata of $\widetilde{\bM}_v$ associated to $\tau$.
\item   For each triple $\sigma < \tau < \rho$, the diagrams associated to the maps in Equation~\eqref{eq:map_decomposed_strata_multimodule} commute.
\end{enumerate}

The cubical structure maps of $\widetilde{\Flow}$ are defined in the same way as in Section \ref{sec:mult-flow-categ-1}, as are the multicompositions and the maps associated to reorderings of the input flow categories.

\begin{proof}[Proof of Proposition \ref{prop:lift_flow_to_tflow}]
  The goal is to assign to each sequence $\vec{\bX}$ of Kuranishi flow categories and each element $\bM$ of $\Flow_{n}(\vec{\bX})$, an element
  \begin{equation}
    \widetilde{\bM} \in     \widetilde{\Flow}_{n}(\vec{\bX})
  \end{equation}
  which is compatible with the structure maps of symmetric cubical sets in the sense that
  \begin{equation}
    \alpha^*(\widetilde{\bM}) = \widetilde{\alpha^* \bM}    
  \end{equation}
  for each map $\alpha \co \square^m \to \square^n$, and which is compatible with multi-compositions in the sense that
  \begin{equation}
 \widetilde{  \bM_{+} \circ_{i} \bM_{-}} = \widetilde{\bM}_{+} \circ_{i}     \widetilde{\bM}_{-}
  \end{equation}
whenever $\bM_{\pm} $ are composable.  

The data of trees associated to each stratum of the $n$-cube will be the same for a flow module with factorised strata and its lift, as is the data of the flow categories associated to each edge of the tree. So it remains to lift the flow modules $\bM_v$ associated to each vertex, and the maps of flow modules in Equation \eqref{eq:map_decomposed_strata_multimodule}.  First, we set the $\Cat^\Gamma$-enriched multimodule $\cQ_{\tilde{\bM}_v, \phi}$ to agree with $\cQ_{\bM_v, \phi}$ for each multimodule $\bM$ with factorised strata, each face $\sigma$ of the $n$ cube, each vertex $v \in T_{\sigma}$, and each permutation $\phi$. We use this to define the projection map to  $  \underset{v \in T}{\bigcirc} \cP_{\vec{\bX}_v}^{\Gamma}   \times \Face  \square^{n}  $. We must then functorially assign to each $P \in \cQ_{\bM_v, \phi}$, $\lambda \in \Gamma$, and $\vec{p} \in \vec{\bX}$ an object of $\widetilde{\Kur}$, compatibly with morphisms in $ \cQ_{\bM_v, \phi}$, the multimodule action, the structure maps of cubical sets, the symmetries of $\vec{\bX} $, as well as the multicompositions associated to inclusions $\sigma \subset \tau$ of faces of the cube.

The proof proceeds by induction on the dimension of non-degenerate cubes. In the base case of a $0$-cube,  we start by considering flow modules with factorised strata whose corresponding tree has a unique vertex $v$. For such an indecomposable $\bM$, there is a canonical lift $\widetilde{\bM}_{v,\id}$ of $\bM_{v,\id}$ specified by assigning to each stratum of the Kuranishi diagram $\bM^{\lambda}_{v,\id}(\vec{p})$, labelled by $\lambda \in \Gamma$ and a sequence $\vec{p}$ of objects of $\vec{\bX}$, its decomposition labelled by vertices of the underlying tree (with a unique vertex of valence the number of elements of the sequence $\bX$, and all other vertices of valence $2$). Since there is no choice to be done in this case, this construction is clearly compatible with any symmetries of $\bX$ in the sense that if $f$ is a permutation of $\{1, \ldots, d\}$, with the property that $\bX_{i} = \bX_{f(i)}$ for all $i$, then the lift of $\bM$ agrees with that of $f^*(\bM)$.

Using compositions, this immediately extends to $0$-cubes labelled by arbitrary trees. We then extend to cubes whose underlying non-degenerate cube has dimension $0$ by lifting the data from the $0$-cubes, since degeneracies are defined in the same way on both sides. There are no choices in this first step, so that it is compatible with the isomorphisms arising from reordering the inputs.

We now assume that a lift has been specified for all flow multimodules whose underlying non-degenerate cube has dimension $i<m$, compatibly with compositions and with the cubical structure maps, and subject to the following constraints:
\begin{enumerate}
  \item If $f$ is a permutation of the elements $(1, \ldots, d)$, then the lift of an element of $ \Flow_{i}(\vec{\bX}) $ agrees with the lift of its image in $  \Flow_{i}(f^* \vec{\bX}) $.
  \item Each orbit of the action of the symmetric group on non-degenerate $i$-cubes has a representative $\bN$ so that, for each stratum $\sigma$, and each vertex $v$ of $T_\sigma$, the lift $\widetilde{\phi_* \bN}_v$ is given by composing $\widetilde{\bN}_v$ with the isomorphism of stratifying categories associated to $\phi$.
  \item The sequence of Kuranishi presentations associated by $ \widetilde{\bN}_{v,\id}$ to each object of $\cQ^{\lambda}_{\bN_v, \id}(\vec{p})$ lying over the top stratum of the $i$-cube is given by its canonical decomposition in terms of the labelled tree associated to its image in $\underset{v \in T}{\bigcirc} \cP_{\vec{\bX}_v}^{\Gamma} $.
\end{enumerate}
 Combining these three inductive assumptions, and using $\underline{    \widetilde{\bM}} $ to denote the image of an object $\widetilde{\bM}$ of $\widetilde{\Kur}$ under the forgetful map to $\Kur$, we see that we have a canonical map
  \begin{equation} \label{eq:isomorphism_lift_and_forgets}
     \underline{    \widetilde{\bM}}^{P}_{v,\phi}(\vec{p}) \to \bM^P_{v,\phi}(\vec{p})
  \end{equation}
  obtained whenever $P$ lies over the top stratum of a cube by composing a permutation with the canonical isomorphism associated to the factorisation of the representative of the orbit of this cube.
  
  We impose an additional induction hypothesis:
  \begin{enumerate}
    \setcounter{enumi}{3}
  \item The isomorphisms in Equation \eqref{eq:isomorphism_lift_and_forgets} define an equivalence of flow modules with factorised strata $\underline{    \widetilde{\bM}} \cong \bM$, which is compatible with face maps.
  \end{enumerate}

  We now proceed with the inductive step, and consider all non-degenerate cubes of dimension $m$ in the sets $ \{\Flow_m(f^* \vec{\bX})\}$ for any reordering $f$ of the inputs.  We choose a representative $\bN$ of each orbit of the action of the symmetric group on $m$-cubes, and of the reorderings of $\vec{\bX}$. 

  If $\sigma$ is a proper stratum of the $m$-cube, then a lift $\widetilde{\bN}_{v,\phi}$ is specified in the inductive step for each $v \in T_{\sigma}$. If $\sigma =  [\square^m]$ is the top stratum, then whenever $v \in T_{[\square^m]}$ is a vertex of the associated tree, the induction hypothesis further specifies the lift of the restriction of $\widetilde{\bN}_{v,\phi}$ to each element of $\cQ_{\bN_v,\phi} $ which projects to a boundary stratum of $\square^m$. For elements which project to the top stratum, we proceed starting with the case $\phi = \id$ as follows: consider a sequence $\vec{p}$ of objects of $\vec{\bX}$, and an element  $P \in  \cQ^{\lambda}_{\bN_v,\id}(\vec{p}) $ which lies over the top face of $\square^m$. Our goal is to specify a lift of $\bN^{P}_{v,\id}(\vec{p})$ to an object of $\widetilde{\Kur}^{\Gamma}$: this itself consists of a choice, for each object $P'$ of $\cQ^{\lambda}_{\bN_v,\id}(\vec{p}) $ lying under $P$, of a sequence of Kuranishi presentations whose product is equipped with an isomorphism to $\partial^{P'} \bN^{P}_{v,\id}(\vec{p})$. If $P'$ lies over a boundary face of the cube, then this sequence is determined by the previous step using the analogue of Diagram \eqref{eq:iso_strata_Kuranishi_diagram} for objects of $\cD \widetilde{\Kur}$. Otherwise, we use the labelled tree corresponding to the image of $P'$ in  $ \underset{v \in T}{\bigcirc} \cP_{\vec{\bX}_v}^{\Gamma}(\vec{p})$  to write down a sequence of Kuranishi presentations whose product is equipped with a canonical isomorphism to $\partial^{P'} \bN^{P}_{v,\id}(\vec{p})$. The inductive assumption and the structure maps of $\bN_{v}$ then give rise to the maps in Equation \eqref{eq:compatibility_factorisations_tkur}: the key point is that each arrow $P' \to P''$ under $P$ gives rise to a canonical map
  \begin{equation}
   \partial^{P''} \bN^{P}_{v,\id}(\vec{p}) \to  \partial^{P'} \bN^{P}_{v,\id}(\vec{p}),
  \end{equation}
  and since the tree corresponding to $P''$ refines the tree corresponding to $P'$, this uniquely lifts to a map between the factorisations.  Next, we specify the lift of $\bN_{v,\phi}$ by composing $ \bN_{v,\id}$ with the isomorphism of $\Cat^{\Gamma}$-enriched multimodules associated to $\phi$.

  This completes the construction of the lift of the chosen representative $\bN$ of the orbit of the permutation group on the inputs, and we specify the lift of any $m$-cube by composing with permutations of the inputs, and pulling back the lift under a permutation of the inputs of $\vec{\bX}$ mapping an $m$-cube to its chosen representative.

The first three induction hypotheses are immediate from the construction. The fourth is a tedious check which we omit.
  
\end{proof}

\bibliographystyle{alpha}
\def\cprime{$'$}

\end{document}